\documentclass{amsart}

\usepackage{amssymb,amsmath,amsthm,amscd,amsfonts}

\newtheorem{theorem}{Theorem}[section]
\newtheorem{proposition}[theorem]{Proposition}
\newtheorem{corollary}[theorem]{Corollary}
\newtheorem{lemma}[theorem]{Lemma}
\newtheorem{claim}[theorem]{Claim}

\theoremstyle{definition}
\newtheorem{definition}[theorem]{Definition}
\newtheorem{remark}[theorem]{Remark}

\newtheorem*{OTP*}{Organization of this paper}
\newtheorem*{Ack*}{Acknowledgements}

\begin{document}


\title[]
{An $\varepsilon$-regularity theorem for line bundle mean curvature flow}


\author{Xiaoli Han}
\address{Department of Mathematical Sciences, Tsinghua University, Beijing 100084, P. R. of China}
\email{hanxiaoli@mail.tsinghua.edu.cn}

\author{Hikaru Yamamoto}
\address{Department of Mathematics, Faculty of Science, Tokyo University of Science, 1-3 Kagurazaka, Shinjuku-ku, Tokyo 162-8601, Japan}
\email{hyamamoto@rs.tus.ac.jp}


\begin{abstract}
In this paper, we study the line bundle mean curvature flow defined by Jacob and Yau \cite{JacobYau}. 
The line bundle mean curvature flow is a kind of parabolic flows to obtain deformed Hermitian Yang-Mills metrics on 
a given K\"ahler manifold. The goal of this paper is to give an $\varepsilon$-regularity theorem for the line bundle mean curvature flow. 
To establish the theorem, we provide a scale invariant monotone quantity. 
As a critical point of this quantity, we define self-shrinker solution of the line bundle mean curvature flow. 
The Liouville type theorem for self-shrinkers is also given. 
It plays an important role in the proof of the $\varepsilon$-regularity theorem. 
\end{abstract} 






\thanks{
The first author was supported by NSFC, No.11471014. 
The second author was supported by JSPS KAKENHI Grant Number 16H07229 and Osaka City University Advanced
Mathematical Institute (MEXT Joint Usage/Research Center on Mathematics
and Theoretical Physics)
}


\maketitle

\section{Introduction}\label{sec:intro}
An $\varepsilon$-regularity theorem ensures the boundedness of derivatives of a solution of some PDE under the assumption that a quantity, usually defined by the integral of the solution, is $\varepsilon$-close to the regular value. 
In this paper, we give an $\varepsilon$-regularity theorem for line bundle mean curvature flows. 
This is motivated by the $\varepsilon$-regularity theorem for mean curvature flows due to White \cite{White}. 
Recently, the line bundle mean curvature flows were defined by Jacob and Yau \cite{JacobYau} 
to acquire deformed Hermitian Yang--mills metrics. 
We will describe the background of these objects later. 
First, we focus on the introduction of the main result. 

\subsection{Basic notions}
Let $(X,g)$ be a K\"ahler manifold with $\dim_{\mathbb{C}}X=n$ and associated K\"ahler form $\omega$. 
We fix a holomorphic line bundle $L\to X$. 
When a Hermitian metric $h$ of $L$ is given, we define a function $\zeta:X\to \mathbb{C}$ by 
$\zeta:=(\omega-F(h))^{n}/\omega^{n}$, where $F(h):=(-1/2)\partial\bar{\partial}\log h$, 
the curvature 2-form of the Chern connection associated with $h$. 
Note that $F(h)$ is pure imaginary valued. 
Then, we define the {\em Hermitian angle} of $h$ by $\theta:=\arg\zeta$ and 
one can see that $\theta$ is lifted as an $\mathbb{R}$-valued function 
rather than $\mathbb{R}/2\pi\mathbb{Z}$-valued in Section \ref{scalinginvsec}. 

Assume that a smooth 1-parameter family of Hermitian metrics $h_{t}$ of $L$ is given for $t\in[0,T)$. 
Define $u(\,\cdot\,,t):X\to\mathbb{R}$ by $h_{t}=e^{-u(t)}h_{0}$. 
Then, it holds that $u(\,\cdot\,,0)\equiv 0$. 

\begin{definition}[\cite{JacobYau}]\label{dfoflbmcf}
$h=\{\,h_{t}\,\}_{t\in[0,T)}$ is called a line bundle mean curvature flow of $L\to X$ with respect to $\omega$ 
if there exists a constant $\hat{\theta}\in\mathbb{R}$ such that 
\begin{equation}\label{defeqoflbmcf}
    \frac{d}{dt}u=\theta-\hat{\theta}, 
\end{equation}
where $\theta$ is the Hermitian angle of $h_{t}$ at each time $t$. 
We call $h_{0}$ the initial metric. 
\end{definition}

The constant $\hat{\theta}$ in \eqref{defeqoflbmcf} 
should be chosen appropriately 
to see \eqref{defeqoflbmcf} as a potential way to get 
a deformed Hermitian metric on $L$ as a limit of the flow. 
Actually, in the paper of Jacob and Yau \cite{JacobYau}, 
the constant $\hat{\theta}$ is specified to satisfy 
$\mathop{\mathrm{Im}}(e^{-\sqrt{-1}\hat{\theta}} Z_{L})=0$, 
where $Z_{L}\in\mathbb{C}$ is defined in Section \ref{scalinginvsec}. 
However, we use \eqref{defeqoflbmcf} just as a PDE in this paper. 
Hence, any constant $\hat{\theta}\in\mathbb{R}$ is allowed. 

\subsection{Key assumptions}
To prove the main theorem (the $\varepsilon$-regularity theorem) we need to assume two things: 
one is for the ambient $(X,g)$ and the other is for the flow $\{h_{t}\}_{t\in[0,T)}$. 
These assumptions seem unnatural and strong at first glance. 
To explain why such condition is supposed, we should back to the work of Leung, Yau and Zaslow \cite{LeungYauZaslow} 
and we postpone it until Section \ref{Background}. 
Thus, in this subsection, we restrict ourselves to the introduction of those assumptions. 

\begin{definition}\label{graphical}
Fix an open set  $U\subset X$. We say that $(X,g)$ is {\it semi-flat} on $U$ 
if the following properties are satisfied: 
\begin{itemize}
\item[(i)] There exists a diffeomorphism $\varphi:B(r)\times B(r')\to U$, 
where $B(r)$ is an open ball in $\mathbb{R}^n$ centered at 
the origin with radius $r$. 
We will use real coordinates $(x^{1},\dots,x^{n})$ on $B(r)$ and $(y^{1},\dots,y^{n})$ on $B(r')$.  
\item[(ii)] Complex coordinates on $B(r)\times B(r')$ defined by $z^{i}:=x^{i}+\sqrt{-1}y^{i}$ match the original holomorphic structure on $U$. 
This implies that $\varphi$ is biholomorphic. 
\item[(iii)] Under these coordinates $(U,(z^{1},\dots,z^{n}))$, the coefficients of the K\"ahler form 
$\omega=(\sqrt{-1}/2)g_{\bar{k}j}dz^{j}\wedge d\bar{z}^{k}$ satisfy, for all $i,j,k\in \{\,1,\dots,n\,\}$, 
\begin{equation}\label{lsfc}
g_{\bar{i}j}=g_{\bar{j}i}
\quad\mbox{and}\quad
\frac{\partial}{\partial y^{k}}g_{\bar{i}j}=0. 
\end{equation}
\end{itemize}
\end{definition}

\begin{definition}\label{graphical2}
Assume that $(X,g)$ is semi-flat on $U$ and coordinates $(z^{1},\dots,z^{n})$ on $U$ 
is induced by $\varphi:B(r)\times B(r')\to U$. 
We further assume that there exists a nonvanishing holomorphic section $e\in \Gamma(U,L)$. 
Then, we say that a pair of a holomorphic line bundle $L\to X$ 
and a Hermitian metric $h$ of $L$ is {\it graphical} on $U$ with respect to $e\in \Gamma(U,L)$ 
if for all $k\in \{\,1,\dots,n\,\}$ 
\begin{equation}\label{lsfc2}
\frac{\partial}{\partial y^{k}}\log h(\bar{e},e)=0. 
\end{equation}
\end{definition}

\subsection{The main theorem}\label{subsecmain}
Let $U\subset X$ be an open set and $U^{c}$ denotes its complement. 
Put $V:=U\times[a,b)$ for some $a,b\in\mathbb{R}$. 
Then, for a space-time point $Q:=(p,t)\in V$, we define the parabolic distance from $Q$ to 
the boundary of $V$ by 
\begin{equation}\label{distance}
    \mathrm{dist}_{g}(Q,V):=\min\left\{\,\inf_{q\in U^{c}}d_{g}(p,q),\sqrt{b-t},\sqrt{t-a}\,\right\}. 
\end{equation}
Now, we can state our main theorem ($\varepsilon$-regularity theorem) except 
for precise definitions of two important quantities: $\bar{\Theta}$ and $K_{3,\alpha}$.  

\begin{theorem}\label{epregthm}
Fix a K\"ahler manifold $(X,g)$, a bounded open set $U'\subset X$, $\alpha\in(0,1)$ and $A>0$. 
Assume that $(X,g)$ is semi-flat on $U'$ with respect to $\varphi:B(4r)\times B(r')\to U'$. 
Then, there exist $\varepsilon, C>0$ with the following property. 
Suppose $L\to X$ is a holomorphic line bundle, 
$h=\{\,h_{t}\,\}_{t\in [0,T)}$ is a line bundle mean curvature flow of $L$ with $T<\infty$ 
and $e\in\Gamma(U',L)$ is a nonvanishing holomorphic section 
so that $h_{t}$ is graphical on $U'$ for all $t\in[0,T)$ 
with respect to $e\in\Gamma(U',L)$. 
Put $U:=\varphi(B(r)\times B(r'))$ and $V:=U\times [0,T)$. 
Assume that $\sup_{V}|F(h(t))|\leq A$ and 
\[\bar{\Theta}(h,Q,t)\leq 1+ \varepsilon\]
for all $ Q=(p,T')\in U\times (0,T)$ and $t\in(T'-(\mathrm{dist}_{g}(Q,V))^2,T')\cap(0,T)$. 
Then, 
\[K_{3,\alpha;V}(g,\phi)\leq C, \]
where $\phi:=-\log h(\bar{e},e)$. 
\end{theorem}

The precise definitions of $\bar{\Theta}(h,Q,t)$ and $K_{3,\alpha;V}(g,\phi)$ are complicated. 
So, we refrain from describing these in this subsection. 
Here, we just put some remarks on these quantities. 
First, $\bar{\Theta}(h,Q,t)$ is called the {\it Gaussian density} of $h=\{\,h_{t}\,\}_{t\in [0,T)}$ 
at $Q=(p,T')$ with scale $t$ and 
defined in Definition \ref{gaussdensity}. 
This is an analogue of the Gaussian density for mean curvature flows 
introduced by Stone \cite{Stone}. 
Next, $K_{3,\alpha;V}(g,\phi)$ are basically defined by $|\partial_{t}\phi|_{C^{0}}$, 
$|\partial_{t}\nabla\phi|_{C^{0,\alpha}}$ and $|\nabla^3\phi|_{C^{0,\alpha}}$. 
Roughly speaking, we first define $K_{3,\alpha}((g,\phi),Q)$ by these three seminorms, following White \cite{White}, and next define $K_{3,\alpha;V}(g,\phi)$ by the supremum of 
the product of $K_{3,\alpha}((g,\phi),Q)$ and $\mathrm{dist}_{g}(Q,V)$ for $Q\in V$. 
Those are explained in Section \ref{proofofmainthm}. 

\subsection{The strategy of the proof}\label{subsec1.4}
Without precise definitions and proofs of facts, 
we explain how Theorem \ref{epregthm} will be proved. 
This instant proof sheds light on three keys we will give in the following sections. 
Let us denote by $\mathcal{A}$ the set of all triplets $a=((X,\omega),L,h)$, 
where $(X,\omega)$ is a K\"ahler manifold, $L$ is a holomorphic line bundle over $X$ and 
$h=\{\,h_{t}\,\}_{t\in [0,T_{\mathrm{max}})}$ is a line bundle mean curvature flow of $L$. 
In this subsection, we write $\bar{\Theta}(a,Q,t)$ and $K_{3,\alpha;V}(a)$ 
instead of $\bar{\Theta}(h,Q,t)$ and $K_{3,\alpha;V}(g,\phi)$, respectively.  

\begin{itemize}
\item[(i)] The first key is the scaling invariance of line bundle mean curvature flows. 
We define a parabolic scaling operator $D_{k}^{T}:\mathcal{A}\to \mathcal{A}$ for 
$T\in\mathbb{R}$ and $k\in\mathbb{N}$ in Section \ref{scalinginvsec}. 
Roughly, it is given by $D_{k}^{T}(a):=((X,k \omega),L^{\otimes k},h^{\otimes k})$ and 
we have to change the scale of time $t$ precisely.  

\item[(ii)] The second key is the Gaussian density $\bar{\Theta}\geq 0$ and its properties: 
scaling invariance and monotonicity. 
The former means $\bar{\Theta}(D_{k}^{T}(a),Q,t)=\bar{\Theta}(a,Q',t')$, where $Q':=(p,0)$ and $t':=T+t/k$. 
The latter means $\partial_{t}\bar{\Theta}(a,Q,t)\leq -B(h)+C$ for $a=((X,\omega),L,h)\in\mathcal{A}$, 
where $B(h)\geq 0$ is defined by $h$ and $C\geq  0$ is a constant. 
If $(X,\omega)$ is $\mathbb{R}^{n}\times B(r')$ with the standard metric, then $C=0$. 
This implies that $\bar{\Theta}(a,Q,t)+C(T'-t)\geq 0$ is monotonically decreasing for $t$ and 
has the limit as $t\to T'$. 
It is also important that the limit of $\lim_{t\to T'}\bar{\Theta}(a,Q,t)\geq 1$ when $T'$ 
of the chosen $Q=(p,T')$ is 
strictly less than $T_{\mathrm{max}}$. 
These are discussed in Section \ref{sec:monoform}. 

\item[(iii)] The third key is the Liouville type theorem for self-shrinkers. 
Roughly speaking, an ancient solution $h=\{\,h_{t}\,\}_{t\in (-\infty,T_{\mathrm{max}})}$ of the line bundle mean curvature flow satisfying $B(h)=0$ is called a {\em self-shrinker}. 
Then, we can prove that if $T_{\mathrm{max}}=\infty$ for a graphical self-shrinker 
then $\phi:=-\log h_{t}$ should be of the form 
$a_{ij}x^{i}x^{j}+b$ for some constants $a_{ij},b\in\mathbb{R}$. 
Then, one may agree that when $\phi=a_{ij}x^{i}x^{j}+b$ then $K_{3,\alpha}(a,Q)=0$ 
since we mentioned that it is defined by $|\partial_{t}\phi|_{C^{0}}$, 
$|\partial_{t}\nabla\phi|_{C^{0,\alpha}}$ and $|\nabla^3\phi|_{C^{0,\alpha}}$ 
though we have not given its precise definition. 
\end{itemize}

Then, the proof of Theorem \ref{epregthm} will be done with these keys as follows. 

\begin{proof}[Sketch of the proof of Theorem \ref{epregthm}]
We do proof by contradiction. So, assume that there exist sequences $C_{i}\to \infty$, 
$\varepsilon_{i}\to 0$ and 
line bundle mean curvature flows $h_{i}$ of $L_{i}$ over $(X,\omega)$ (we put 
$a_{i}:=((X,\omega),L_{i},h_{i})$) such that 
\begin{equation}\label{shortcont}
\bar{\Theta}(a_{i},Q,t)\leq 1+\varepsilon_{i}\quad\mbox{and}\quad K_{3,\alpha;V_{i}}(a_{i})\geq C_{i},
\end{equation}
where we omitted the ranges of $Q$ and $t$. 
We also assume that each $a_{i}$ satisfies all additional assumptions in Theorem \ref{epregthm}. 
Then, one can prove that $K_{3,\alpha}(a_{i},\,\cdot\,)\to\infty$ uniformly. 
Then, by choosing $k_{i}$ precisely, we can normalize these so that 
\begin{equation}\label{norm1}
K_{3,\alpha}(D_{k_{i}}^{T_{i}}(a_{i}),Q_{i})=1
\end{equation}
at some point $Q_{i}$ since $K_{3,\alpha}$ performs in inverse proportion for the scaling. 

On the other hand, since the density is scaling invariant, 
we have 
\[\bar{\Theta}(D_{k_{i}}^{T_{i}}(a_{i}),Q,t)=\bar{\Theta}(a_{i},Q',t')\] 
and the right hand side tends to $1$ by \eqref{shortcont}. 
Moreover, we can prove that $D_{k_{i}}^{T_{i}}(a_{i})$ converges to $a_{\infty}\in\mathcal{A}$ in some sense, 
where $a_{\infty}=((X_{\infty},\omega_{\mathrm{st}}),\underline{\mathbb{C}},h_{\infty})$ 
with $X_{\infty}:=\mathbb{R}^{n}\times B(r')$ 
and $h_{\infty}=\{\,h_{\infty,t}\,\}_{t\in\mathbb{R}}$. 
Then, by the second key with $C=0$, we see that $\bar{\Theta}(a_{\infty},Q_{\infty},t)\geq 1$. 
Letting $i\to\infty$ in \eqref{shortcont}, we know that $\bar{\Theta}(a_{\infty},Q_{\infty},t)\leq 1$. 
Thus, we see that $\bar{\Theta}(a_{\infty},Q_{\infty},t) \equiv 1$, 
so $\partial_{t}\bar{\Theta}(a_{\infty},Q_{\infty},t)\equiv 0$. 
This together with the second key and $C=0$ implies that $B(h_{\infty})=0$, 
that is, $h_{\infty}$ is a self-shrinker. 

Now, $h_{\infty}$ is a self-shrinker defined for all time. 
Thus, by the third key (the Liouville type theorem for self-shrinkers) we can say that 
\[K_{3,\alpha}(a_{\infty},Q_{\infty})=0. \]
But, this contradicts to the normalization \eqref{norm1} with $D_{k_{i}}^{T_{i}}(a_{i})\to a_{\infty}$. 
\end{proof}

\subsection{Organization of this paper} 
Section \ref{sec:intro} is the shortest path to the main theorem of this paper 
and gives the sketch of the proof of the main theorem. 
Section \ref{Background} gives the background of the present work which is related to mirror symmetry. 
Section \ref{scalinginvsec} gives the basic notations and the scaling invariance 
of the line bundle mean curvature flow PDE. 
Section \ref{sec:divergence} is devoted to build the divergence theorem 
for a Hermitian metric as an analog of it for a submanifold. 
In Section \ref{sec:monoform}, we provide the monotonicity formula for line bundle mean curvature flows, 
define the Gaussian density and prove important properties of it. 
In Section \ref{sec:self-shrinker}, we define a self-shrinker for the line bundle mean curvature flow PDE 
and prove the Liouville type theorem for it. 
In Section \ref{proofofmainthm}, we give the proof of the main theorem 
after the definition of $K_{3,\alpha}$-quantity. 
\subsection*{Acknowledgments}
The first author would like to thank Professor S.-T. Yau for inviting her to visit Harvard University where the research studied.
The second author would like to thank Professor A. Futaki for introducing him some previous results 
relating to this paper and for private communication. 

\section{Background}\label{Background}
In this section, we provide the background of the present work. 
We review the importance of deformed Hermitian Yang-Mills metrics and 
line bundle mean curvature flows along the history of mirror symmetry. 
By going back to the origin of deformed Hermitian Yang-Mills metrics, 
one can see that the semi-flat condition (Definition \ref{graphical}) and graphical condition (Definition \ref{graphical2}) are naturally satisfied 
in important cases. 

\subsection{Short history of mirror symmetry}
There is no room for doubt that mirror symmetry is not only important for physicists but also mathematicians. 
From the proposal by Kontsevich \cite{Kontsevich}, the so-called homological mirror symmetry, it is widely recognized as an equivalence of a triangulated category between the bounded derived category of coherent sheaves on $X$, denoted by $\mathop{D^{b}Coh}(X)$, 
and the one of Fukaya category, denoted by $\mathop{D^{b}Fuk}(Y)$ for mirror Calabi-Yau manifolds $X$ and $Y$. 
Roughly speaking $\mathop{D^{b}Coh}(X)$ is determined by the complex structure of $X$ and $\mathop{D^{b}Fuk}(Y)$ is by the symplectic structure of $Y$. In superstring theories, this is regarded as T-duality between type IIA string theory (related to complex geometry) and type IIB (related to symplectic geometry).

Although the homological mirror symmetry tells us what should happen when a mirror Calabi-Yau pair is given, 
it does not provide a way to construct such a mirror pair. 
Amid such circumstances, Strominger, Yau and Zaslow \cite{StromingerYauZaslow} proposed a way to create mirror Calabi-Yau partners, 
now it is called the SYZ conjecture. 
Simply speaking, they proposed that a mirror partner should be obtained by the real Fourier-Mukai transform 
when one side is the total space of a special Lagrangian torus fibration over some base manifold $B$. 
Since the SYZ conjecture, special Lagrangian submanifolds have acquired much attention. 
We remark that special Lagrangian submanifolds had been originally defined by Harvey and Lawson \cite{HarveyLawson} before the SYZ conjecture. 

The real Fourier-Mukai transform is not only a tool to construct a mirror partner 
but also a map which sends D-branes in one side to the other side. 
This is explained by Mari\~no, Minasian, Moore and Strominger \cite{MarinoMinasianMooreStrominger} from the physical side and 
by Leung, Yau and Zaslow \cite{LeungYauZaslow} from the mathematical side. 
Their consequence is that the corresponding objects to special Lagrangian submanifolds in the type IIB side are 
deformed Hermitian Yang-Mills connections in the type IIA side. 

To be precise, let $\theta\in\mathbb{R}$ be a constant, $(X,g)$ a K\"ahler manifold with $\dim_{\mathbb{C}}X=n$ and associated K\"ahler form $\omega$ and $L\to X$ a complex line bundle with a Hermitian metric $h$. 

\begin{definition}\label{dHYMconnection}
A deformed Hermitian Yang-Mills connection with phase $e^{\sqrt{-1}\theta}$ is a Hermitian connection $\nabla$ of $(L,h)$ so that 
its curvature 2-from $F$ satisfies 
\[F^{0,2}=0\quad\mbox{and}\quad \mathrm{Im}\left(e^{-\sqrt{-1}\theta}\bigl(\omega+F\bigr)^{n}\right)=0. \]
\end{definition}

It is well-known that the first condition, $F^{0,2}=0$, is equivalent to that the existence of a holomorphic structure so that 
the Chern connection associated to $h$ is $\nabla$, that is, the integrability condition. 
The second condition is nonlinear in general, however it is rewritten as $\omega\wedge F=0$ when $\dim_{\mathbb{C}}X=2$ and $\theta=0$, 
and this is just the Hermitian Yang-Mills equation. 
After a blank period of about fifteen years from \cite{LeungYauZaslow}, 
the study of dHYM has been developed recently, see \cite{CollinsJacobYau, CollinsYau, CollinsXieYau, Pingali} and references therein. 

\subsection{Introduction to the work of Leung-Yau-Zaslow}
In our main theorem (Theorem \ref{epregthm}), we assume locally semi-flat and graphical condition for $X$ and $h$. 
It seems unnatural at first glance. 
To explain why such conditions are supposed, we go back to the origin of 
deformed Hermitian Yang-Mills connections, that is, the work of Leung, Yau and Zaslow \cite{LeungYauZaslow}. 

Let $B$ be an open set in $\mathbb{R}^{n}$ with standard coordinates $x^{i}$ 
and $\phi$ be a strictly convex smooth function on $B$. 
Then, other coordinates on $B$ are introduced by 
$\tilde{x}_{i}:=\partial\phi/\partial x^{i}$ as the Legendre transform of $\phi$. 
Put $M:=B\times T^{n}$ and $W:=B\times(T^{n})^{*}$, 
where $T^{n}$ ($\cong\mathbb{R}^{n}/\mathbb{Z}^{n}$) is an $n$-torus 
with coordinates $y^{i}$ 
and $(T^{n})^{*}$ ($\cong(\mathbb{R}^{n})^{*}/(\mathbb{Z}^{n})^{*}$) is its dual 
with coordinates $\tilde{y}_{i}$.  
A complex structure and K\"ahler form on $M$ are defined by 
\[z^{i}:=x^{i}+\sqrt{-1}y^{i}\quad\mbox{and}\quad
\omega:=\frac{\sqrt{-1}}{2}\phi_{ij}(x)
dz^{i}\wedge d\bar{z}^{j}\]
with $\phi_{ij}(x)=\partial^2\phi(x)/\partial x^{i}\partial x^{j}$; 
those on $W$ are defined by 
\begin{equation}\label{WinLYZ}
\tilde{z}_{i}:=\tilde{x}_{i}+\sqrt{-1}\tilde{y}_{i}\quad\mbox{and}\quad
\tilde{\omega}:=\frac{\sqrt{-1}}{2}\phi^{ij}(x)
d\tilde{z}^{i}\wedge d\bar{\tilde{z}}^{j}
\end{equation}
with $(\phi^{ij})=(\phi_{ij})^{-1}$. 
We equip $M$ with a holomorphic volume form $\Omega:=dz^{1}\wedge\dots\wedge dz^{n}$. 

Fix a section $Y=(Y^{1},\dots,Y^{n})$ of $M$, regarding $M$ as a torus fibration over $B$, 
and put its graph by $S_{Y}:=\{\,(x,Y(x))\mid x\in B\,\}$. 
On the other hand, $Y$ assigns each point $x\in B$ to 
a connection $\nabla^{Y(x)}$ on the torus fiber $T^{n}(x)$ over $x$. 
This is defined by the canonical identification 
$T^{n}(x)\cong \mathop{\mathrm{Hom}}(\pi_{1}((T^{n}(x))^{*}),U(1))$, 
where we used the fact that the right hand side is just the moduli space of 
flat connections on $(T^{n}(x))^{*}$. 
The family of connections $\nabla^{Y(x)}$ along $x\in B$ constitutes 
a connection of the trivial $\mathbb{C}$ bundle $L:=\underline{\mathbb{C}}$ 
(with the standard metric $h:=\langle\, , \,\rangle$) on the whole $W$, 
written explicitly by 
\[\nabla^{Y}=d+\sqrt{-1}Y^{j}d\tilde{y}_{j}. \]
Then, the result of Leung, Yau and Zaslow is stated as follows. 

\begin{theorem}[Leung, Yau and Zaslow, \cite{LeungYauZaslow}]
$S_{Y}$ in $M$ is a special Lagrangian submanifold with phase $e^{\sqrt{-1}\theta}$ 
if and only if $\nabla^{Y}$ of $(L,h)$ on $W$ is a deformed Hermitian Yang-Mills connection with phase $e^{\sqrt{-1}\theta}$. 
\end{theorem}

Here, we observe the holomorphic structure on $L$ induced by $\nabla^{Y}$ 
under the assumption that $\nabla^{Y}$ is integrable. 
In Section 3.1 of \cite{LeungYauZaslow}, one can see that 
the integrability condition is equivalent to the existence of 
a locally defined smooth function $f$, which does not depends on $\tilde{y}$, 
so that $Y^{j}=\partial f/\partial \tilde{x}_{j}$. 
Put $\mathbf{e}:=e^{f}\cdot 1$ and regard this as a local frame of $L=\underline{\mathbb{C}}$. 
Then, one can see that $(\nabla^{Y}\mathbf{e})^{0,1}=0$. 
This means that $\mathbf{e}$ defines a holomorphic structure on $L$ with the associated Chern connection $\nabla^{Y}$. 

In the above explanation of the work of Leung, Yau and Zaslow, 
we pay attention to the following two properties. 

\begin{itemize}
    \item[(a)] The ambient space $W$ is (at least locally) diffeomorphic to the total space of a torus bundle. 
    Moreover, the coefficients of the K\"ahler form do not depend on $\tilde{y}$-coordinates 
    and are real values, see \eqref{WinLYZ}. 
    \item[(b)] There exists a holomorphic local frame $\mathbf{e}$ of $L$ so that $h(\overline{\mathbf{e}},\mathbf{e})$ does not depend on $\tilde{y}$-coordinates. 
    In the above case, we have $h(\overline{\mathbf{e}},\mathbf{e})=\langle \overline{\mathbf{e}},\mathbf{e} \rangle=e^{2f}$. 
\end{itemize}

Then, the first property (a) corresponds to locally semi-flat condition (Definition \ref{graphical}); 
the second one (b) corresponds to graphical condition (Definition \ref{graphical2}). 

These properties are also satisfied in the case where $W$ is the complement of the anti canonical divisor of a toric K\"ahler manifold, 
$L$ is a $T^{n}$-equivariant holomorphic line bundle and $h$ is a $T^{n}$-invariant Hermitian metric, see Section 9 of \cite{CollinsYau}. 
 
\subsection{Review of the work of Jacob and Yau}\label{JY}
In the work of Leung, Yau and Zaslow, main objects are connections. 
More precisely, those are Hermitian connections of a fixed complex line bundle $L$---rather than holomorphic apriori---with 
a given Hermitian metric $h$. 
As a consequence of dHYM condition, $L$ is given a holomorphic structure defined by the connection. 
Recently, Jacob and Yau \cite{JacobYau} switched main objects from connections to metrics. 
Namely, they fixed a holomorphic line bundle $L$, rather than complex, 
over a K\"ahler manifold $(X,\omega)$, 
and they tried to fined special Hermitian metrics of $L$ in the following sense. 

\begin{definition}\label{dHYMmetric}
A deformed Hermitian Yang--Mills metric with phase $e^{\sqrt{-1}\theta}$ is a Hermitian metric $h$ 
of $L$ so that its Chern connection satisfies 
\[\mathrm{Im}\left(e^{-\sqrt{-1}\theta}\bigl(\omega-F(h)\bigr)^{n}\right)=0, \]
where $F(h)$ is the curvature 2-form of the Chern connection explicitly given by 
$F(h):=(-1/2)\partial\bar{\partial}\log h$. 
\end{definition}

Readers may find that signs on the front of $F$ in Definition \ref{dHYMconnection} 
and \ref{dHYMmetric} are different. 
But, this is just a matter of convention. 
Actually, if $h$ is a dHYM metric of $L$ in the sense of Definition \ref{dHYMmetric}, 
then the Chern connection of $h^{-1}$ of $L^{-1}$ is a dHYM connection 
in the sense of Definition \ref{dHYMconnection} and vice versa. 

To find dHYM metrics, Jacob and Yau \cite{JacobYau} introduced 
a volume functional $V$ on the space of Hermitian metrics (see \eqref{volumeofh}) 
so that its minimizers are just dHYM metrics, 
and they studied its negative gradient flow. 
They named it the line bundle mean curvature flow and 
that is nothing but what we defined in Definition \ref{dfoflbmcf}. 

If the line bundle mean curvature flow has long time solution $\{\,h_{t}\,\}_{t\in[0,\infty)}$ and 
converges to some Hermitian metric $h_{\infty}$, 
we can say that $h_{\infty}$ is a dHYM metric since the flow is the negative gradient flow of $V$ 
and its minimizers are dHYM metrics. 
However, due to its nonlinearity, we do not know whether the flow exists for all time 
or blows up in finite time. 
Hence, it is very important to give a sufficient condition to ensure that a flow $h_{t}$ defined 
for $t\in [0,T)$ can be extended beyond $T$. 
Theorem 1.3 and Theorem 1.4 of \cite{JacobYau} are examples giving such sufficient conditions, 
and Proposition 5.2 of \cite{JacobYau} also can be considered as a sufficient condition. 
For comparison with our main theorem, we introduce Proposition 5.2 of \cite{JacobYau}. 

\begin{proposition}[Jacob and Yau, \cite{JacobYau}]\label{JacobYaupr52}
Suppose that $X$ is compact and $h_{t}$ is a line bundle mean curvature flow defined for $t\in [0,T)$. 
If there exist $A>0$ satisfying 
\[
    \frac{1}{A}\omega\leq \sqrt{-1}F(h_{t}) \leq A\omega
\]
for all $t\in [0,T)$, then $h_{t}$ can be extended beyond $T$. 
\end{proposition}

We note that replacing the assumption of Proposition \ref{JacobYaupr52} to 
\begin{equation}\label{introbd2}
-A\omega\leq \sqrt{-1}F(h_{t}) \leq A\omega
\end{equation}
for some $A>0$ 
causes serious problems because 
the positivity of all eigenvalues of $\sqrt{-1}F(h_{t})$ 
plays the important role in their proof relying on the Evans--Krylov theory. 
In that theory, the concavity of the operator 
$h \mapsto\theta(h)$ is essential and it is ensured by 
the positivity of all eigenvalues of $\sqrt{-1}F(h)$. 
In contrast, our main theorem (Theorem \ref{epregthm}) treats the case so that \eqref{introbd2} holds. 
It is written as $\sup_{V}|F(h(t))|\leq A$ in the theorem. 

\section{Scaling invariance}\label{scalinginvsec}
In this section, we fix some basic notations following \cite{JacobYau} and introduce a scaling which acts on line bundle mean curvature flows. 
Let $(X,g)$ be a K\"ahler manifold with $\dim_{\mathbb{C}}X=n$. 
Then, its K\"ahler form is locally given by
\[\omega=\frac{\sqrt{-1}}{2}g_{\bar{k}j}dz^{j}\wedge d\bar{z}^{k}. \]
Let $\pi:L\to X$ be a holomorphic line bundle. 
For a Hermitian metric $h$ on $L$, its curvature 2-from $F=F(h)$ is locally given by 
\[F=\frac{1}{2}F_{\bar{k}j}dz^{j}\wedge d\bar{z}^{k}:=-\frac{1}{2}\partial_{j}\partial_{\bar{k}}\log(h)dz^{j}\wedge d\bar{z}^{k}. \]
Then, one can easily prove that a complex number 
$Z_{L}:=\int_{X}(\omega-F(h))^{n}$ 
does not depend on the choice of metric $h$, see \cite{JacobYau} for detail. 
Hence, $Z_{L}\in\mathbb{C}$ is an invariant of $L$. 
Define $\zeta=\zeta(\omega,h):X\to \mathbb{C}$ by
\[
\zeta=\frac{(\omega-F(h))^{n}}{\omega^{n}}. 
\]
It is shown that $|\zeta|\geq 1$ in \cite{JacobYau}. 
We define $\theta=\theta(\omega,h):X\to (-\pi n/2,\pi n/2)$ by 
\[\theta:=\arctan \lambda_{1}+\dots+\arctan \lambda_{n}, \]
where $\lambda_{i}$ are eigenvalues of 
the endomorphism $K$ on $T^{1,0}X$ defined by 
\[K:=g^{j\bar{k}}F_{\bar{k}\ell}\frac{\partial}{\partial z^{j}}\otimes dz^{\ell}. \]
This definition of $\theta$ is based on the equation (2.5) in \cite{JacobYau} and it is called the {\it angle function} since it satisfies 
$\zeta/|\zeta|=e^{\sqrt{-1}\theta}$, 
see the equation (2.4) in \cite{JacobYau}. 

Then, in terms of the angle function, $h$ is a deformed Hermitian Yang--Mills metric with phase $e^{\sqrt{-1}\hat{\theta}}$ 
if and only if $\theta(\omega,h)=\hat{\theta}$. 
We also define a 1-form on $X$ by $H:=H(\omega,h)=d\theta$ 
and call it the {\em mean curvature 1-form} of $h$ with respect to $\omega$. 
Then, it is clear that $h$ is a deformed Hermitian Yang--Mills metric with some phase if and only if $H=0$. 
This is an analog of that a Lagrangian submanifold is special if and only if it is minimal. 

\begin{remark}\label{remforlbmcf}
Acting the exterior derivative to the both hand side of \eqref{defeqoflbmcf} and using the definition of line bundle mean curvature flows and $H=d\theta$, we get 
\[d\dot{u}=H, \]
where $\dot{u}$ is the time derivative of $u$. 
In this paper, we use this equation frequently. 
\end{remark}

The {\it volume}, mentioned in Subsection \ref{JY}, of a Hermitian metric $h$ of $L\to X$ with respect to $\omega$ is defied by 
\begin{equation}\label{volumeofh}
    V(\omega,h):=\int_{X}|\zeta|\frac{\omega^n}{n!}
\end{equation}
whenever it is finite. 
The {\it induced metric} of $h$ is also defined by 
\begin{equation}\label{defofeta}
\eta_{\bar{k}j}:=g_{\bar{k}j}+F_{\bar{k}\ell}g^{\ell\bar{m}}F_{\bar{m}j}. 
\end{equation}
Since $\eta$ is a positive $(1,1)$-form on $X$, 
we can define the following elliptic operator on $C^{\infty}(X)$: 
\[\Delta_{\eta}f:=\eta^{j\bar{k}}\partial_{j}\partial_{\bar{k}}f.\]
The following is the first variation formula of the volume given in \cite{JacobYau}. 

\begin{proposition}\label{firstvari}
For any smooth family of Hermitian metric 
$h_{t}=e^{-u(t)}h_{0}$ on $(X,\omega)$ 
so that $\mathop{\mathrm{supp}}u(t)$ is compact, we have 
\[
\frac{d}{dt}V(h_{t})=-\int_{X}\langle \bar{\partial}\dot{u},H^{(1,0)}\rangle_{\eta}|\zeta|\frac{\omega^n}{n!}
=\int_{X}(L_{\eta}\theta)\dot{u}|\zeta|\frac{\omega^n}{n!}, 
\]
where $L_{\eta}\theta:=\Delta_{\eta}\theta-\langle K^{*}(\bar{\partial}\theta),\partial \theta\rangle_{\eta}$, 
$(K^{*}(\bar{\partial}\theta))(\,\cdot\,):=(\bar{\partial}\theta)(\bar{K}(\,\cdot\,))$ and $H^{(1,0)}=\partial\theta$. 
\end{proposition}
\begin{proof}
The first equality is given by Proposition 3.4 in \cite{JacobYau}. 
To see the second equality, we first compute as follows: 
\[
\begin{aligned}
\eta^{\ell\bar{q}}H_{\ell}\nabla_{\bar{q}}\dot{u}|\zeta|=&
\nabla_{\bar{q}}(\eta^{\ell\bar{q}}H_{\ell}\dot{u}|\zeta|)
-\nabla_{\bar{q}}(\eta^{\ell\bar{q}}|\zeta|)H_{\ell}\dot{u}
-\eta^{\ell\bar{q}}\nabla_{\bar{q}}H_{\ell}\dot{u}|\zeta|. 
\end{aligned}
\]
By computation in the proof of Proposition 3.4 in \cite{JacobYau}, one sees that 
\[\nabla_{\bar{q}}(\eta^{\ell\bar{q}}|\zeta|)=-H_{\bar{u}}g^{i\bar{u}}F_{\bar{r}i}\eta^{\ell\bar{r}}|\zeta|. \]
By $\nabla_{\ell}\theta=H_{\ell}$, we have $\eta^{\ell\bar{q}}\nabla_{\bar{q}}H_{\ell}=\Delta_{\eta}\theta$ and $H_{\bar{u}}g^{i\bar{u}}F_{\bar{r}i}=(K^{*}(\bar{\partial}\theta))_{\bar{r}}$. 
Putting everything together and using the divergence theorem give the second equality. 
\end{proof}

From Proposition \ref{firstvari}, it follows that $h$ is a critical point of the volume functional if and only if its angle $\theta:X\to (-\pi n/2,\pi n/2)$ satisfies $L_{\eta}\theta=0$ and 
also that the volume is nonincreasing along a line bundle mean curvature flow $h_{t}$ since $\bar{\partial}\dot{u}=H^{(0,1)}$ 
by Remark \ref{remforlbmcf}. 

\begin{proposition}
If $X$ is compact, then for any initial metric $h$ of $L$ and constant $\hat{\theta}\in\mathbb{R}$, 
there exists $T>0$ and a solution $h_{t}$ of \eqref{defeqoflbmcf} defined for $t\in[0,T)$ with $h_{0}=h$. 
Moreover, the solution is unique. 
\end{proposition}
\begin{proof}
By the equation (5.1) in \cite{JacobYau}, we have 
\[\ddot{u}=\Delta_{\eta}\dot{u}\]
for a line bundle mean curvature flow $h_{t}=e^{-u(t)}h_{0}$. 
Since this is a strongly parabolic PDE for $f:=\dot{u}$, there exists $T>0$ and a unique solution $f$ of 
$\dot{f}=\Delta_{\eta}f$ defined for $t\in[0,T)$ with initial condition $f(0)=\theta(\omega,h)-\hat{\theta}$. 
For $t\in[0,T)$, define 
\[u(\,\cdot\,,t):=\int_{0}^{t}f(\,\cdot\,,s)ds. \]
Then, $u(\,\cdot\,,0)\equiv 0$ and $\ddot{u}=\Delta_{\eta}\dot{u}=\dot{\theta}$ where we used the equation (3.4) in \cite{JacobYau}. 
Thus, there exists a time-independent function $w$ on $X$ such that $\dot{u}=\theta-w$. 
Then, by the initial condition $\dot{u}(\,\cdot\,,0)=f(\,\cdot\,,0)=\theta(\omega,h)-\hat{\theta}$, 
we see that $w\equiv \hat{\theta}$. 
Thus, $h_{t}:=e^{-u(t)}h$ is a solution of \eqref{defeqoflbmcf} with $h_{0}=h$. 
The above construction indicates the uniqueness of solution. 
\end{proof}

The following reveals a scaling invariance of $\zeta$. 
\begin{proposition}\label{scinv}
For $a>0$ and $k\in\mathbb{N}_{>0}$, the function $\zeta=\zeta(\omega,h):X\to\mathbb{C}$ satisfies 
\[
\begin{aligned}
\zeta(\omega,ah)=\zeta(\omega,h) \quad\mbox{and}\quad
\zeta(k\omega,h^{\otimes k})=\zeta(\omega,h), 
\end{aligned}
\]
where $h^{\otimes k}$ is regarded as 
a Hermitian metric of $L^{\otimes k}$. 
\end{proposition}

\begin{proof}
The first one is clear since $F(ah)=F(h)$. The second one follows from
\[\zeta(k\omega,h^{\otimes k})=\frac{(k\omega-kF(h))^{n}}{(k\omega)^{n}}=\zeta(\omega,h) \]
since $F(h^{\otimes k})=kF(h)$. 
\end{proof}

\begin{proposition}\label{scoflbmcf}
Let $h=\{\,h_{t}\,\}_{t\in[0,T)}$ $(T<\infty)$ be a line bundle mean curvature flow of $L\to X$ with respect to $\omega$. 
For $a>0$, $k\in \mathbb{N}_{>0}$, $T'\in\mathbb{R}$ and $s\in [-kT',k(T-T'))$, define a Hermitian metric of $L^{\otimes k}$ by
\[\tilde{h}_{s}=a h_{t}^{\otimes k}\] 
with relation $t=T'+s/k$. 
Then, $\tilde{h}_{s}$ is a line bundle mean curvature flow on $L^{\otimes k}\to X$ with respect to $k\omega$ and initial metric $a h_{0}^{\otimes k}$. 
\end{proposition}
\begin{proof}
Put $h_{t}=e^{-u(t)}h_{0}$ and $\tilde{h}_{s}=e^{-\tilde{u}(s)}ah_{0}^{\otimes k}$. 
Then, we have $\tilde{u}(s)=k u(t)$ with relation $t=T'+s/k$. 
Thus, we have 
\[\frac{d}{ds}\tilde{u}(s)=k\frac{d}{dt}u(t)\times \frac{1}{k}=\frac{d}{dt}u(t). \]
On the other hand, by Proposition \ref{scinv}, we have $\theta(k\omega,ah_{t}^{\otimes k})=\theta(\omega,h_{t})$. 
Thus, the proof is complete. 
\end{proof}

\begin{proposition}
For $k\in \mathbb{N}_{>0}$, it holds that $|\tilde{\nabla}F(h^{\otimes k})|^2=|\nabla F(h)|^2/k$, 
where $\tilde{\nabla}$ is the Levi-Civita connection of $k\omega$ and 
$|\,\cdot\,|$ on the left hand side is the norm with respect to $k\omega$. 
\end{proposition}
\begin{proof}
It is clear that $F(h^{\otimes k})=k F(h)$ 
and $\tilde{\nabla}=\nabla$. 
Thus, $\tilde{\nabla}F(h^{\otimes k})=k\nabla F(h)$. 
Let $\tilde{e}_{j}$ be a local orthonormal frame with respect to $k\omega$. 
Then, $e_{j}:=\sqrt{k}\tilde{e}_{j}$ becomes a local orthonormal frame with respect to $\omega$ and 
\[
(\tilde{\nabla}F(h^{\otimes k}))(\tilde{e}_{i},\tilde{e}_{j},\tilde{e}_{k})=
k(\nabla F(h))(e_{i},e_{j},e_{k})\times \frac{1}{\sqrt{k}^3}
=\frac{1}{\sqrt{k}}(\nabla F(h))(e_{i},e_{j},e_{k}). 
\]
Then, the proof is complete. 
\end{proof}

\begin{definition}\label{defofscaling}
Let $((X,\omega),L,h)$ be a triplet of a K\"ahler manifold $(X,g)$, 
a holomorphic line bundle $\pi:L\to X$ and a line bundle mean curvature flow 
$h=\{\,h_{t}\,\}_{t\in[0,T)}$ of $L$. For given $T'\in\mathbb{R}$ and 
$k\in\mathbb{N}_{>0}$, we define the scaling operator $D_{k}^{T'}$ by 
$D_{k}^{T}((X,\omega),L,h):=((X,k\omega),L^{\otimes k},D_{k}^{T'}h)$, 
where $(D_{k}^{T'}h)_{s}$ is defined by 
\[(D_{k}^{T'}h)_{s}:=h_{t}^{\otimes k}\]
for $s\in[-kT',k(T-T'))$ with relation $t=T'+s/k$. 
\end{definition}

By Proposition \ref{scoflbmcf}, we see that $D_{k}^{T'}h$ is a line bundle 
mean curvature flow of $L^{\otimes k}$ on $(X,k\omega)$. 

\section{Divergence theorem}\label{sec:divergence}
In this section, we build a parallel framework of Hermitian metrics with geometry of submanifolds 
and give an analog of the divergence theorem for submanifolds. 
We also give an application of it in the latter subsection. 
\subsection{A divergence theorem}
We fix a K\"ahler manifold $(X,g)$ with $\dim_{\mathbb{C}}X=n$ and  
a holomorphic line bundle $\pi:L\to X$.  
For a Hermitian metric $h$, a new measure $d\mu(h)$ on $X$ is defied by 
\[d\mu(h)=|\zeta|\frac{\omega^{n}}{n!}. \]
Put $v:=|\zeta|:X\to\mathbb{R}^{+}$. 
For a smooth section $Y=Y^{i}(\partial/\partial z^{i})$ of $T^{1,0}X$, the $v$-weighted divergence of $Y$ is defined by
\[\mathop{\mathrm{div}_{v}}Y:=v^{-1}\nabla_{i}\left( v Y^{i} \right).\]
Then, by the usual divergence theorem and the definition of $d\mu(h)$, we have 
\begin{equation}\label{eq:stddivthm}
\int_{U}\mathop{\mathrm{div}_{v}}Y d\mu(h)=\int_{\partial U}g\left(vY,\nu\right)d\mu|_{\partial U}
\end{equation}
on a relatively compact open set $U\subset X$ with piecewise smooth boundary $\partial U$, 
where $d\mu|_{\partial U}$ is the induced measure on $\partial U$ with respect to the induced metric $g|_{\partial U}$ 
and $\nu$ is the outer unit normal vector field along $\partial U$. 

On a chart $U$ with holomorphic coordinates $(z_{1},\dots,z_{n})$, put
\[
\mathcal{E}_{i}:=
\frac{\partial}{\partial z^{i}}\oplus F_{\bar{j}i}d\bar{z}^{j}
=\frac{\partial}{\partial z^{i}}\oplus \bar{\partial}\left(\frac{\partial (-\log h)}{\partial z^{i}}\right)
\]
for $i=1,\dots,n$. 
It is clear that $\{\mathcal{E}_{i}\}_{i=1}^{n}$ are $\mathbb{C}$-linearly independent sections of $T^{1,0}X\oplus \Lambda^{0,1}X$ over $U$. 
Here, $V\oplus W:=\{\,v\oplus w:=(v,w)\in V\times W\,\}$ is the formal sum 
of vector spaces $V$, $W$ with sum $(v_{1}\oplus w_{1})+(v_{2}\oplus w_{2}):=(v_{1}+v_{2})\oplus (w_{1}+w_{2})$ and scalar product $\lambda\cdot (v\oplus w):=(\lambda v)\oplus (\lambda w)$. 
Let $U'$ be another chart with holomorphic coordinates $(w_{1},\dots,w_{n})$ satisfying $U\cap U'\neq \emptyset$, 
and put 
\[
\mathcal{E}'_{j}:=\frac{\partial}{\partial w^{j}}\oplus \bar{\partial}\left(\frac{\partial (-\log h)}{\partial w^{j}}\right). 
\]
Then, on $U\cap U'$, it follows that 
\begin{equation}\label{changeofe}
\begin{aligned}
\mathcal{E}_{i}=&\left(\sum_{j=1}^{n}\frac{\partial w^{j}}{\partial z^{i}}\frac{\partial}{\partial w^{j}}\right)
\oplus \bar{\partial}\left(\sum_{j=1}^{n}\frac{\partial w^{j}}{\partial z^{i}}\frac{\partial(-\log h)}{\partial w^{j}}\right)\\
=&\left(\sum_{j=1}^{n}\frac{\partial w^{j}}{\partial z^{i}}\frac{\partial}{\partial w^{j}}\right)
\oplus \left(\sum_{j=1}^{n}\frac{\partial w^{j}}{\partial z^{i}}\bar{\partial}\left(\frac{\partial(-\log h)}{\partial w^{j}}\right)\right)\\
=&\sum_{j=1}^{n}\frac{\partial w^{j}}{\partial w^{i}}\left(\frac{\partial}{\partial w^{j}}\oplus \bar{\partial}\left(\frac{\partial(-\log h)}{\partial w^{j}}\right)\right)\\
=&\sum_{j=1}^{n}\frac{\partial w^{j}}{\partial z^{i}} \mathcal{E}'_{j}. 
\end{aligned}
\end{equation}

Thus, transition functions from $\{\mathcal{E}_{i}\}_{i=1}^{n}$ to $\{\mathcal{E}_{i}'\}_{i=1}^{n}$ are holomorphic, 
and the following definition makes sense. 

\begin{definition}
For a Hermitian metric $h$ on $L$, a holomorphic subbundle of $T^{1,0}X\oplus \Lambda^{0,1}X$ of rank $n$, denoted by $Th$, is defined by 
\[Th|_{U}:=\mathrm{Span}_{\mathbb{C}}\{\,\mathcal{E}_{1},\dots,\mathcal{E}_{n}\,\}\]
on each $U$. 
We call this subbundle $Th\subset T^{1,0}X\oplus \Lambda^{0,1}X$ the {\it tangent bundle} of $h$. 
\end{definition}

\begin{remark}\label{whyEiisso}
The notion of $Th$ is an analog of the tangent bundle $TL$ 
of a Lagrangian submanifold $L\subset \mathbb{C}^{n}$ which is written 
as the graph of the gradient of a function. 
Precisely, the tangent bundle of a Lagrangian submanifold $L=\{\,(x,\nabla\psi(x))\mid x\in\mathbb{R}^{n}\,\}$, where $\psi=\psi(x)$ is 
a smooth function on $\mathbb{R}^{n}$, is spanned by 
\begin{equation}\label{subei}
    E_{i}:=\frac{\partial}{\partial x^{i}}+\frac{\partial^2 \psi}{\partial x^{i}\partial x^{j}}\frac{\partial}{\partial y^{j}},\quad i=1,\dots, n. 
\end{equation}
\end{remark}

Note that $Th$ is holomorphically isomorphic to $T^{1,0}X$ since the transition functions are 
$\partial w^{j}/\partial z^{i}$ by \eqref{changeofe}. 
Actually, the isomorphism is given by $\mathcal{E}_{i}\mapsto \partial /\partial z^{i}$. 
We denote this isomorphism by $\widetilde{\bullet}:Th\to T^{1,0}X$. 

\begin{definition}
Let $\mathcal{Y}$ and $\mathcal{Z}$ be smooth sections of $T^{1,0}X\oplus \Lambda^{0,1}X$ with local expressions 
\begin{equation}\label{expYZ}
    \mathcal{Y}=Y^{j}\frac{\partial}{\partial z^{j}}\oplus Y_{\bar{j}}d\bar{z}^{j}\quad\text{and}\quad 
\mathcal{Z}=Z^{k}\frac{\partial}{\partial z^{k}}\oplus Z_{\bar{k}}d\bar{z}^{k}. 
\end{equation}
Then, a Hermitian metric $\langle\,\cdot\,,\,\cdot\,\rangle$ on $T^{1,0}X\oplus \Lambda^{0,1}X$ is defined by 
\[\langle \overline{\mathcal{Y}},\mathcal{Z}\rangle:=g_{\bar{j}k}\overline{Y^{j}}Z^{k}+g^{j\bar{k}}\overline{Y_{\bar{j}}}Z_{\bar{k}}. \]
The orthogonal compliment of $Th\subset T^{1,0}X\oplus \Lambda^{0,1}X$ with respect to this Hermitian metric is denoted by $T^{\bot}h$ 
and called the {\it normal bundle} of $h$. 
\end{definition}

\begin{definition}
Let $\mathcal{Y}$ be a smooth section of $T^{1,0}X\oplus \Lambda^{0,1}X$. 
We denote the $Th$-part (resp., $T^{\bot}h$-part) of $\mathcal{Y}$ by $\mathcal{Y}^{\top}$ (resp., $\mathcal{Y}^{\bot}$), 
and call it the {\it tangential part} (resp., the {\it normal part}) of $\mathcal{Y}$ with respect to $h$. 
Moreover, we call type $(1,0)$ vector field $\widetilde{\mathcal{Y}^{\top}}$ the {\it associated vector field} with $\mathcal{Y}$. 
\end{definition}

Since the Hermitian metric $\langle\,\cdot\,,\,\cdot\,\rangle$ of $T^{1,0}X\oplus \Lambda^{0,1}X$ and 
the induced metric $\eta$ on $T^{1,0}X$ perform nicely as 
\[
\langle \overline{\mathcal{E}_{i}},\mathcal{E}_{\ell}\rangle
=g_{\bar{i}\ell}+g^{j\bar{k}}F_{\bar{i}j}F_{\bar{k}\ell}
=\eta_{\bar{i}\ell}, 
\]
the tangential part of $\mathcal{Y}$ with respect to $h$ and its associated vector field are easily written by
\begin{equation}\label{YandYtil}
\mathcal{Y}^{\top}:=\eta^{i\bar{j}}\langle \overline{\mathcal{E}_{j}},\mathcal{Y}\rangle \mathcal{E}_{i}
\quad\mbox{and}\quad
\widetilde{\mathcal{Y}^{\top}}:=\eta^{i\bar{j}}\langle \overline{\mathcal{E}_{j}},\mathcal{Y}\rangle \frac{\partial}{\partial z^{i}}. 
\end{equation}
Moreover, smooth sections $\mathcal{F}_{i}$ of $T^{1,0}X\oplus \Lambda^{0,1}X$ defined by 
\[\mathcal{F}_{i}:=\left(-F_{\bar{k}i}g^{j\bar{k}}\frac{\partial}{\partial z^{j}}\right)\oplus \left(g_{\bar{\ell}i}d\bar{z}^{\ell}\right)\]
satisfy $\langle\overline{\mathcal{F}_{i}}, \mathcal{F}_{j}\rangle=\eta_{\bar{i}j}$ and 
$\langle\overline{\mathcal{E}_{i}}, \mathcal{F}_{j}\rangle=0$. 
Thus, $\{\, \mathcal{F}_{i}\,\}_{i=1}^{n}$ is a basis of $T^{\bot}h$, 
and the normal part of $\mathcal{Y}$ with respect to $h$ is given by 
\begin{equation}\label{Yn1}
\mathcal{Y}^{\bot}=\langle\overline{\mathcal{F}_{i}}, \mathcal{Y}\rangle \eta^{j\bar{i}}\mathcal{F}_{j}. 
\end{equation}

\begin{definition}
Let $\mathcal{Y}$ be a smooth sections of $T^{1,0}X\oplus \Lambda^{0,1}X$ with a local expression as in \eqref{expYZ}. 
Then, we define its {\it divergence along} $h$, which is a smooth function on $X$, by 
\begin{equation}\label{divalongh}
    \mathop{\mathrm{div}_{h}}\mathcal{Y}:=\nabla_{i}Y^{k}\eta^{i\bar{j}}g_{\bar{j}k}
    +\nabla_{i}Y_{\bar{k}}\eta^{i\bar{j}}F_{\bar{j}\ell}g^{\ell\bar{k}}. 
\end{equation}
\end{definition}

\begin{remark}
The reason why we define the divergence along $h$ as above is the following. 
As in Remark \ref{whyEiisso}, consider the graphical Lagrangian submanifold $L=\{\,(x,\nabla\psi(x))\mid x\in\mathbb{R}^{n}\,\}$. 
Then, its tangent bundle is spanned by $E_{i}$ defined in Remark \ref{whyEiisso}. 
Assume that a vector field $Z=X_{i}(\partial/\partial x^{i})+Y_{i}(\partial/\partial y^{i})$ along $L$ is given. 
Then, the usual divergence of $Z$ along $L$ is given by $\mathop{\mathrm{div}_{L}}Z:=\sum_{i=1}^{n}\nabla_{E_{i}}\langle E_{i},Z\rangle$. 
Expanding the right hand side of this with \eqref{subei}, one can find similarities between it and \eqref{divalongh}. 
\end{remark}

\begin{definition}
For a Hermitian metric $h$ of $L$, we define the {\it mean curvature section}, which is a smooth section of $T^{1,0}X\oplus \Lambda^{0,1}X$, by
\[
\mathcal{H}=\mathcal{H}(\omega,h):=\left(-g^{q\bar{k}}H_{\bar{p}}\eta^{\ell\bar{p}}F_{\bar{k}\ell}\frac{\partial}{\partial z^{q}}\right)
\oplus\left(g_{\bar{q}k}H_{\bar{\ell}}\eta^{k\bar{\ell}}d\bar{z}^{q}\right). 
\]
\end{definition}

The mean curvature section has some nice properties. First, it holds that 
\begin{equation}\label{normofHH}
\begin{aligned}
|\mathcal{H}|^2=&g^{k\bar{k}'}F_{\bar{\ell}k}F_{\bar{k}'\ell'}H_{p}\eta^{p\bar{\ell}}
H_{\bar{p}'}\eta^{\ell'\bar{p}'}
+g_{\bar{\ell}\ell'}H_{p}\eta^{p\bar{\ell}}H_{\bar{p}'}\eta^{\ell'\bar{p}'}\\
=&\left(g_{\bar{\ell}\ell'}+g^{k\bar{k}'}F_{\bar{\ell}k}F_{\bar{k}'\ell'}\right)H_{p}\eta^{p\bar{\ell}}H_{\bar{p}'}\eta^{\ell'\bar{p}'}\\
=&\eta_{\bar{\ell}\ell'}H_{p}\eta^{p\bar{\ell}}H_{\bar{p}'}\eta^{\ell'\bar{p}'}\\
=&\eta^{\ell\bar{q}}H_{\ell}H_{\bar{q}}\\
=&|H^{(1,0)}|^2_{\eta}. 
\end{aligned}
\end{equation}
Second, the mean curvature section of $h$ is normal to the tangent bundle $Th$, that is, $\mathcal{H}^{\top}=0$. It easily follows from
\[
\begin{aligned}
\langle \overline{\mathcal{E}_{i}}, \mathcal{H}\rangle
=&g_{\bar{j}q}\left(-g^{q\bar{k}}H_{\bar{p}}\eta^{\ell\bar{p}}F_{\bar{k}\ell}\right)\delta^{\bar{j}}_{\bar{i}}
+g^{j\bar{q}}\left(g_{\bar{q}k}H_{\bar{\ell}}\eta^{k\bar{\ell}}\right)F_{\bar{i}j}\\
=&-H_{\bar{p}}\eta^{\ell\bar{p}}F_{\bar{i}\ell}
+H_{\bar{\ell}}\eta^{k\bar{\ell}}F_{\bar{i}k}\\
=&0. 
\end{aligned}
\]
In the geometry of submanifolds, it is well-known that the mean curvature vector field of a submanifold 
in a Riemannian manifold is normal, 
and the above property can be considered as an analog of that. 
The following is an analog of the divergence theorem for vector fields along submanifolds. 

\begin{theorem}\label{thm:div}
For any smooth section $\mathcal{Y}$ of $T^{1,0}X\oplus \Lambda^{0,1}X$, it holds that 
\begin{equation}\label{eq:divdivH}
\mathop{\mathrm{div}_{v}}\widetilde{\mathcal{Y}^{\top}}
=\mathop{\mathrm{div}_{h}}\mathcal{Y}+\langle \overline{\mathcal{H}},\mathcal{Y}\rangle. 
\end{equation}
Moreover, on a relatively compact open set $U\subset X$ with piecewise smooth boundary $\partial U$, we have 
\begin{equation}\label{eq:divthm1}
\int_{U} \mathop{\mathrm{div}_{h}}\mathcal{Y} d\mu(h)
=-\int_{U} \langle \overline{\mathcal{H}},\mathcal{Y}\rangle d\mu(h)
+\int_{\partial U} g\left(v\widetilde{\mathcal{Y}^{\top}},\nu\right) d\mu|_{\partial U}. 
\end{equation}
\end{theorem}

\begin{proof}
We will expand $\mathop{\mathrm{div}_{v}}\widetilde{\mathcal{Y}^{\top}}$ explicitly. 
Since 
\[\langle \overline{\mathcal{E}_{j}},\mathcal{Y}\rangle=g_{\bar{j}k}Y^{k}+g^{\ell\bar{k}}F_{\bar{j}\ell}Y_{\bar{k}}, \]
we have
\begin{equation}\label{eq:used1}
v \eta^{i\bar{j}}\langle \overline{\mathcal{E}_{j}},\mathcal{Y}\rangle
=Y^{k}\left(v \eta^{i\bar{j}}\right)g_{\bar{j}k}+Y_{\bar{k}}\left(v\eta^{i\bar{j}}F_{\bar{j}\ell}\right)g^{\ell\bar{k}}. 
\end{equation}
This is the coefficient of $\partial/\partial z^{i}$ of $v\widetilde{\mathcal{Y}^{\top}}$ by \eqref{YandYtil}. 
Thus, 
\begin{equation}\label{eq:divhAA}
\begin{aligned}
\mathop{\mathrm{div}_{v}}\widetilde{\mathcal{Y}^{\top}}
=&v^{-1}\nabla_{i}\left(v \eta^{i\bar{j}}\langle \overline{\mathcal{E}_{j}},\mathcal{Y}\rangle\right)\\
=&\nabla_{i}Y^{k}\eta^{i\bar{j}}g_{\bar{j}k}+\nabla_{i}Y_{\bar{k}}\eta^{i\bar{j}}F_{\bar{j}\ell}g^{\ell\bar{k}}\\
&+Y^{k}v^{-1}\nabla_{i}\left(v \eta^{i\bar{j}}\right)g_{\bar{j}k}+Y_{\bar{k}}v^{-1}\nabla_{i}\left(v\eta^{i\bar{j}}F_{\bar{j}\ell}\right)g^{\ell\bar{k}}\\
=&\mathop{\mathrm{div}_{h}}\mathcal{Y}
+Y^{k}v^{-1}\nabla_{i}\left(v \eta^{i\bar{j}}\right)g_{\bar{j}k}
+Y_{\bar{k}}v^{-1}\nabla_{i}\left(v\eta^{i\bar{j}}F_{\bar{j}\ell}\right)g^{\ell\bar{k}}\\
=&:\mathop{\mathrm{div}_{h}}\mathcal{Y}+Y^{k}A_{k}+Y_{\bar{k}}B^{\bar{k}}. 
\end{aligned}
\end{equation}
We further compute $A_{k}$ and $B^{\bar{k}}$. 
First, we focus on $A_{k}$. Then, we have
\[
\begin{aligned}
A_{k}=&\nabla_{i}\log v \eta^{i\bar{j}} g_{\bar{j}k}+\nabla_{i}\eta^{i\bar{j}} g_{\bar{j}k}\\
=&\eta^{\ell\bar{m}}F_{\bar{m}p}g^{p\bar{q}}\nabla_{i}F_{\bar{q}\ell} \eta^{i\bar{j}} g_{\bar{j}k}
+\nabla_{i}\eta^{i\bar{j}} g_{\bar{j}k}, 
\end{aligned}
\]
where the second equality follows from the formula appeared just after the equation (5.5) in \cite{JacobYau}. 
Note that
\[
\begin{aligned}
\nabla_{i}\eta^{i\bar{j}}=&-\eta^{i\bar{s}}\eta^{r\bar{j}}\nabla_{i}\eta_{\bar{s}r}\\
=&-\eta^{i\bar{s}}\eta^{r\bar{j}}\nabla_{i}(g_{\bar{s}r}+g^{\bar{q}p}F_{\bar{s}p}F_{\bar{q}r})\\
=&-\eta^{i\bar{s}}\eta^{r\bar{j}}g^{\bar{q}p}\nabla_{i}F_{\bar{s}p}F_{\bar{q}r}-\eta^{i\bar{s}}\eta^{r\bar{j}}g^{\bar{q}p}F_{\bar{s}p}\nabla_{i}F_{\bar{q}r}\\
=&-\eta^{r\bar{j}}g^{\bar{q}p}H_{p}F_{\bar{q}r}-\eta^{i\bar{s}}\eta^{r\bar{j}}g^{\bar{q}p}F_{\bar{s}p}\nabla_{i}F_{\bar{q}r}, 
\end{aligned}
\]
where the final equality follows from the identity $\eta^{i\bar{s}}\nabla_{i}F_{\bar{s}p}=H_{p}$. Thus, 
\begin{equation}\label{A1pre}
\begin{aligned}
A_{k}=&\eta^{\ell\bar{m}}F_{\bar{m}p}g^{p\bar{q}}\nabla_{i}F_{\bar{q}\ell} \eta^{i\bar{j}} g_{\bar{j}k}
-\eta^{r\bar{j}}g^{p\bar{q}}H_{p}F_{\bar{q}r} g_{\bar{j}k}\\
&\quad\quad-\eta^{i\bar{s}}\eta^{r\bar{j}}g^{\bar{q}p}F_{\bar{s}p}\nabla_{i}F_{\bar{q}r} g_{\bar{j}k}. 
\end{aligned}
\end{equation}

Here, to simplify each term, we introduce the so-called {\it normal coordinates}, which are also used in \cite{JacobYau}. 
For a fixed point $p\in X$, the normal coordinates (centered at $p$) are coordinates $(z^{1},\dots,z^{n})$ so that 
$g_{\bar{k}j}=\delta_{kj}$ and $F_{\bar{k}j}=\lambda_{j}\delta_{kj}$ at $p$, 
where $\lambda_{j}$ ($j=1,\dots,n$) are the eigenvalues of $F$. 
Using the normal coordinates, only at $p$, we have
\[
\begin{aligned}
\eta^{\ell\bar{m}}F_{\bar{m}p}g^{p\bar{q}}\nabla_{i}F_{\bar{q}\ell} \eta^{i\bar{j}} g_{\bar{j}k}
=&\sum_{i,\ell=1}^{n}(1+\lambda_{\ell}^2)^{-1}\lambda_{\ell}\nabla_{i}F_{\bar{\ell}\ell} (1+\lambda_{i}^2)^{-1}\delta_{ik},\\
\eta^{i\bar{s}}\eta^{r\bar{j}}g^{\bar{q}p}F_{\bar{s}p}\nabla_{i}F_{\bar{q}r} g_{\bar{j}k}
=&\sum_{\ell,i=1}^{n}(1+\lambda_{\ell}^2)^{-1}(1+\lambda_{i}^2)^{-1}\lambda_{\ell}\nabla_{\ell}F_{\bar{\ell}i}\delta_{ik}. 
\end{aligned}
\]
Moreover, it holds that 
\[
\begin{aligned}
\nabla_{\ell}F_{\bar{\ell}i}
&=\partial_{\ell}(\partial_{\bar{\ell}}\partial_{i}(-\log h))-\Gamma_{\ell i}^{k}\partial_{\bar{\ell}}\partial_{k}(-\log h)\\
=&\partial_{i}(\partial_{\bar{\ell}}\partial_{\ell}(-\log h))-\Gamma_{i \ell}^{k}\partial_{\bar{\ell}}\partial_{k}(-\log h)
=\nabla_{i}F_{\bar{\ell}\ell} 
\end{aligned}
\]
since $(X,\omega)$ is K\"ahler. 
Thus, the first and third term on the right hand side of \eqref{A1pre} cancel each other. 
On the second term of \eqref{A1pre}, by using the normal coordinates, we have
\[
\eta^{r\bar{j}}g^{p\bar{q}}F_{\bar{q}r} g_{\bar{j}k}
=(1+\lambda_{p}^2)^{-1}\lambda_{p}\delta_{p k}
=\eta^{p\bar{\ell}}F_{\bar{\ell}k}. 
\]
These imply that 
\begin{equation}\label{eq:A1}
A_{k}=-H_{p}\eta^{p\bar{\ell}}F_{\bar{\ell}k}. 
\end{equation}
Next, we treat $B^{\bar{k}}$. Then, we have 
\[
B^{\bar{k}}=v^{-1}\nabla_{i}\left(v\eta^{i\bar{j}}\right)F_{\bar{j}\ell}g^{\ell\bar{k}}
+\eta^{i\bar{j}}\nabla_{i}F_{\bar{j}\ell}g^{\ell\bar{k}}. 
\]
Note that, by a consequence of the computation of $A_{k}$, we have shown that 
\begin{equation}\label{used2}
v^{-1}\nabla_{i}\left(v\eta^{i\bar{j}}\right)=-\eta^{r\bar{j}}g^{p\bar{q}}H_{p}F_{\bar{q}r}. 
\end{equation}
Combining these with the general identity $\eta^{i\bar{j}}\nabla_{i}F_{\bar{j}\ell}=H_{\ell}$ yields that 
\[
B^{\bar{k}}=-\eta^{r\bar{j}}g^{p\bar{q}}H_{p}F_{\bar{q}r}F_{\bar{j}\ell}g^{\ell\bar{k}}+H_{\ell}g^{\ell\bar{k}}
=H_{\ell}\left(g^{\ell\bar{k}}-\eta^{r\bar{j}}g^{\ell\bar{q}}F_{\bar{q}r}F_{\bar{j}p}g^{p\bar{k}}\right). 
\]
By using the normal coordinates, one can see that 
\[
g^{\ell\bar{k}}-\eta^{r\bar{j}}g^{\ell\bar{q}}F_{\bar{q}r}F_{\bar{j}p}g^{p\bar{k}}
=\delta_{\ell k}-(1+\lambda_{\ell}^2)^{-1}\lambda_{\ell}^2\delta_{\ell k}
=(1+\lambda_{\ell}^2)^{-1}\delta_{\ell k}
=\eta^{\ell\bar{k}}, 
\]
and this implies that
\begin{equation}\label{eq:A2}
B^{\bar{k}}=H_{\ell}\eta^{\ell\bar{k}}. 
\end{equation}
Then, substituting \eqref{eq:A1} and \eqref{eq:A2} into \eqref{eq:divhAA} yields 
\[
\begin{aligned}
\mathop{\mathrm{div}_{v}}\widetilde{\mathcal{Y}^{\top}}
=&\mathop{\mathrm{div}_{h}}\mathcal{Y}-Y^{k}H_{p}\eta^{p\bar{\ell}}F_{\bar{\ell}k}+Y_{\bar{k}}H_{\ell}\eta^{\ell\bar{k}}\\
=&\mathop{\mathrm{div}_{h}}\mathcal{Y}-g_{\bar{q}i}Y^{i}\overline{g^{q\bar{k}}H_{\bar{p}}\eta^{\ell\bar{p}}F_{\bar{k}\ell}}
+g^{q\bar{i}}Y_{\bar{i}}\overline{g_{\bar{q}k}H_{\bar{\ell}}\eta^{k\bar{\ell}}}\\
=&\mathop{\mathrm{div}_{h}}\mathcal{Y}+\langle \overline{\mathcal{H}},\mathcal{Y}\rangle, 
\end{aligned}
\]
and this is the first desired formula \eqref{eq:divdivH}. 
Integrating both hand side of \eqref{eq:divdivH} with the divergence theorem \eqref{eq:stddivthm} 
deduces the second desired formula \eqref{eq:divthm1}. 
\end{proof}

\begin{remark}
Theorem \eqref{thm:div} can be considered as an analog of the divergence formula for a submanifold, which is also called the first variation formula. 
Actually, for a submanifold $L$ in a Riemannian manifold $(M,g)$ 
and a section $V$ of $TM$ along $L$ with compact support, 
it holds that 
\[
\int_{L} \mathop{\mathrm{div}_{L}} V d\mu_{L}=-\int_{L} g(H,V) d\mu_{L}, 
\]
where $\mathop{\mathrm{div}_{L}}$ is the divergence of $V$ along $L$, $H$ is the mean curvature vector field of $L$ and $d\mu_{L}$ is the induced measure on $L$. 
\end{remark}

\subsection{An application of the divergence theorem}
In this subsection, we give an application of the divergence formula \eqref{eq:divthm1}. 
Recall that $(X,g)$ is a given K\"ahler manifold with $\mathrm{dim}_{\mathbb{C}}X=n$ 
and $\pi:L\to X$ is a holomorphic line bundle. 
Recall that we introduced special conditions for $(X,g)$, called the semi-flat condition in Definition \ref{graphical}, 
and for $(L,h)$, called the graphical condition in Definition \ref{graphical2}. 
We also remark that from the former condition in \eqref{lsfc} it follows that 
\begin{equation}\label{gxiyj0}
g\left(\frac{\partial}{\partial x^{i}}, \frac{\partial}{\partial y^{j}}\right)=0
\end{equation}
for all $1\leq i,j\leq n$. 

\begin{definition}
Assume that $(X,g)$ is locally semi-flat on $U\subset X$ with the coordinates $(z^{1},\dots,z^{n})$ 
induced from $\varphi:B(r)\times B(r')\to U$ and $(L,h)$ is graphical with respect to a section $e\in\Gamma(U,L)$. 
Then, we define a smooth function $\phi:U\to \mathbb{R}$ by 
\[\phi:=-\log h(\bar{e},e). \]
Put $U_{\delta}:=\varphi(B(\delta)\times B(r'))$ for $\delta\in(0,r]$, 
where the radius of the first component is changed and the second one is fixed. 
Then, for $p\in U_{r/4}$, we define a smooth section of $T^{1,0}X\oplus \Lambda^{0,1}X$ over $U_{3r/4}$ by 
\[
\mathcal{P}_{p}:=\left(2(x^{k}-x_{0}^{k})\frac{\partial}{\partial z^{k}}\right)\oplus \left(\frac{1}{2}\frac{\partial \phi}{\partial x^{k}}d\bar{z}^{k}\right), 
\]
where $x_{0}^{k}$ are the coordinates of the $B(r/4)$-component of $\varphi^{-1}(p)\in B(r/4)\times B(r')$. 
We call $\mathcal{P}_{p}$ the {\it position section} of $h$ centered at $p$ and usually omit the subscript $p$. 
\end{definition}

\begin{definition}
For a smooth function $f:X\to \mathbb{C}$, we define a differential operator $\mathcal{D}$ by
\[
\mathcal{D}f:=\left( \nabla_{\bar{i}}f\eta^{j\bar{i}}\frac{\partial}{\partial z^{j}}\right)
\oplus\left(\nabla_{\bar{i}}f\eta^{\bar{i}j}F_{\bar{\ell}j}d\bar{z}^{\ell}\right). 
\]
\end{definition}

It is clear that $\mathcal{D}f$ is a smooth section of $T^{1,0}X\oplus \Lambda^{0,1}X$ and satisfies 
\[
\mathcal{D}(f_{1}f_{2})=f_{1}\mathcal{D}f_{2}+f_{2}\mathcal{D}f_{1}. 
\]

\begin{lemma}
A position section $\mathcal{P}$ and a smooth function $f$ on $U$ satisfy 
\begin{equation}\label{eq:dvhfP2}
\mathop{\mathrm{div}_{h}}(f\mathcal{P})=\langle \overline{\mathcal{D}f},\mathcal{P}\rangle+nf. 
\end{equation}
\end{lemma}

\begin{proof}
Since $\partial \phi/\partial y^{k}=0$ for all $k$ by \eqref{lsfc2}, we have 
\[
\frac{1}{2}\frac{\partial\phi}{\partial x^{k}}=\nabla_{\bar{k}}\phi
\quad\mbox{and}\quad
\nabla_{i}\frac{\partial \phi}{\partial x^{k}}=\nabla_{\bar{i}}\frac{\partial \phi}{\partial x^{k}}=2F_{\bar{i}k}\,\,(=2F_{\bar{k}i}). 
\]
By the definition of $\mathop{\mathrm{div}_{h}}$, see \eqref{eq:divhAA}, 
and noting 
\[\nabla_{i}(x^{k}-x_{0}^{k})=\frac{1}{2}\left(\frac{\partial}{\partial x^{i}}-\sqrt{-1}\frac{\partial}{\partial y^{i}}\right)x^{k}=\frac{1}{2}\delta_{ik},\] 
we have
\begin{equation}\label{used3}
\begin{aligned}
\mathop{\mathrm{div}_{h}}(f\mathcal{P})=&2\nabla_{i}\left(f(x^{k}-x_{0}^{k})\right)\eta^{i\bar{j}}g_{\bar{j}k}
+\nabla_{i}\left(f\nabla_{\bar{k}}\phi\right)\eta^{i\bar{j}}F_{\bar{j}\ell}g^{\ell\bar{k}}\\
=&\left(2\nabla_{i}f(x^{k}-x_{0}^{k})+f\delta_{ik}\right)\eta^{i\bar{j}}g_{\bar{j}k}
+\left(\nabla_{i}f\nabla_{\bar{k}}\phi+fF_{\bar{k}i}\right)\eta^{i\bar{j}}F_{\bar{j}\ell}g^{\ell\bar{k}}\\
=&2\nabla_{i}f\eta^{i\bar{j}}g_{\bar{j}k}(x^{k}-x_{0}^{k})+\nabla_{i}f\eta^{i\bar{j}}F_{\bar{j}\ell}g^{\ell\bar{k}}\nabla_{\bar{k}}\phi
+f\eta^{i\bar{j}}\left(g_{\bar{j}i}+F_{\bar{k}i}F_{\ell\bar{j}}g^{\ell\bar{k}}\right)\\
=&\langle \overline{\mathcal{D}f},\mathcal{P}\rangle+fn, 
\end{aligned}
\end{equation}
where the last equality follows from $g_{\bar{j}i}+F_{\bar{k}i}F_{\bar{j}\ell}g^{\ell\bar{k}}=\eta_{\bar{j}i}$. 
\end{proof}

\begin{lemma}
A position section $\mathcal{P}$ satisfy 
\begin{equation}\label{eq:PDPtop2}
\langle \overline{\mathcal{D}|\mathcal{P}|^2},\mathcal{P}\rangle=2|\mathcal{P}^{\top}|^2. 
\end{equation}
\end{lemma}

\begin{proof}
Since 
\begin{equation}\label{psquare}
|\mathcal{P}|^2=4g_{\bar{p}q}(x^{p}-x_{0}^{p})(x^{q}-x_{0}^{q})+g^{p\bar{q}}\nabla_{p}\phi \nabla_{\bar{q}}\phi, 
\end{equation}
we have 
\begin{equation}\label{dpsruare}
\begin{aligned}
\nabla_{i}|\mathcal{P}|^2
=&2g_{\bar{i}p}(x^{p}-x_{0}^{p})+2g_{\bar{p}i}(x^{p}-x_{0}^{p})+\nabla_{i}(\nabla_{q}\phi\nabla_{\bar{p}}\phi g^{q\bar{p}})\\
=&2g_{\bar{i}p}(x^{p}-x_{0}^{p})+2g_{\bar{p}i}(x^{p}-x_{0}^{p})+\nabla_{i}\nabla_{q}\phi\nabla_{\bar{p}}\phi g^{q\bar{p}}+\nabla_{q}\phi F_{\bar{p}i} g^{q\bar{p}}\\
=&2g_{\bar{i}p}(x^{p}-x_{0}^{p})+2g_{\bar{p}i}(x^{p}-x_{0}^{p})+F_{\bar{q}i}\nabla_{\bar{p}}\phi g^{q\bar{p}}+\nabla_{q}\phi F_{\bar{p}i} g^{q\bar{p}}\\
=&2\left(2g_{\bar{p}i}(x^{p}-x_{0}^{p})+\nabla_{q}\phi F_{\bar{p}i} g^{q\bar{p}}\right), 
\end{aligned}
\end{equation}
where we used the condition \eqref{lsfc} and \eqref{lsfc2} several times. 
Thus, 
\[
\begin{aligned}
\langle\overline{\mathcal{D}|\mathcal{P}|^2},\mathcal{P}\rangle
=&2\left(2g_{\bar{p}i}(x^{p}-x_{0}^{p})+\nabla_{q}\phi F_{\bar{p}i} g^{q\bar{p}}\right)2(x^{k}-x_{0}^{k})\eta^{i\bar{j}}g_{\bar{j}k}\\
&+2\left(2g_{\bar{p}i}(x^{p}-x_{0}^{p})+\nabla_{q}\phi F_{\bar{p}i} g^{q\bar{p}}\right)\nabla_{\bar{k}}\phi\eta^{i\bar{j}}F_{\bar{j}\ell}g^{\ell\bar{k}}\\
=&2\left(2g_{\bar{p}i}(x^{p}-x_{0}^{p})+\nabla_{q}\phi F_{\bar{p}i} g^{q\bar{p}}\right)\eta^{i\bar{j}}
\left(2(x^{k}-x_{0}^{k})g_{\bar{j}k}+\nabla_{\bar{k}}\phi F_{\bar{j}\ell}g^{\ell\bar{k}}\right). 
\end{aligned}
\]
On the other hand, one can easily see that 
\begin{equation}\label{ptopexp}
\mathcal{P}^{\top}=\eta^{i\bar{j}}\langle \overline{\mathcal{E}_{j}},\mathcal{P}\rangle \mathcal{E}_{i}
=\eta^{i\bar{j}}\left(2(x^{k}-x_{0}^{k})g_{\bar{j}k}+\nabla_{\bar{k}}\phi F_{\bar{j}\ell}g^{\ell\bar{k}}\right)\mathcal{E}_{i} 
\end{equation}
by \eqref{eq:used1} without $v$. Then, since $\langle \overline{\mathcal{E}_{i}},\mathcal{E}_{\ell}\rangle=\eta_{\bar{i}\ell}$, 
the desired identity holds. 
\end{proof}

The following is the application of the divergence formula \eqref{eq:divthm1}. 

\begin{theorem}\label{corofdivthm}
Assume that $(X,g)$ is semi-flat on $U\subset X$ with the coordinates $(z^{1},\dots,z^{n})$ 
induced from $\varphi:B(r)\times B(r')\to U$ and $(L,h)$ is graphical with respect to a section $e\in\Gamma(U,L)$. 
Fix $p\in U_{r/4}$ and let $\mathcal{P}:=\mathcal{P}_{p}$ be the position 
section of $h$ centered at $p$. 
Then, for any smooth function $f:U\to\mathbb{R}$ with 
\begin{equation}\label{condiforcuttoff}
(\mathrm{a})\quad\frac{\partial f}{\partial y^{k}}=0 \quad \mbox{and} \quad(\mathrm{b})\quad \mathop{\mathrm{supp}}f(\,\cdot\,,0)\Subset B(r) 
\end{equation}
and a constant $\alpha\in\mathbb{R}$, it holds that 
\[
\int_{U} \left(n+ \langle \overline{\mathcal{H}}, \mathcal{P}\rangle +2\alpha\left|\mathcal{P}^{\top}\right|^2\right) f\varphi d\mu(h)
=-\int_{U}\langle\overline{\mathcal{D}f},\mathcal{P}\rangle\varphi d\mu(h), 
\]
where $\varphi:=\exp(\alpha|\mathcal{P}|^2):U_{3r/4}\to\mathbb{R}$. 
\end{theorem}

\begin{proof}
It follows from \eqref{eq:dvhfP2} and \eqref{eq:PDPtop2} that 
\[
\begin{aligned}
\mathop{\mathrm{div}_{h}}(f\varphi\mathcal{P})
=&\langle\overline{\mathcal{D}(f\varphi}),\mathcal{P}\rangle+nf\varphi\\
=&\alpha\langle\overline{\mathcal{D}|\mathcal{P}|^2},\mathcal{P}\rangle f\varphi
+\langle\overline{\mathcal{D}f},\mathcal{P}\rangle\varphi+nf\varphi\\
=&2\alpha\left|\mathcal{P}^{\top}\right|^2 f\varphi+\langle\overline{\mathcal{D}f},\mathcal{P}\rangle\varphi+nf\varphi. 
\end{aligned}
\]
Then, by the divergence formula \eqref{eq:divthm1}, we have
\begin{equation}\label{divapp2}
\begin{aligned}
-\int_{U} \langle \overline{\mathcal{H}}, f\varphi\mathcal{P}\rangle d\mu(h)
=&\int_{U} \mathop{\mathrm{div}_{h}}(f\varphi\mathcal{P}) d\mu(h)
-\int_{\partial U}g\left(vf\varphi\widetilde{\mathcal{P}^{\top}},\nu\right)d\mu|_{\partial U}\\
=& \int_{U}\left(2\alpha\left|\mathcal{P}^{\top}\right|^2f\varphi+\langle\overline{\mathcal{D}f},\mathcal{P}\rangle\varphi+nf\varphi\right)d\mu(h)\\
&-\int_{\partial U}g\left(vf\varphi\widetilde{\mathcal{P}^{\top}},\nu\right)d\mu|_{\partial U}. 
\end{aligned}
\end{equation}
We can prove that the last term is actually $0$ as follows. 
First, $\partial U$ is the union of $(\partial B(r))\times B(r')$ and $B(r)\times (\partial B(r'))$, 
and the integral over $(\partial B(r))\times B(r')$ is $0$ by $f|_{(\partial B(r))\times B(r')}\equiv 0$. 
Next, it is easy to see that the integral over $B(r)\times (\partial B(r'))$ is pure imaginary 
since $\nu$ is written as $\nu=\nu^{i}(\partial/\partial y^{i})$ (for some $\nu^{i}\in\mathbb{R}$) and \eqref{gxiyj0}. 
On the other hand, one can easily prove that $\langle \overline{\mathcal{H}}, f\varphi\mathcal{P}\rangle$ 
and $\mathop{\mathrm{div}_{h}}(f\varphi\mathcal{P})$ in \eqref{divapp2} are real valued functions by assumptions. 
Then, by the first equality of $\eqref{divapp2}$, the last term of it should be $0$. 
This gives the desired equality. 
\end{proof}
\section{Monotonicity formula}\label{sec:monoform}
In this section, we give a monotonicity formula and density for line bundle mean curvature flows. 
This is an analog given by Huisken \cite{Huisken} for mean curvature flows. 
The proof of our monotonicity formula based on Theorem \ref{corofdivthm}. 

As in the previous sections, 
let $(X,g)$ be a K\"ahler manifold with $\dim_{\mathbb{C}}X=n$ and let $\pi:L\to X$ be a holomorphic line bundle. 
Assume that $h=\{\,h_{t}\,\}_{t\in[0,T)}$ is a line bundle mean curvature flow of $L$. 
We further assume that $(X,g)$ is semi-flat on $U\subset X$ with the coordinates $(z^{1},\dots,z^{n})$ 
induced from $\varphi:B(r)\times B(r')\to U$ and $(L,h)$ is graphical with respect to a section $e\in\Gamma(U,L)$. 
Fix $T'\in (0,T)$ and a smooth function $f:U\times[0,T') \to\mathbb{R}$ so that $f(\,\bullet\,,t)$ satisfies (a) and (b) of \eqref{condiforcuttoff} for each $t$. 
Let $\psi:U\times [0,T')\to\mathbb{R}$ be a smooth function so that $\psi(\,\bullet\,,t)$ satisfies (a) of \eqref{condiforcuttoff} for each $t$. 
For each $k\in\mathbb{R}$, define 
\[\varphi:=\frac{1}{(4\pi (T'-t))^{k}}\exp\left(-\frac{\psi(t)}{4(T'-t)}\right)\]
and
\[
\Theta_{\psi,f,k}(h,T',t):=\int_{U}\varphi fd\mu(h). 
\]

\begin{proposition}
It holds that 
\begin{equation}\label{premonoform1}
\begin{aligned}
\frac{d}{dt}\Theta_{\psi,f,k}(h,T',t)
=&\int_{U}(\mathcal{L}_{\eta}\psi)f\varphi d\mu(h)
+\int_{U}\left(\frac{\partial}{\partial\tau}f-\Delta_{\eta}f\right)\varphi d\mu(h)\\
&+\frac{1}{2(T'-t)}\int_{U}\mathop{\mathrm{Re}}\left(\langle\overline{\partial\psi},\partial f\rangle_{\eta}\right)\varphi d\mu(h), 
\end{aligned}
\end{equation}
where
\begin{equation}\label{ELL}
\begin{aligned}
\mathcal{L}_{\eta}\psi:=&\frac{1}{4(T'-t)}\left(-\partial_{t}\psi+\Delta_{\eta}\psi-\frac{\psi}{4(T'-t)}+4k\right)\\
&-|H^{(1,0)}|^2_{\eta} -\frac{|\partial\psi|^2_{\eta}}{(4(T'-t))^2}. 
\end{aligned}
\end{equation}
\end{proposition}
\begin{proof}
A straightforward calculation gives 
\[
\begin{aligned}
\frac{d}{dt}\Theta_{\psi,f,k}(h,T',t)
=&\int_{U}\left(\frac{d}{dt}\frac{1}{(4\pi (T'-t))^{k}}\right)\exp\left(-\frac{\psi}{4(T'-t)}\right)fd\mu(h)\\
&+\int_{U}\frac{1}{(4\pi (T'-t))^{k}}\left(\frac{\partial}{\partial t}\exp\left(-\frac{\psi(t)}{4(T'-t)}\right)\right)fd\mu(h)\\
&+\int_{U}\frac{1}{(4\pi (T'-t))^{k}}\exp\left(-\frac{\psi}{4(T'-t)}\right)f\frac{\partial}{\partial t}\left(d\mu(h(t))\right)\\
&+\int_{U}\frac{1}{(4\pi (T'-t))^{k}}\exp\left(-\frac{\psi}{4(T'-t)}\right)\frac{\partial}{\partial t}f(t)d\mu(h)\\
=:&I_{1}+I_{2}+I_{3}+I_{4}. 
\end{aligned}
\]
It is easy to see that 
\[
I_{1}=\int_{U}\frac{k}{T'-t}f\varphi d\mu(h)\quad\mbox{and}\quad
I_{2}=\int_{U}\left(-\frac{\psi}{4(T'-t)^2}
-\frac{\partial_{t}\psi}{4(T'-t)}\right)f\varphi d\mu(h). 
\]
To calculate $I_{3}$, we need to use 
\[
\frac{\partial}{\partial t}\left(d\mu(h(t))\right)=\left(\frac{\partial}{\partial t}|\zeta(h(t))|\right)\frac{\omega^n}{n!}
=\left(-\eta^{\ell\bar{q}}H_{\ell}H_{\bar{q}}|\zeta |
+\nabla_{j}\left( \eta^{j\bar{k}}F_{\bar{k}\ell}g^{\ell\bar{q}}H_{\bar{q}}|\zeta| \right)\right)\frac{\omega^n}{n!}, 
\]
where the second equality follows from the equation (3.7) in \cite{JacobYau} 
and $\nabla_{\bar{q}}\dot{\phi}=H_{\bar{q}}$. 
Taking the complex conjugate of both hand side of \eqref{used2} gives 
\[
v^{-1}\nabla_{\bar{i}}\left(v\eta^{j\bar{i}}\right)=-\eta^{j\bar{k}}g^{\ell\bar{q}}H_{\bar{q}}F_{\bar{k}\ell}. 
\]
Combining these two equations gives 
\[
\frac{\partial}{\partial t}\left(d\mu(h(t))\right)
=-|H^{(1,0)}|^2_{\eta}d\mu(h(t))
-\nabla_{j}\nabla_{\bar{i}}\left(v\eta^{j\bar{i}}\right)\frac{\omega^{n}}{n!}. 
\]
Thus, 
\[
I_{3}=-\int_{X}|H^{(1,0)}|^2_{\eta}f\varphi d\mu(h)-I_{5}, 
\]
where
\[
I_{5}:=\int_{U}\frac{1}{(4\pi (T'-t))^{k}}\exp\left(-\frac{\psi}{4(T'-t)}\right)f\nabla_{j}\nabla_{\bar{i}}\left(v\eta^{j\bar{i}}\right)\frac{\omega^n}{n!}. 
\]
By using the divergence theorem twice, with the similar argument as in the last part of the proof of Theorem \ref{corofdivthm} which ensures 
the boundary contribution is $0$, 
one can verify that 

\[
\begin{aligned}
I_{5}
=&\int_{X}\frac{\eta^{j\bar{i}}\nabla_{\bar{i}}\psi\nabla_{j}\psi}{(4(T'-t))^2}f\varphi d\mu(h)
-\int_{X}\frac{\Delta_{\eta}\psi}{4(T'-t)}f\varphi d\mu(h)+\int_{X}\Delta_{\eta}f\varphi d\mu(h)\\
&-\int_{X}\frac{\nabla_{j}\psi}{4(T'-t)}\nabla_{\bar{i}}f\eta^{j\bar{i}}\varphi d\mu(h)
-\int_{X}\frac{\nabla_{\bar{i}}\psi}{4(T'-t)}\nabla_{j}f\eta^{j\bar{i}}\varphi d\mu(h). 
\end{aligned}
\]
Combining all above calculations together gives the desired formula. 
\end{proof}

Fix $Q:=(p,T')\in U_{r/4}\times (0,T)$. 
Let $\mathcal{P}_{p}(t)$ be the position section 
of $h_{t}$ centered at $p$. 
We denote $\Theta_{|\mathcal{P}_{p}|^2,f,n/2}(h,T',t)$ by $\Theta_{f}(h,Q,t)$ simply, that is, 
\[\Theta_{f}(h,Q,t):=\int_{U}\frac{1}{(4\pi (T'-t))^{n/2}}\exp\left(-\frac{|\mathcal{P}_{p}(t)|^2}{4(T'-t)}\right)f(t)d\mu(h(t)).\]
We basically omit the subscript $p$ of $\mathcal{P}_{p}(t)$. 

\begin{theorem}\label{premono2}
It holds that 
\[
\begin{aligned}
&\frac{d}{dt}\Theta_{f}(h,Q,t)\\
=&-\int_{U}\left|\mathcal{H}+\frac{\mathcal{P}^{\bot}}{2(T'-t)}\right|^2f\varphi d\mu(h)
+\int_{U}\left(\frac{\partial}{\partial t}f-\Delta_{\eta}f\right)\varphi d\mu(h)
\end{aligned}
\]
where
\[\varphi:=\frac{1}{(4\pi (T'-t))^{n/2}}\exp\left(-\frac{|\mathcal{P}(t)|^2}{4(T'-t)}\right). \]
\end{theorem}

\begin{proof}
We will calculate $\mathcal{L}_{\eta}|\mathcal{P}|^2$ first, 
see \eqref{ELL} for the definition of $\mathcal{L}_{\eta}$. 
By \eqref{psquare}, we have 
\[
\partial_{t}|\mathcal{P}|^2=g^{p\bar{q}}H_{p} \nabla_{\bar{q}}\phi+g^{p\bar{q}}\nabla_{p}\phi H_{\bar{q}}=2g^{p\bar{q}}\nabla_{p}\phi H_{\bar{q}}, 
\]
where we used the semi-flat condition. 
From \eqref{dpsruare} and \eqref{ptopexp}, it follows that 
\begin{equation}\label{used4}
\left| \partial|\mathcal{P}|^2\right|^2_{\eta}=4|\mathcal{P}^{\top}|^2. 
\end{equation}
Differentiating \eqref{dpsruare} gives
\[
\begin{aligned}
\nabla_{\bar{j}}\nabla_{i}|\mathcal{P}|^2
=&2\left(g_{\bar{j}i}+F_{\bar{j}q} F_{\bar{p}i} g^{q\bar{p}}+\nabla_{q}\phi \nabla_{\bar{j}}F_{\bar{p}i} g^{q\bar{p}}\right)\\
=&2\left(\eta_{\bar{j}i}+\nabla_{q}\phi \nabla_{\bar{j}}F_{\bar{p}i} g^{q\bar{p}}\right). 
\end{aligned}
\]
This yields 
\[\Delta_{\eta}|\mathcal{P}|^2=\eta^{i\bar{j}}\nabla_{\bar{j}}\nabla_{i}|\mathcal{P}|^2=2n+2\nabla_{q}\phi H_{\bar{p}} g^{q\bar{p}}. \]
Combining the above formulas and \eqref{normofHH} implies 
\[
\begin{aligned}
\mathcal{L}_{\eta}|\mathcal{P}|^2=&\frac{n}{2(T'-t)}-\frac{|\mathcal{P}|^2}{4(T'-t)^2}-\frac{2g^{p\bar{q}}\nabla_{p}\phi H_{\bar{q}}}{4(T'-t)}\\
&-|\mathcal{H}|^2-\frac{|\mathcal{P}^{\top}|^2}{4(T'-t)^2}+\frac{2n+2\nabla_{q}\phi H_{\bar{p}} g^{q\bar{p}}}{4(T'-t)}\\
=&-\left|\mathcal{H}+\frac{\mathcal{P}^{\bot}}{2(T'-t)}\right|^2+\frac{n}{T'-t}
+\frac{\langle\overline{\mathcal{H}},\mathcal{P}\rangle}{T'-t}-\frac{2|\mathcal{P}^{\top}|^2}{4(T'-t)^2}, 
\end{aligned}
\]
where we used the fact that $\mathcal{H}$ is normal. 
Thus, 
\[
\begin{aligned}
\int_{U}\mathcal{L}_{\eta}|\mathcal{P}|^2f\varphi d\mu(h)=&-\int_{U}\left|\mathcal{H}+\frac{\mathcal{P}^{\bot}}{2(T'-t)}\right|^2f\varphi d\mu(h)\\
&+\frac{1}{T'-t}\int_{U}\left(n+\langle\overline{\mathcal{H}},\mathcal{P}\rangle-\frac{2|\mathcal{P}^{\top}|^2}{4(T'-t)}\right)f\varphi d\mu(h). 
\end{aligned}
\]
Applying Theorem \ref{corofdivthm} with $\alpha=-1/(4(T'-t))$ yields 
\[
\frac{1}{T'-t}\int_{U}\left(n+\langle\overline{\mathcal{H}},\mathcal{P}\rangle-\frac{2|\mathcal{P}^{\top}|^2}{4(T'-t)}\right)f\varphi d\mu(h)
=-\frac{1}{T'-t}\int_{U}\langle\overline{\mathcal{D}f},\mathcal{P}\rangle\varphi d\mu(h). 
\]
Moreover, by a partial consequence of \eqref{used3}, we have
\[
\langle\overline{\partial|\mathcal{P}|^2},\partial f\rangle_{\eta}
=2\left(2g_{\bar{i}p}(x^{p}-x_{0}^{p})+\nabla_{\bar{q}}\phi F_{\bar{i}p} g^{p\bar{q}}\right)\nabla_{j}f\eta^{j\bar{i}}
=2\langle \overline{\mathcal{D}f},\mathcal{P}\rangle. 
\]
Thus, 
$
\langle\overline{\partial|\mathcal{P}|^2},\partial f\rangle_{\eta}
=2\langle \overline{\mathcal{D}f},\mathcal{P}\rangle. 
$
Then, substituting the above formulas into \eqref{premonoform1} gives 
the desired formula. 
\end{proof}

As an application of Theorem \ref{premono2}, we get a monotonicity formula. 
Assume that $Q=(p,T')\in U_{r/4}\times(0,T)$ is given. 
Let $\tilde{f}:\mathbb{R}\to [0,1]$ be a smooth cut-off function 
which is strictly decreasing on the interval $[1,2]$ satisfying 
\[
\tilde{f}(x)=\begin{cases}1 \quad\mbox{if}\quad x\in(-\infty,1]\\0 \quad\mbox{if}\quad x\in[2,\infty)\end{cases}
\quad\mbox{and}\quad
|\tilde{f}'|+|\tilde{f}''|\leq C'
\]
for some constant $C'>0$. 
Let $\lambda=\lambda(g)>0$ be the square root of the minimum of the lowest eigenvalue of $(g_{\bar{i}j})$ on the closure of $U$. 
Define $f:U_{3r/4}\times[0,T')\to \mathbb{R}$ by 
\[f(z,t):=\tilde{f}\left(\frac{4|\mathcal{P}_{p}(z,t)|}{\lambda r}\right). \]
Note that $f((x,y),t)$ is $y$-invariant and the support of $f((\,\cdot\,,0),t)$ is contained in $B(r/2)$ for each $t\in [0,T')$. 
Actually, by \eqref{psquare}, we have 
\begin{equation}\label{lowbdofP}
    |\mathcal{P}(z,t)|\geq 2\lambda|x-x_{0}|\geq 2\lambda (|x|-|x_{0}|).
\end{equation}
This yields that if $|x|\geq r/2$ then $f=0$. 
Thus, $f(\,\bullet\,,t)$ satisfies (a) and (b) of \eqref{condiforcuttoff} for each $t$. 

We denote $\Theta_{f}(h,Q,t)$ by $\Theta(h,Q,t)$ simply, that is, 
\[
\Theta(h,Q,t):=\int_{U}\frac{1}{(4\pi (T'-t))^{n/2}}\exp\left(-\frac{|\mathcal{P}_{p}(t)|^2}{4(T'-t)}\right)
\tilde{f}\left(\frac{4|\mathcal{P}_{p}(t)|}{\lambda r}\right)d\mu(h(t)).
\]
\begin{theorem}\label{mono}
If $X$ is closed, then there exists a constant $C>0$ such that 
\begin{equation}\label{monofrom2}
\frac{d}{dt}\Theta(h,Q,t)
\leq -\int_{X}\left|\mathcal{H}+\frac{\mathcal{P}^{\bot}}{2(T'-t)}\right|^2f\varphi d\mu(h)+C, 
\end{equation}
where
\[\varphi:=\frac{1}{(4\pi (T'-t))^{n/2}}\exp\left(-\frac{|\mathcal{P}(t)|^2}{4(T'-t)}\right). \]
The constant $C$ is given by $C=C'C''(n)V(h(0))\lambda^{-(n+2)}r^{-(n+2)}$, 
where $V(h(0))$ is the volume of $h(0)$ and $C''(n)>0$ is a constant which depends only on $n$. 
\end{theorem}

\begin{proof}
Put $Y:=4|\mathcal{P}(t)|/\lambda r$ for short. 
Then, we have 
\[
\frac{\partial}{\partial t}f
=\tilde{f}'(Y)\frac{2}{\lambda r|\mathcal{P}|}\frac{\partial}{\partial t}|\mathcal{P}|^2, \quad
\nabla_{\bar{j}}f=\tilde{f}'(Y)\frac{2\nabla_{\bar{j}}|\mathcal{P}|^2}{\lambda r|\mathcal{P}|}
\]
and 
\[
\nabla_{i}\nabla_{\bar{j}}f
=\tilde{f}''(Y)\frac{4\nabla_{i}|\mathcal{P}|^2\nabla_{\bar{j}}|\mathcal{P}|^2}{(\lambda r|\mathcal{P}|)^2}
-\tilde{f}'(Y)\frac{\nabla_{i}|\mathcal{P}|^2\nabla_{\bar{j}}|\mathcal{P}|^2}{\lambda r|\mathcal{P}|^3}
+\tilde{f}'(Y)\frac{2\nabla_{i}\nabla_{\bar{j}}|\mathcal{P}|^2}{\lambda r|\mathcal{P}|}. 
\]
By using $\tilde{f}'(Y)\leq 0$, $|\tilde{f}''(Y)|\leq C'$ and $\partial_{t}|\mathcal{P}|^2-\Delta_{\eta}|\mathcal{P}|^2=-2n\leq0$, we estimate 
\[
\frac{\partial}{\partial t}f-\Delta_{\eta}f
\leq C'\frac{4\left|\partial|\mathcal{P}|^2\right|^2_{\eta}}{(\lambda r|\mathcal{P}|)^2}\chi_{A(t)}, 
\]
where $\chi_{A(t)}$ is the characteristic function of a set $A(t):=\{\,z\in U\mid \, \lambda r/4\leq |\mathcal{P}(z,t)|\leq \lambda r/2\,\}$. 
By \eqref{used4}, we have
$\left|\partial|\mathcal{P}|^2\right|^2_{\eta}=4\left|\mathcal{P}^{\top}\right|^2\leq 4|\mathcal{P}|^2$. 
This yields that 
\[
\frac{\partial}{\partial t}f-\Delta_{\eta}f
\leq \frac{4^3C'}{\lambda^2r^2}\chi_{A(t)}. 
\]
Thus, we have 
\[
\begin{aligned}
\left(\frac{\partial}{\partial t}f-\Delta_{\eta}f\right)\varphi
=&\frac{1}{(4\pi (T'-t))^{n/2}}\exp\left(-\frac{|\mathcal{P}(t)|^2}{4(T'-t)}\right)\left(\frac{\partial}{\partial t}f-\Delta_{\eta}f\right)\\
\leq &\frac{4^3C'}{\lambda^2r^2} \frac{1}{(4\pi (T'-t))^{n/2}}\exp\left(-\frac{(\lambda r/4)^2}{4(T'-t)}\right)\\
=&\frac{4^3C'}{\lambda^2r^2} \frac{1}{\pi^{n/2}}\left(\frac{(\lambda r/4)^2}{4(T'-t)}\right)^{n/2}\left\{\frac{1}{(\lambda r/4)^{n}}\exp\left(-\frac{(\lambda r/4)^2}{4(T'-t)}\right)\right\}\\
\leq &C'C''(n)\frac{1}{\lambda^{n+2}r^{n+2}}=:C''', 
\end{aligned}
\]
where we put 
\[
C''(n):=\frac{4^{n+3}}{\pi^{n/2}}\max\{\,x^{n/2}\exp(-x)\mid x\geq 0\,\}. 
\]
Thus, we have 
\[
\int_{U}\left(\frac{\partial}{\partial t}f-\Delta_{\eta}f\right)\varphi d\mu(h)
\leq C'''\int_{X} d\mu(h(t))
= C'''V(h(t))\leq C'''V(h(0))=:C, 
\]
where we used the fact that the volume is finite on the closed $X$ 
and decreasing along a line bundle mean curvature flow. 
Then, by Theorem \ref{premono2}, we have 
\[
\frac{d}{dt}\Theta(h,Q,t)
\leq -\int_{X}\left|\mathcal{H}+\frac{\mathcal{P}^{\bot}}{2(T'-t)}\right|^2f\varphi d\mu(h)+C, 
\]
and the proof is complete. 
\end{proof}

\begin{remark}
The first term on the right hand side of \eqref{monofrom2} multiplied by $-1$  
is just $B(h)$ mentioned in (iii) of Subsection \ref{subsec1.4}. 
\end{remark}

We give an application of Theorem \ref{mono}. 
Hence, assume that $X$ is closed. 
Fix a point $Q=(p,T')\in U_{r/4}\times (0,T')$. 
We define a kind of ``translation'' of $h_{t}$ as follows. 
First, let $\phi_{2}(x,T')$ be the Taylor expansion 
of $\phi(\,\cdot\,,T')$ at $x=x_{0}$ up to the first order, 
where $x_{0}$ is the $B(r/4)$-component of $p$ on $U_{r/4}$ 
via $\varphi:B(r/4)\times B(r')\to U_{r/4}$. 
Precisely, we have 
\[\phi_{2}(x,T'):=\phi(x_{0},T')+\frac{\partial \phi(x_{0},T')}{\partial x^{i}}(x^{i}-x_{0}^{i}). \]
This is a function on $U_{3r/4}$ which does not depend on $y$. 
Next, subtract $\phi_{2}(x,T')$ from $\phi(x,t)$ and denote it by 
\begin{equation}\label{tildephi}
(A_{Q}\phi)(x,t):=\phi(x,t)-\left(\phi(x_{0},T')+\sum_{i=1}^{n}\frac{\partial\phi(x_{0},T')}{\partial x^{i}}(x^{i}-x_{0}^{i})\right)
\end{equation}
and put 
\[(A_{Q}h)_{t}:=e^{-(A_{Q}\phi)(t)}(\bar{e}^{*}\otimes e). \]
Then, each $(A_{Q}h)_{t}$ is a Hermitian metric of $L$ defined 
only on $U_{3r/4}$ and is also graphical for all $t\in[0,T)$. 
Moreover, $A_{Q}h:=\{\,(A_{Q}h)_{t}\,\}_{t\in[0,T)}$ is also a line bundle mean curvature flow on $U_{3r/4}$. 
This can be easily seen as follows. 
The function $\phi(x_{0},T')+\frac{\partial\phi(x_{0},T')}{\partial x^{i}}(x^{i}-x_{0}^{i})$ does not 
depend on $t$ and the angle function $\theta$ is invariant under the first order perturbation 
since it is defined by the second derivative of $\log h$. 

Thus, we can apply Theorem \ref{mono} to the line bundle mean curvature flow $A_{Q}h$. 
Then, we can see that $\Theta(A_{Q}h,Q,t)+C(T'-t)$ is monotonically decreasing and its limit exists as $t\to T'$. 
This implies the existence of the limit of $\Theta(A_{Q}h,Q,t)$ as $t\to T'$. 

\begin{definition}\label{gaussdensity}
For $Q=(p,T')\in U_{r/4}\times (0,T)$, we define
\begin{equation*}
\begin{aligned}
\bar{\Theta}(h,Q,t):=&\frac{(2\sqrt{2})^{n}}{\mathop{\mathrm{Vol}_{g}}(B(r')_{p})}\Theta(A_{Q}h,Q,t), \\
\bar{\Theta}(h,Q):=&\lim_{t\to T'}\bar{\Theta}(h,Q,t), 
\end{aligned}
\end{equation*}
and call $\bar{\Theta}(h,Q,t)$ the {\it Gaussian density} of $h$ at $Q=(p,T')$ with scale $t$ and $\bar{\Theta}(h,Q)$ the {\it Gaussian density} of $h$ at $Q=(p,T')$, where $B(r')_{p}:=\varphi^{-1}(\{\,x_{0}\,\}\times B(r'))\subset U$ 
and the volume of $B(r')_{p}$ is measured by $g$. 
\end{definition}

In what follows, we prove that $\bar{\Theta}(h,Q)\geq 1$. 
Put $\tilde{\phi}:=(A_{Q}\phi)(t)$ and $\tilde{h}_{t}:=(A_{Q}h)_{t}$ for short. 
Recall that in Definition \ref {defofscaling} for $T''\in\mathbb{R}$ and $k\in\mathbb{N}_{>0}$ a scaling of $h$ is defined by 
$(D_{k}^{T''}\tilde{h})_{s}:=\tilde{h}_{t}^{\otimes k}$
with $t=T''+s/k$. 
Put $f:=e^{\otimes k}$. Then, we have 
\[-\log\Bigl((D_{k}^{T''}\tilde{h})_{s}(\bar{f},f)\Bigr)=k(A_{Q}\phi)(T''+s/k). \]
Since the 0-th and first derivative at $(p,k(T'-T''))=:Q'$ of the right hand side 
with respect to $x$ are zero, we see that 
\[
A_{Q'}\left(-\log\Bigl((D_{k}^{T''}\tilde{h})_{s}(\bar{f},f)\Bigr)\right)=k(A_{Q}\phi)(T''+s/k). 
\]
Thus, for given $k\in\mathbb{N}$, it is clear that 
\[
\Theta(D_{k}^{T''}\tilde{h},Q',k(t-T''))=\sqrt{k}^{n}\Theta(\tilde{h},Q,t), 
\]
where the left (resp. right) hand side is calculated with respect to $kg$ (resp. $g$). 
On the other hand, we have 
\begin{equation}\label{voloftorus}
\mathrm{Vol}_{kg}(B(r')_{p})=\sqrt{k}^{n}\mathrm{Vol}_{g}(B(r')_{p}). 
\end{equation}
Thus, we have proved the following.  
\begin{lemma}
We have 
\begin{align}\label{densityinv2}
\frac{(2\sqrt{2})^{n}}{\mathrm{Vol}_{kg}(B(r')_{p})}\Theta(D_{k}^{T''}\tilde{h},Q',k(t-T''))
=\frac{(2\sqrt{2})^{n}}{\mathrm{Vol}_{g}(B(r')_{p})}\Theta(\tilde{h},Q,t).
\end{align}
\end{lemma}
Putting $T'':=T'$ and $t:=T'-1/k$ in this formula gives 
\begin{align}\label{densityinv}
\frac{(2\sqrt{2})^{n}}{\mathrm{Vol}_{kg}(B(r')_{p})}\Theta(D_{k}^{T'}\tilde{h},Q',-1)
=\frac{(2\sqrt{2})^{n}}{\mathrm{Vol}_{g}(B(r')_{p})}\Theta(\tilde{h},Q,T'-1/k).
\end{align}
\begin{lemma}
We have 
\begin{equation}\label{densitygeq1}
1\leq \lim_{k\to\infty}\frac{(2\sqrt{2})^{n}}{\mathrm{Vol}_{kg}(B(r')_{p})}\Theta(D_{k}^{T'}\tilde{h},Q',-1), 
\end{equation}
where $\tilde{h}:=A_{Q}h$. 
\end{lemma}

\begin{proof}
Note that we also rescale the K\"ahler metric on $X$ as $kg$ implicitly when we use the rescaled flow $D_{k}^{T'}\tilde{h}$. 
We will see how each quantity in the definition of $\Theta$ changes by this rescaling procedure. 
It's easy to see that $\lambda(kg)=\sqrt{k}\lambda(g)$. 
By \eqref{psquare}, we can see that $|\mathcal{P}((D_{k}^{T'}\tilde{h})(-1))|^2=k|\mathcal{P}(\tilde{h}(T'-1/k))|^2$. 
By Proposition \ref{scinv}, 
\[
\begin{aligned}
d\mu((D_{k}^{T'}\tilde{h})(-1))=&|\zeta(k\omega,(D_{k}^{T'}\tilde{h})(-1))|(k\omega)^{n}/n!\\
=&k^{n}|\zeta(\omega,\tilde{h}(T'-1/k))|\omega^{n}/n!=k^{n}d\mu(\tilde{h}(T'-1/k)). 
\end{aligned}
\]
Substituting these into the definition of $\Theta(D_{k}^{T'}\tilde{h},Q',-1)$, we have 
\[
\begin{aligned}
\Theta(D_{k}^{T'}\tilde{h},Q',-1)=&\int_{U}\frac{k^{n}}{(4\pi)^{n/2}}\exp\left(-\frac{k|\mathcal{P}_{x_{0}}(\tilde{h}(t_{k}))|^2}{4}\right)\\
&\hspace{10mm} \times\tilde{f}\left(\frac{|\mathcal{P}_{x_{0}}(\tilde{h}(t_{k}))|}{\lambda(g) r}\right)d\mu(\tilde{h}(t_{k}))
\end{aligned}
\]
where $t_{k}:=T'-1/k$. 
Dividing the both hand side by $(2\sqrt{2})^{-n}\mathrm{Vol}_{kg}(B(r')_{p})$ noting \eqref{voloftorus} implies that 
\begin{equation}\label{thetaeq1}
\begin{aligned}
&\frac{(2\sqrt{2})^{n}}{\mathrm{Vol}_{kg}(B(r')_{p})}\Theta(D_{k}^{T'}\tilde{h},Q',-1)\\
=&\frac{(2\sqrt{2})^{n}}{\mathrm{Vol}_{g}(B(r')_{p})}\int_{X}\frac{\sqrt{k}^{n}}{(4\pi)^{n/2}}
\exp\left(-\frac{k|\mathcal{P}_{x_{0}}(\tilde{h}(t_{k}))|^2}{4}\right)\\
&\hspace{25mm} \times \tilde{f}\left(\frac{|\mathcal{P}_{x_{0}}(\tilde{h}(t_{k}))|}{\lambda(g) r}\right)|\zeta(\omega,\tilde{h}(t_{k}))|\frac{\omega^{n}}{n!}
\end{aligned}
\end{equation}
with 
\[\frac{\omega^{n}}{n!}=(-1)^{n(n-1)/2}\det(g_{\bar{i}j})
dx^{1}\wedge\dots\wedge dx^{n}\wedge dy^{1}\wedge\dots\wedge dy^{n}. \]
Thus, by \eqref{psquare}, we have 
\begin{equation*}
|\mathcal{P}_{x_{0}}(\tilde{h}(t_{k}))|^2(x)
=4g_{\bar{p}q}(x)(x^{p}-x_{0}^{p})(x^{q}-x_{0}^{q})
+g^{p\bar{q}}(x)(\nabla_{p}\tilde{\phi}(t_{k}))(x)(\nabla_{\bar{q}}\tilde{\phi}(t_{k}))(x). \end{equation*}
Let $\tilde{x}^{i}:=\sqrt{k} (x^{i}-x_{0}^{i})$. Then, we have 
\begin{equation}\label{psquarek}
\begin{aligned}
|\mathcal{P}(\tilde{h}(t_{k}))|^2(x)
=&4g_{\bar{p}q}\left(\frac{\tilde{x}}{\sqrt{k}}+x_{0}\right)
\frac{\tilde{x}^{p}}{\sqrt{k}}\frac{\tilde{x}^{q}}{\sqrt{k}}\\
&+g^{p\bar{q}}\left(\frac{\tilde{x}}{\sqrt{k}}+x_{0}\right)\times 
(\nabla_{p}\tilde{\phi}(t_{k}))\left(\frac{\tilde{x}}{\sqrt{k}}+x_{0}\right)\\
&\times (\nabla_{\bar{q}}\tilde{\phi}(t_{k}))\left(\frac{\tilde{x}}{\sqrt{k}}+x_{0}\right).  
\end{aligned}
\end{equation}
Since $t_{k}\to T'$ as $k\to\infty$ and $T'$ is strictly smaller than $T$, 
it follows that the right hand side of \eqref{psquarek} uniformly converges to 
\[g^{p\bar{q}}(x_{0})(\nabla_{p}\tilde{\phi}(T'))(x_{0})(\nabla_{\bar{q}}\tilde{\phi}(T'))(x_{0})\]
as functions with variables $\tilde{x}$ on each compact set in $\mathbb{R}^{n}$, 
and this value is actually zero by the definition of $\tilde{\phi}=A_{Q}\phi$, see \eqref{tildephi}. 
By \eqref{psquarek}, we have 
\begin{equation}\label{psquarekk}
\begin{aligned}
k|\mathcal{P}(\tilde{h}(t_{k}))|^2(x)
=&4g_{\bar{p}q}\left(\frac{\tilde{x}}{\sqrt{k}}+x_{0}\right)
\tilde{x}^{p}\tilde{x}^{q}\\
&+g^{p\bar{q}}\left(\frac{\tilde{x}}{\sqrt{k}}+x_{0}\right)\times 
\sqrt{k}(\nabla_{p}\tilde{\phi}(t_{k}))\left(\frac{\tilde{x}}{\sqrt{k}}+x_{0}\right)\\
&\times \sqrt{k}(\nabla_{\bar{q}}\tilde{\phi}(t_{k}))\left(\frac{\tilde{x}}{\sqrt{k}}+x_{0}\right). \end{aligned}
\end{equation}
Then, it follows that the right hand side of \eqref{psquarekk} uniformly converges to 
\begin{equation}\label{psquarekkk}
4g_{\bar{p}q}(x_{0})\tilde{x}^{p}\tilde{x}^{q}
+g^{p\bar{q}}(x_{0})\left(\frac{\partial\nabla_{p}\tilde{\phi}(T')}{\partial x^{i}}(x_{0})\tilde{x}^{i}\right)
\left(\frac{\partial\nabla_{\bar{q}}\tilde{\phi}(T')}{\partial x^{j}}(x_{0})\tilde{x}^{j}\right)
\end{equation}
as functions with variables $\tilde{x}$ on each compact set in $\mathbb{R}^{n}$. 
Since $\partial\tilde{\phi}/\partial y^{k}=0$ by assumption, we have 
$\partial(\nabla_{p}\tilde{\phi}(T'))/\partial x^{i}
=2\nabla_{\bar{i}}\nabla_{p}\tilde{\phi}(T')=2F_{\bar{i}p}(\tilde{h}_{T'})$ and 
similarly $\partial(\nabla_{\bar{q}}\tilde{\phi}(T'))/\partial x^{j}
=2F_{\bar{q}j}(\tilde{h}_{T'})$. 
Thus, \eqref{psquarekkk} is equal to 
\[4\left(g_{\bar{i}j}+g^{p\bar{q}}F_{\bar{i}p}(\tilde{h}_{T'})F_{\bar{q}j}(\tilde{h}_{T'})\right)(x_{0})\tilde{x}^{i}\tilde{x}^{j}=4\eta(\tilde{h}_{T'})_{\bar{i}j}(x_{0})\tilde{x}^{i}\tilde{x}^{j}, \]
where $\eta(\tilde{h}_{T'})$ is the induced metric of $\tilde{h}_{T'}$, see \eqref{defofeta}. 
Put $A_{ij}:=\eta(\tilde{h}_{T'})_{\bar{i}j}(x_{0})$ for notational simplicity. 

In \cite{JacobYau}, it is proved that $|\zeta|=\sqrt{\det(I+K^2)}$. 
From this fact and the definition of $K$, it follows that 
$|\zeta|=(\sqrt{\det g_{\bar{i}j}})^{-1}\sqrt{\det \eta_{\bar{i}j}}$. 
Thus, combining everything together, we see that 
the limit of the right hand side of \eqref{thetaeq1} as $k\to\infty$ 
is greater than or equal to 
\[\frac{\int_{B(r')_{p}}\sqrt{\det g_{\bar{i}j}(x_{0})}dy}{\mathrm{Vol}_{g}(B(r')_{p})}\frac{(2\sqrt{2})^{n}}{(4\pi)^{n/2}}\int_{B(N)}
\exp(-A_{ij}\tilde{x}^{i}\tilde{x}^{j})\sqrt{\det A_{ij}}d\tilde{x}\]
for all sufficiently large open ball $B(N)$ ($N\in\mathbb{N}$). 
Letting $N\to\infty$ with the standard Gaussian integral formula implies 
this converges to  
\begin{equation}\label{avalue}
\frac{\int_{B(r')_{p}}\sqrt{\det g_{\bar{i}j}(x_{0})}dy}{\mathrm{Vol}_{g}(B(r')_{p})}\frac{(2\sqrt{2})^{n}}{(4\pi)^{n/2}}\pi^{n/2}. 
\end{equation}
Finally, we see that $g(\frac{\partial}{\partial y^{i}},\frac{\partial}{\partial y^{j}})=g_{i\bar{j}}+g_{\bar{i}j}=2g_{\bar{i}j}$ by 
the semi-flat assumption. Thus, the volume form of $B(r')_{p}$ is $\sqrt{2^{n}\det g_{\bar{i}j}(x_{0})}dy$. 
Then, \eqref{avalue} is actually $1$, and the proof is complete. 
\end{proof}
Combining \eqref{densityinv} and \eqref{densitygeq1}, we see the following theorem. 
\begin{theorem}
For $Q=(p,T')\in U_{r/4}\times(0,T)$, we have 
\begin{equation}\label{limdensityg1} 
1\leq \bar{\Theta}(h,Q). 
\end{equation}
\end{theorem}

In the proof of the main theorem given in Section \ref{proofofmainthm}, 
we need an analog of Theorem \ref{mono} in the case where $X$ is noncompact. 
Thus, in what follows, we assume that $r=\infty$ in the setting mentioned just before Theorem \ref{mono}, 
that is, $U\cong \mathbb{R}^{n}\times B(r')$, 
and further assume that $\lambda=\lambda(g)\in (0,\infty)$. 
Assume that $Q=(p,T')\in U\times(0,T)$ is given. 
For $j\in\mathbb{N}$, let $\tilde{f}_{j}:\mathbb{R}\to [0,1]$ be a smooth cut-off function 
which is strictly decreasing on the interval $[j,j+1]$ satisfying 
\[
\tilde{f}_{j}(x)=
\begin{cases}
1 \quad\mbox{if}\quad x\in(-\infty,j]\\
0 \quad\mbox{if}\quad x\in[j+1,\infty)
\end{cases}
\quad\mbox{and}\quad
|\tilde{f}_{j}'|+|\tilde{f}_{j}''|\leq C'
\]
for some constant $C'>0$ which does not depend on $j$. 
Define $f_{j}:X\times[0,T')\to \mathbb{R}$ by $f_{j}(z,t):=\tilde{f}_{j}(|\mathcal{P}_{p}(z,t)|/2\lambda)$. 
Then, by \eqref{lowbdofP}, $f_{j}(\,\bullet\,,t)$ satisfies (a) and (b) of \eqref{condiforcuttoff} (with $r=\infty$) for each $t$. 
We denote $\Theta_{f_{j}}(h,Q,t)$ by $\Theta_{j}(h,Q,t)$ simply, that is, 
\[
\Theta_{j}(h,Q,t):=\int_{U}\frac{1}{(4\pi (T'-t))^{n/2}}\exp\left(-\frac{|\mathcal{P}_{p}(t)|^2}{4(T'-t)}\right)
\tilde{f}_{j}\left(\frac{|\mathcal{P}_{p}(z,t)|}{2\lambda}\right)d\mu(h(t)).
\]
\begin{theorem}
It follows that 
\begin{equation}\label{monofrom3}
\frac{d}{dt}\Theta_{j}(h,Q,t)
\leq -\int_{U}\left|\mathcal{H}+\frac{\mathcal{P}^{\bot}}{2(T'-t)}\right|^2f_{j}\varphi d\mu(h)+\frac{C'}{\lambda^2}\int_{U}\varphi \chi_{A_{j}(t)}d\mu(h), 
\end{equation}
where
\[\varphi:=\frac{1}{(4\pi (T'-t))^{n/2}}\exp\left(-\frac{|\mathcal{P}(t)|^2}{4(T'-t)}\right)\]
and $\chi_{A_{j}(t)}$ is the characteristic function of $A_{j}(t):=\{\,z\in U\mid \, 2\lambda j\leq |\mathcal{P}(z,t)|\leq 2\lambda (j+1)\,\}$. 
\end{theorem}

\begin{proof}
By a similar computation as in the proof of Theorem \ref{mono}, we can see that 
\[
\frac{\partial}{\partial t}f_{j}-\Delta_{\eta}f_{j}
\leq \frac{C'}{\lambda^2}\chi_{A_{j}(t)}. 
\]
Then, by Theorem \ref{premono2}, we have 
\[
\frac{d}{dt}\Theta_{j}(h,Q,t)
\leq -\int_{U}\left|\mathcal{H}+\frac{\mathcal{P}^{\bot}}{2(T'-t)}\right|^2f_{j}\varphi d\mu(h)+\frac{C'}{\lambda^2}\int_{U}\varphi \chi_{A_{j}(t)}d\mu(h), 
\]
and the proof is complete. 
\end{proof}

Then, if there exists $C''>0$ so that $\int_{U}\varphi d\mu(h)<C''$ for all $t\in [0,T')$, 
the second term on the right hand side of \eqref{monofrom3} is bounded from above by $C'C''/\lambda^2=:C'''$. 
Hence, $\Theta_{j}(h,Q,t)+C'''(T'-t)$ is monotonically decreasing and its limit exists as $t\to T'$. 
Moreover, putting 
\[\bar{\Theta}_{j}(h,Q,t):=\frac{(2\sqrt{2})^{n}}{\mathrm{Vol}_{g}(B(r')_{p})}\Theta_{j}(A_{Q}h,Q,t), \]
we can also prove that 
\begin{equation}\label{densitygeq2}
1\leq \lim_{t\to T'}\bar{\Theta}_{j}(h,Q,t), 
\end{equation}
whenever $T'\in (0,T)$, 
by the similar way as the proof of \eqref{limdensityg1}. 

The following corollary is used directly in the proof of 
main theorem given in Section \ref{proofofmainthm}. 
Put 
\begin{equation}\label{densitynoncomp}
\begin{aligned}
&\Theta_{\infty}(h,Q,t):=\int_{U}\frac{1}{(4\pi (T'-t))^{n/2}}\exp\left(-\frac{|\mathcal{P}(t)|^2}{4(T'-t)}\right)d\mu(h(t)), \\
&\bar{\Theta}_{\infty}(h,Q,t):=\frac{(2\sqrt{2})^{n}}{\mathrm{Vol}_{g}(B(r')_{p})}\Theta_{\infty}(A_{Q}h,Q,t).
\end{aligned}
\end{equation}
We do not knot whether $\Theta_{\infty}(h,Q,t)$ is finite or not 
since the support of the integrand is noncompact for each $t$. 

\begin{corollary}\label{impcor}
Assume $\bar{\Theta}_{\infty}(h,Q,t)\leq 1$ for all $t\in [a,T')$ for some $a<T'$. 
Further assume that $A_{Q}h=h$ for simplicity.  
Then, $h_{t}$ satisfies 
\[\mathcal{H}(h_{t})=-\frac{1}{2(T'-t)}\mathcal{P}^{\bot}(h_{t})\]
for all $t\in [a,T')$
\end{corollary}

\begin{proof}
Integrate the both hand side of \eqref{monofrom3} 
on $[a,T'-\varepsilon]$ and multiply it by  
$(2\sqrt{2})^{n}/\mathrm{Vol}_{g}(B(r')_{p})$. 
Then, letting $\varepsilon\to 0$ implies that 
\begin{equation}\label{monofrom4}
\begin{aligned}
&\frac{(2\sqrt{2})^{n}}{\mathrm{Vol}_{g}(B(r')_{p})}\int_{a}^{T'}\int_{U}\left|\mathcal{H}+\frac{\mathcal{P}^{\bot}}{2(T'-t)}\right|^2f_{j}\varphi d\mu(h)dt\\
\leq &\bar{\Theta}_{j}(h,Q,a)-\lim_{t\to T'}\bar{\Theta}_{j}(h,Q,t)+\frac{C'}{\lambda^2}\int_{a}^{T'}\left(\frac{(2\sqrt{2})^{n}}{\mathrm{Vol}_{g}(B(r')_{p})}\int_{U}\varphi \chi_{A_{j}(t)}d\mu(h)\right)dt\\
\leq & \frac{C'}{\lambda^2}\int_{a}^{T'}\left(\frac{(2\sqrt{2})^{n}}{\mathrm{Vol}_{g}(B(r')_{p})}\int_{U}\varphi \chi_{A_{j}(t)}d\mu(h)\right)dt, 
\end{aligned}
\end{equation}
where the last inequality follows from $\bar{\Theta}_{j}(h,Q,t)\leq \bar{\Theta}_{\infty}(h,Q,t)\leq 1$ and \eqref{densitygeq2}. 
For $j\geq 1$, put 
\[
a_{j}(t):=\frac{(2\sqrt{2})^{n}}{\mathrm{Vol}_{g}(B(r')_{p})}\int_{U}\varphi \chi_{A_{j}(t)}d\mu(h), 
\]
and if $j=0$ we define $a_{0}$ by putting $A_{0}:=\{\,z\in X\mid \, |\mathcal{P}(z,t)|\leq 2\lambda\,\}$. 
Then, it is easy to see that 
\[\sum_{j=0}^{\infty}a_{j}(t)=\bar{\Theta}_{\infty}(h,Q,t)\leq 1. \]
Thus, by Lebesgue's dominated convergence theorem, 
the right hand side of \eqref{monofrom4} converges to $0$. 
Moreover, also by Lebesgue's dominated convergence theorem, 
the left hand side of \eqref{monofrom4} converges to 
\[\frac{(2\sqrt{2})^{n}}{\mathrm{Vol}_{g}(B(r')_{p})}\int_{a}^{T'}\int_{U}\left|\mathcal{H}+\frac{\mathcal{P}^{\bot}}{2(T'-t)}\right|^2\varphi d\mu(h)dt. \]
Thus, we know that this value is zero and the proof is complete. 
\end{proof}

\section{On self-shrinker}\label{sec:self-shrinker}
In this section, we give the definition of self-shrinker for 
line bundle mean curvature flows and prove that 
self-shrinkers have Liouville type properties. 

We assume that $X:=\mathbb{R}^{n}\times B(r')$. 
Then, by the inclusion $X\ni (x,y)\mapsto z=x+\sqrt{-1}y$, we admit the standard complex structure on $X$. 
Assume that a K\"ahler metric $g$ on $X$ is given and 
its coefficients are constants satisfying $g_{\bar{i}j}=g_{\bar{j}i}$. 
Then, $(X,g)$ satisfies semi-flat condition globally on $\mathbb{R}^{n}$. 

\begin{definition}
Assume a Hermitian metric $h$ of the trivial line bundle over
$X=\mathbb{R}^{n}\times B(r')$ satisfies graphical condition globally on $\mathbb{R}^{n}$. 
Let $\mathcal{P}:=\mathcal{P}_{0}(h)$ be the position  section of $h$ 
centered at the origin. 
In addition, if $h$ satisfies 
\begin{equation}\label{selshri2}
\mathcal{H}=\lambda \mathcal{P}^{\bot}, 
\end{equation}
we call $h$ a self-similar solution with coefficient $\lambda$. 
Moreover if $\lambda<0$ (resp. $\lambda>0$) we call 
$h$ a self-shrinker (resp. self-expander).  
\end{definition}

\begin{proposition}
Assume that $h$ of the trivial line bundle over
$X=\mathbb{R}^{n}\times B(r')$ satisfies graphical condition globally on $\mathbb{R}^{n}$. 
Then, $h$ is a self-similar solution with coefficient $\lambda$ 
if and only if 
\begin{equation}\label{selshri4}
\theta=2\lambda\left(\phi-\phi(0)-\frac{1}{2}x^{k}\frac{\partial \phi}{\partial x^{k}}\right)+\theta(0).
\end{equation}
\end{proposition}
\begin{proof}
By \eqref{Yn1}, we have 
\[\mathcal{P}^{\bot}=\langle\overline{\mathcal{F}_{i}}, \mathcal{P}\rangle \eta^{j\bar{i}}\mathcal{F}_{j}. \]
By definition, we have 
\[
\langle\overline{\mathcal{F}_{i}}, \mathcal{P}\rangle=-2x^{k}F_{\bar{i}k}+\frac{1}{2}\frac{\partial \phi}{\partial x^{i}}. 
\]
Thus, we have 
\[
\begin{aligned}
\mathcal{P}^{\bot}=&\left\{\left(-2x^{k}F_{\bar{i}k}+\frac{1}{2}\frac{\partial \phi}{\partial x^{i}}\right)
\eta^{\bar{i}j}\left(-F_{\bar{\ell}j}g^{q\bar{\ell}}\frac{\partial}{\partial z^{q}}\right)\right\}\\
&\oplus \left\{\left(-2x^{k}F_{\bar{i}k}+\frac{1}{2}\frac{\partial \phi}{\partial x^{i}}\right)\eta^{\bar{i}j}g_{\bar{q}j}d\bar{z}^{q}\right\}.
\end{aligned}
\]
By the definition of $\mathcal{H}$, the equation \eqref{selshri2} is equivalent to 
\[
\begin{aligned}
&-g^{q\bar{k}}H_{\bar{p}}\eta^{\ell\bar{p}}F_{\bar{k}\ell}
=-\lambda \left(-2x^{k}F_{\bar{i}k}+\frac{1}{2}\frac{\partial \phi}{\partial x^{i}}\right)\eta^{j\bar{i}}F_{\bar{\ell}j}g^{q\bar{\ell}}\\
&g_{\bar{q}k}H_{\bar{\ell}}\eta^{k\bar{\ell}}
=\lambda \left(-2x^{k}F_{\bar{i}k}+\frac{1}{2}\frac{\partial \phi}{\partial x^{i}}\right)\eta^{j\bar{i}}g_{\bar{q}j}
\end{aligned}
\]
One can easily show that the second equality implies the first equality, 
and the second equality is equivalent to 
\begin{equation}\label{selshri3}
H_{\bar{i}}
=\lambda \left(-2x^{k}F_{\bar{i}k}+\nabla_{\bar{i}}\phi\right). 
\end{equation}
Moreover, one can easily see that 
\[
-2x^{k}F_{\bar{i}k}+\nabla_{\bar{i}}\phi
=2\nabla_{\bar{i}}\left(\phi-x^{k}\nabla_{k}\phi\right). 
\]
Then, by $H_{\bar{i}}=\nabla_{\bar{i}}\theta$, we have 
\[\theta=2\lambda\left(\phi-\phi(0)-\frac{1}{2}x^{k}\frac{\partial \phi}{\partial x^{k}}\right)+\theta(0), \]
and the proof is complete. 
\end{proof}

The following theorem can be considered as a kind of Liouville type theorem. 
In general, it claims that solutions of some PDE are special. 

\begin{theorem}\label{tosayquadra}
Assume that $h=\{\,h_{t}\,\}_{t\in\mathbb{R}}$ satisfies graphical condition for all time $t\in\mathbb{R}$ and 
the line bundle mean curvature flow equation on 
$X=\mathbb{R}^{n}\times B(r')$, 
that is, $\partial_{t}\phi=\theta-\hat{\theta}$ 
for some $\hat{\theta}\in\mathbb{R}$. 
Let $\mathcal{P}$ be the position section of $h_{t}$ centered at the origin. 
Furthermore, assume that each $h_{t}$ with $t\in(-\infty,0)$ is a self-shimilar solution 
with coefficient $t/2$, that is, it satisfies 
\begin{equation}\label{timedepself}
    \mathcal{H}=\frac{1}{2t}\mathcal{P}^{\bot}
\end{equation}
for all $t\in(-\infty,0)$. 
Then, $-\log h_{t}=b+a_{ij}x^{i}x^{j}$ for some $b\in \mathbb{R}$ and a symmetric matrix $A=(a_{ij})\in M(n,\mathbb{R})$. 
\end{theorem}

\begin{proof}
Fix $i,j\in\{\,1,\dots,n\,\}$. 
Put $\phi(\,\cdot\,,t):=-\log h_{t}$. 
We remark that $y$-variable in the first component of $\phi$ can be omitted since $h$ is graphical. 
By \eqref{selshri3}, we have 
\begin{equation}\label{selshri5}
H_{\bar{i}}
=\frac{1}{2t} \left(-2x^{k}F_{\bar{i}k}+\nabla_{\bar{i}}\phi\right)
=\frac{1}{4t} \left(-x^{k}\frac{\partial^2 \phi}{\partial x^{k}\partial x^{i}}+\frac{\partial \phi}{\partial x^{i}}\right). 
\end{equation}
Since $\phi$ satisfies the line bundle mean curvature equation, we have 
$\frac{\partial}{\partial t}\frac{\partial \phi}{\partial x^{i}}=2H_{\bar{i}}$. 
Then, combining \eqref{selshri5} yields that 
\[
\frac{\partial}{\partial t}\frac{\partial \phi}{\partial x^{i}}=\frac{1}{2t} \left(-x^{k}\frac{\partial^2 \phi}{\partial x^{k}\partial x^{i}}+\frac{\partial \phi}{\partial x^{i}}\right). 
\]
Taking one more derivative of the both hand side implies 
\begin{equation}\label{selshri7}
\frac{\partial}{\partial t}\frac{\partial^2 \phi}{\partial x^{i}\partial x^{j}}=-\frac{1}{2t} x^{k}\frac{\partial^3 \phi}{\partial x^{k}\partial x^{i}\partial x^{j}}. 
\end{equation}
Put $\psi(x,t):=\frac{\partial^2 \phi}{\partial x^{i}\partial x^{j}}(x,t)$.  Then, \eqref{selshri7} is rewritten as 
\begin{equation}\label{selshri8}
\frac{\partial}{\partial t}\psi=-\frac{1}{2t} x^{k}\frac{\partial \psi}{\partial x^{k}}. 
\end{equation}
Fix $x\in \mathbb{R}^{n}$ and put $f_{x}(t):=\psi(\sqrt{-t}x,t)$ for all $t\in(-\infty,0)$. 
Then, for $t\in(-\infty,0)$, we have 
\[
\begin{aligned}
\frac{d}{dt}f_{x}(t)=&\frac{\partial \psi}{\partial x^{k}}(\sqrt{-t}x,t)\frac{-x^{k}}{2\sqrt{-t}}+\frac{\partial \psi}{\partial t}(\sqrt{-t}x,t)\\
=&\frac{1}{2t}(\sqrt{-t}x^{k})\frac{\partial \psi}{\partial x^{k}}(\sqrt{-t}x,t)+\frac{\partial \psi}{\partial t}(\sqrt{-t}x,t)=0, 
\end{aligned}
\]
where we used \eqref{selshri8} at the last equality. 
This means that $f_{x}$ is constant on $(-\infty,0)$. 
By the assumption, $f_x(t)$ is continuous up to $t=0$. 
Thus, for any $t\in (-\infty,0)$, we have 
\[\frac{\partial^2 \phi}{\partial x^{i}\partial x^{j}}(x,t)=\psi(\sqrt{-t}y,t)=f_{y}(t)=f_{y}(0)=\psi(0,0)=\frac{\partial^2 \phi}{\partial x^{i}\partial x^{j}}(0,0)=:a_{ij}, \]
where $y:=x/\sqrt{-t}$, and the right hand side does not depend on $x$ and $t$. 
Thus, we have proved that $\phi(x,t)$ is a quadratic function with respect to $x$-variables for every $t\in(-\infty,0]$ since $\phi(x,t)$ is smooth up to $t=0$. 
This implies that the angle function $\theta$ of $h_{t}$ is constant on $\mathbb{R}^{n}\times B(r')$ since the angle function is determined by the second derivatives of $\phi$. 
Then, by $\partial_{t}\phi=\theta-\hat{\theta}$, we see that $\phi$ is a constant with respect to $t$. 
By \eqref{selshri4}, we get for each $t\in (-\infty,0)$
\[\phi(x,t)-\phi(0,t)-\frac{1}{2}x^{k}\frac{\partial\phi}{\partial x^{k}}(x,t)=0\]
on $X=\mathbb{R}^{n}\times B(r')$. 
Substituting $\phi(x,t)=\phi(0,t)+c_{i}(t)x^{i}+b_{ij}x^{i}x^{j}$ and $\phi(0,t)=\phi(0,0)$ into the above PDE implies $c_{i}(t)=0$. 
Then, the proof is complete. 
\end{proof}

\begin{remark}
If $h_{t}$ satisfies \eqref{timedepself}, then the first term 
of the right hand side of \eqref{monofrom2} vanishes. 
This is similar to relations between self-shrinkers
of mean curvature flows and Huisken's monotonicity formula \cite{Huisken}, 
or between shrinking Ricci solitons of Ricci flows and 
Perelman's $\mathcal{W}$-entropy formula \cite{Perelman}. 
\end{remark}
\section{$\varepsilon$-regularity theorem}\label{proofofmainthm}
In this section we give the precise definition of $K_{3,\alpha}$-quantity 
and prove Theorem \ref{epregthm}, the $\varepsilon$-regularity theorem. 

As in the previous sections, 
let $(X,g)$ be a K\"ahler manifold with $\dim_{\mathbb{C}}X=n$ and let $\pi:L\to X$ be a holomorphic line bundle. 
Let $U\subset X$ be an open set and $[a,b)$ be an semi-open interval. 
Put $V:=U\times[a,b)$. 
In Subsection \ref{subsecmain}, 
we defined the parabolic distance from $Q=(p,t)\in V$ to 
the boundary of $V$, denoted by $\mathrm{dist}_{g}(Q,V)$, 
see \eqref{distance}. 
Moreover, to define the $K_{3,\alpha}$-quantity, 
we need to use the parabolic distance between $Q$ and 
$Q'=(p',t')\in X\times\mathbb{R}$ defined by 
\[
\mathrm{dist}_{g}(Q,Q'):=\max\left\{d_{g}(p,p'),\sqrt{|t-t'|}\right\}. 
\]
We fix a background Riemannian metric $\bar{g}$ on $X$ 
and write $B(Q):=\{\,Q'\in X\times\mathbb{R}\mid \mathrm{dist}_{\bar{g}}(Q',Q)<1\,\}$. 
Fix $0<\alpha<1$. 
Then, for a pair of a smooth function $f:V\to \mathbb{R}$ 
and a K\"ahler metric $g$ on $X$, we define 
its parabolic partial $C^{3,\alpha}$-norm at $Q\in V$ by
\[
\begin{aligned}
|(g,f)|_{3,\alpha}(Q):=&\sup_{Q'\in B(Q)\cap V}\left(|\partial_{t} f|+|\partial_{t}\nabla f|+|\nabla^3 f|\right)(Q')\\
&+\sup_{\substack{Q_{1},Q_{2}\in B(Q)\cap V\\Q_{1}\neq Q_{2}}}\frac{|\partial_{t}\nabla f(Q_{1})-\partial_{t}\nabla f(Q_{2})|}{\mathrm{dist}_{g}(Q_{1},Q_{2})^{\alpha}}\\
&+\sup_{\substack{Q_{1},Q_{2}\in B(Q)\cap V\\ Q_{1}\neq Q_{2}}}\frac{|\nabla^{3} f(Q_{1})-\nabla^{3} f(Q_{2})|}{\mathrm{dist}_{g}(Q_{1},Q_{2})^{\alpha}}. 
\end{aligned}
\]

\begin{remark}
Actually, $|(g,f)|_{3,\alpha}$ is not a norm in the strict sense. 
Since it clearly depends on the metric $g$, the symbol $g$ is included in $|(g,f)|_{3,\alpha}$. 
We remark that $\nabla$ is the Levi-Civita connection with respect to $g$ and we measure norms of tensors and $\mathrm{dist}_{g}(Q_{1},Q_{2})$ by $g$, but $B(Q)$ is always defined by the fixed background metric $\bar{g}$. 
We also remark that $|(g,f)|_{3,\alpha}$ is almost the usual 
parabolic $C^{3,\alpha}$-norm, however, 
$|f|$, $|\nabla f|$ and $|\nabla^{2}f|$ are not included in $|(g,f)|_{3,\alpha}$. 
\end{remark}

Following Definition \ref{defofscaling}, 
we define the $\lambda$-parabolic scaling of $(g,f)$ at $t=t_{0}$ for $\lambda\in\mathbb{N}$ by 
\[D_{\lambda}^{t_{0}}(g,f):=(\lambda g,f_{\lambda})\quad\mbox{with}\quad f_{\lambda}(\,\cdot\,,t):=\lambda f(\,\cdot\,,t_{0}+t/\lambda), \]
where $f_{\lambda}(t)$ is defined for $t\in [\lambda (a-t_{0}),\lambda(b-t_{0}))$. 
We also define 
\[D^{t_{0}}_{\lambda}(V):=U\times  [\lambda (a-t_{0}),\lambda(b-t_{0})). \]
It is easy to see that 
\begin{equation}\label{DlDl}
D_{\lambda}^{s_{0}}\circ D_{\kappa}^{t_{0}}(g,f)=D_{\lambda\kappa}^{t_{0}+s_{0}\kappa^{-1}}(g,f). 
\end{equation}
One can also prove that if $0<\lambda\leq 1$ then 
\begin{equation}\label{scalesmall}
|(g,f)|_{3,\alpha}(p,0)\leq |D^{0}_{\lambda}(g,f)|_{3,\alpha}(p,0)\leq \lambda^{-(1+\alpha)/2}|(g,f)|_{3,\alpha}(p,0), 
\end{equation}
and if $\lambda\geq 1$ then 
\begin{equation}\label{scalebig}
\lambda^{-(1+\alpha)/2}|(g,f)|_{3,\alpha}(p,0)\leq |D^{0}_{\lambda}(g,f)|_{3,\alpha}(p,0)\leq |(g,f)|_{3,\alpha}(p,0). 
\end{equation}

For $Q=(p_{0},t_{0})$, we define 
\[K_{3,\alpha}((g,f),Q):=\inf\left\{\, \sqrt{\lambda}>0\mid |D_{\lambda}^{t_{0}}(g,f)|_{3,\alpha}(p_{0},0)\leq 1 \,\right\}. \]
Then, by \eqref{DlDl}, we have 
\begin{equation}\label{scaleproK}
\begin{aligned}
&K_{3,\alpha}(D_{\kappa}^{t_{0}}(g,f),(p_{0},s_{0}))\\
=&\inf\left\{\, \sqrt{\lambda}>0\mid |D_{\lambda\kappa}^{t_{0}+s_{0}\lambda^{-1}}(g,f)|_{3,\alpha}(p_{0},0)\leq 1 \,\right\}\\
=&\sqrt{\kappa}^{-1}K_{3,\alpha}((g,f),(p_{0},t_{0}+s_{0}\lambda^{-1})). 
\end{aligned}
\end{equation}
On the other hand, we have 
\begin{equation}\label{scaleprodist}
\begin{aligned}
&\mathrm{dist}_{\kappa g}((p_{0},s_{0}), D^{t_{0}}_{\kappa}(V))\\
=&\min\left\{\inf_{q\in U^{c}} d_{\kappa g_i}(p_{0},q), \sqrt{\kappa(b-t_{0})-s_{0}} ,\sqrt{s_{0}-\kappa (a-t_{0})}\right\}\\
=& \sqrt{\kappa}\min\left\{\inf_{q\in U^{c}} d_{g}(p_{0},q), \sqrt{b-(t_{0}+s_{0}\kappa^{-1})} ,\sqrt{(t_{0}+s_{0}\kappa^{-1})-a}\right\}\\
=& \sqrt{\kappa} \mathrm{dist}_g((p_{0},t_{0}+s_{0}\kappa^{-1}),V).
\end{aligned}
\end{equation}
Hence, by putting $D_{\kappa}^{t_{0}}(p,t):=(p,\kappa(t-t_{0}))$, we have 
\begin{equation}\label{invDK}
\begin{aligned}
&\mathrm{dist}_g(Q,V)\cdot K_{3,\alpha}((g,f),Q)\\
=&\mathrm{dist}_{\kappa g}(D_{\kappa}^{t_{0}}(Q), D^{t_{0}}_{\kappa}(V))\cdot K_{3,\alpha}(D_{\kappa}^{t_{0}}(g,f),D_{\kappa}^{t_{0}}(Q)), 
\end{aligned}
\end{equation}
for all $Q\in V$. 
Then, define 
\[K_{3,\alpha;V}(g,f):=\sup_{Q\in V}\biggl(\mathrm{dist}_g(Q,V)\cdot K_{3,\alpha}((g,f),Q)\biggr). \]

Now, we can start the proof of Theorem \ref{epregthm}, 
the $\varepsilon$-regularity theorem. 

\begin{proof}[Proof of Theorem \ref{epregthm}]
If the statement is false, 
then for any sequences $C_{i}\to \infty$ and $\varepsilon_{i}\to 0$ 
there exists a sequences of holomorphic line bundle $L_{i}\to X$, 
line bundle mean curvature flows $h_{i}=\{\,h_{i}(t)\,\}_{t\in[0,T_{i})}$ on $X$ 
so that each $h_{i}(t)$ is a Hermitian metric of $L_{i}$ 
and a nonvanishing holomorphic section $e_{i}\in\Gamma(U',L_{i})$ 
so that $h_{t}$ is graphical on $U'$ for all $t\in[0,T)$ with respect to $e\in\Gamma(U',L)$. 
Put $\phi_{i}:=-\log h_{i}(\bar{e}_{i},e_{i}):U'\times[0,T_{i})\to\mathbb{R}$. 
We can further assume that, by putting $U:=\varphi(B(r)\times B(r'))$ and $V:=U\times [0,T_{i})$, 
$\sup_{V_{i}}|F(h_{i}(t))|\leq A$ and 
\begin{equation}\label{trigger1}
    \bar{\Theta}(h_{i},Q,t)\leq 1+ \varepsilon_{i}
\end{equation}
for all $Q=(p,T')\in U\times(0,T_{i})$ and $t\in(T'-(\mathrm{dist}_{g}(Q,V_{i}))^2,T')\cap(0,T_{i})$, and
\[K_{3,\alpha;V_{i}}(g,\phi_{i})=\sup_{Q\in V_{i}}\biggl(\mathrm{dist}_g(Q,V_{i})\cdot K_{3,\alpha}((g,\phi_{i}),Q)\biggr)>C_{i}. \]
Put $\sqrt{k_{i}}:=K_{3,\alpha;V_{i}}(g,\phi_{i})>C_{i}$. 
Fix a point $\tilde{Q}_{i}=(\tilde{p}_{i},\tilde{T}_{i})\in V_{i}$ so that 
\begin{equation}\label{defOi}
\mathrm{dist}_g(\tilde{Q}_{i},V_{i})\cdot K_{3,\alpha}((g,\phi_{i}),\tilde{Q}_{i})>\frac{\sqrt{k_{i}}}{2}. 
\end{equation}

We do the blow-up argument to get a contradiction. 
Put 
\[\tilde{k}_{i}:=\left\lfloor (K_{3,\alpha}((g,\phi_{i}),\tilde{Q}_{i}))^2\right\rfloor 
\quad\mbox{and}\quad
\nu_{i}:=(K_{3,\alpha}((g,\phi_{i}),\tilde{Q}_{i}))^2-\tilde{k}_{i}
\]
where $\lfloor x\rfloor $ is the biggest integer which does not exceed $x$. 
Thus, $\nu_{i}$ is just the fractional part of 
$(K_{3,\alpha}((g,\phi_{i}),\tilde{Q}_{i}))^2$, 
and it's clear that $0\leq \nu_{i}<1$. 

By the definition of $\mathrm{dist}_{g}(\tilde{Q}_{i},V_{i})$ 
and the assumption which ensures that $U'$ is bounded, 
we see that $\mathrm{dist}_{g}(\tilde{Q}_{i},V_{i})\leq \mathrm{diam}_{g}(U)\leq \mathrm{diam}_{g}(U')<\infty$. 
Then, it is easy to see that 
\[\sqrt{\tilde{k}_{i}}+1\geq K_{3,\alpha}((g,\phi_{i}),\tilde{Q}_{i})\geq \frac{\sqrt{k_{i}}}{2}\times\frac{1}{\mathrm{diam}_{g}(U)}. \]
Since $k_{i}\to \infty$ as $i\to \infty$, we have proved that 
\begin{equation}\label{tildkinf}
\tilde{k}_{i}\to \infty \quad\mbox{as}\quad i\to \infty. 
\end{equation}

Define the rescaled triplets by $((X,g_{i}),\tilde{L}_{i},\tilde{i}):=D_{\tilde{k}_{i}}^{\tilde{T}_{i}}((X,g),L_{i},h_{i})$, explicitly 
\[g_{i}:=\tilde{k}_{i}g \quad \mbox{and} \quad \tilde{h}_{i}(s):=h_{i}^{\otimes\tilde{k}_{i}}(\tilde{T}_{i}+s/\tilde{k}_{i}). \]
Put $S_{i}:=\tilde{k}_{i}(T_{i}-\tilde{T}_{i})>0$, $S'_{i}:=\tilde{k}_{i}\tilde{T}_{i}>0$, 
$I'_{i}:=[-S'_{i},S_{i})$ and $V_{i}':=U\times I'_{i}$ for notational simplicity. 
Then, $\tilde{h}_{i}(s)$ is a Hermitian metric of $L_{i}^{\otimes\tilde{k}_{i}}$ 
defined for $s\in I'_{i}$ 
and by putting $\tilde{\phi}_{i}(s):=-\log\tilde{h}_{i}(s)(\bar{e}_{i}^{\otimes\tilde{k}_{i}},e_{i}^{\otimes\tilde{k}_{i}})$, we have 
\begin{equation}\label{phitildei}
\tilde{\phi}_{i}(s)=\tilde{k}_{i}\phi_{i}\left( \tilde{T}_{i}+s/\tilde{k}_{i} \right). 
\end{equation}
This means that 
\[(g_{i},\tilde{\phi}_{i})=D_{\tilde{k}_{i}}^{\tilde{T}_{i}}(g,\phi_{i})\quad\mbox{and}\quad D_{\tilde{k}_{i}}^{\tilde{T}_{i}}(V_{i})=V_{i}'. \]
Then, by \eqref{scaleproK}, we have 
\begin{equation}\label{normDF}
K_{3,\alpha}((g_{i},\tilde{\phi}_{i}),(\tilde{p}_{i},0))
=\tilde{k}_{i}^{-1/2}K_{3,\alpha}((g,\phi_{i}),\tilde{Q}_{i})
=\sqrt{1+\frac{\nu_{i}}{\tilde{k}_{i}}}. 
\end{equation}

\begin{claim}
For any point $Q'=(p',s')$ in $U\times I'_{i}=V_{i}'$, we have 
 \begin{equation}\label{triDF}
\mathrm{dist}_{g_{i}}((\tilde{p}_{i},0),V'_{i})\leq
 \mathrm{dist}_{g_{i}}(Q',V'_{i})+\left(d_{g_{i}}(\tilde{p}_{i},p')+\sqrt{|s'|}\right). 
 \end{equation}
\end{claim}
\noindent\textit{\proofname.}
It is easy to see that 
\[\inf_{q\in V^{'c}_i} d_{g_i}(\tilde{p}_{i},q)
\leq \inf_{q\in V^{'c}_i} d_{g_i}(p',q)+d_{g_i}(\tilde{p}_{i},p'). \]
Hence, it is enough to prove 
\[\min\left\{\sqrt{S_{i}}, \sqrt{S'_{i}}\right\}
\leq \min\left\{ \sqrt{S_i-s'},\sqrt{S'_{i}+s'}\right\}+\sqrt{|s'|}. \]
But, this follows from an elementary inequality $\sqrt{a+b}\leq \sqrt{a}+\sqrt{b}$ for $a,b\geq 0$. Then, the proof of this claim is complete. 
\hspace{\fill}\qedsymbol\vspace{3mm}
 
By \eqref{invDK}, the definition of $k_{i}$ and \eqref{defOi}, with a relation $D_{\tilde{k}_{i}}^{\tilde{T}_{i}}(Q)=Q'$, we have 
\begin{equation}\label{invDF2}
\begin{aligned}
&\mathrm{dist}_{g_{i}}(Q',V_{i}')\cdot K_{3,\alpha}((g_{i},\tilde{\phi}_{i}),Q')\\
=&\mathrm{dist}_{g}(Q,V_{i})\cdot K_{3,\alpha}((g,\phi_{i}),Q)\\
\leq & \sqrt{k_{i}}\\
<& 2\mathrm{dist}_g(\tilde{Q}_{i},V_{i})\cdot K_{3,\alpha}((g,\phi_{i}),\tilde{Q}_{i}), 
\end{aligned}
\end{equation}
for all $Q'=(p',s')$ in $U\times I'_{i}=V_{i}'$. 
By the first equality of \eqref{invDF2} with $Q':=(\tilde{p}_{i},0)$, we have 
\begin{equation}\label{invDF3}
\begin{aligned}
&\mathrm{dist}_{g_{i}}((\tilde{p}_{i},0),V_{i}')\cdot K_{3,\alpha}((g_{i},\tilde{\phi}_{i}),(\tilde{p}_{i},0))\\
=&\mathrm{dist}_{g}(\tilde{Q}_{i},V_{i})\cdot K_{3,\alpha}((g,\phi_{i}),\tilde{Q}_{i}). 
\end{aligned}
\end{equation}
Combining \eqref{invDF2}, \eqref{invDF3} and \eqref{normDF} implies that
\[
\mathrm{dist}_{g_{i}}(Q',V_{i}')\cdot K_{3,\alpha}((g_{i},\tilde{\phi}_{i}),Q')\leq 2\sqrt{1+\frac{\nu_{i}}{\tilde{k}_{i}}}\mathrm{dist}_{g_{i}}((\tilde{p}_{i},0),V'_{i}) 
\]
for all $Q'=(p',s')$ in $U\times I'_{i}=V_{i}'$. 
Dividing both hand side by $\mathrm{dist}_{g_{i}}(Q',V'_{i})$ and using \eqref{triDF} yield that 
\begin{equation}\label{estDDF2}
K_{3,\alpha}((g_{i},\tilde{\phi}_{i}),Q')\leq 2\sqrt{1+\frac{\nu_{i}}{\tilde{k}_{i}}}
\left(1-\frac{\left(d_{g_{i}}(\tilde{p}_{i},p')+\sqrt{|s'|}\right)}{\mathrm{dist}_{g_{i}}((\tilde{p}_{i},0),V'_{i})}\right)^{-1}
\end{equation}
for all $Q'=(p',s')$ in $U\times I'_{i}=V_{i}'$ 
whenever the right hand side is positive. 
Combining \eqref{invDF2}, \eqref{invDF3} and \eqref{normDF} also implies 
\[
\frac{1}{2}\sqrt{k_{i}}<\sqrt{1+\frac{\nu_{i}}{\tilde{k}_{i}}}\mathrm{dist}_{g_{i}}((\tilde{p}_{i},0),V'_{i}). 
\]
Since $1\leq 1+\frac{\nu_{i}}{\tilde{k}_{i}}<2$ and $k_{i}\to \infty$ as $i\to\infty$, we see that 
\begin{equation}\label{infdisout}
\mathrm{dist}_{g_{i}}((\tilde{p}_{i},0),V'_{i})\to \infty\quad\mbox{as}\quad i\to\infty. 
\end{equation}
Especially, we have 
\begin{equation}\label{infinf}
\inf_{q\in U_{i}^{c}}d_{g_{i}}(\tilde{p}_{i},q)\to \infty,\quad -S'_{i}\to-\infty \quad\mbox{and}\quad S_{i}\to\infty\quad\mbox{as}\quad i\to\infty.
\end{equation}

Now we have biholomorphic maps $\varphi: B(4r)\times B(r')\to U'$. 
Define $x_{i}\in B(r)$ and $y_{i}\in B(r')$ so that 
$\varphi(x_{i},y_{i})=\tilde{p}_{i}$. 
Fix $R>0$ and 
consider a map $\tilde{\varphi}_{i}:B(R)\times B(r')\to U$ defined by 
\begin{equation}\label{defoftilphi}
\tilde{\varphi}_{i}(x,y):=\varphi\left(x_{i}+\tilde{k}_{i}^{-1/2}x,y_{i}+\tilde{k}_{i}^{-1/2}y\right). 
\end{equation}
We remark that $\tilde{\varphi}_{i}$ is locally biholomorphic and $\tilde{\varphi}_{i}(0,0)=\tilde{p}_{i}$. 

\begin{claim}\label{Claim3}
For any $R>0$, there exists $N>0$ such that $x_{i}+\tilde{k}_{i}^{-1/2}x\in B(r)$ and $y_{i}+\tilde{k}_{i}^{-1/2}y\in B(r')$ for all $i>N$ and $(x,y)\in B(R)\times B(r')$. 
\end{claim}
\noindent\textit{\proofname.}
Since $\tilde{k}_{i}^{-1/2}\to 0$ and the condition so that 
$y_{i}+\tilde{k}_{i}^{-1/2}y\in B(r')$ for all $i>N$ and $y\in B(r')$ does not depends on  $R$, 
we can assume that $y_{i}+\tilde{k}_{i}^{-1/2}y\in B(r')$ already. 
Assume that there exists $\beta>0$ such that the following inequality holds for for all $i$: 
\begin{equation}\label{distdist}
\beta^{-1}\inf_{q\in U^{c}}d_{g_{i}}(\tilde{p}_{i},q)\leq \tilde{k}_{i}^{1/2}\inf_{x'\in B(r)^{c}}d_{\mathbb{R}^{n}}(x_{i},x'). 
\end{equation}
Then, the left hand side tends to $\infty$ when $i\to \infty$ by \eqref{infinf}
and the right hand side is just $\tilde{k}_{i}^{1/2}(r-|x_{i}|)$. 
Thus, for any $R$ there exits $N>0$ such that $R<\tilde{k}_{i}^{1/2}(r-|x_{i}|)$ for all $i>N$. 
Then, for any $x\in B(R)$, we have $|x_{i}+\tilde{k}_{i}^{-1/2}x|\leq |x_{i}|+\tilde{k}_{i}^{-1/2}R<r$ and this is the desired conclusion. 
Thus, it is enough to prove \eqref{distdist}. 
 
To prove \eqref{distdist}, 
fix a point $x'\in \partial B(r)$ such that 
\[d_{\mathbb{R}^{n}}(x_{i},x')=\inf_{x'\in B(r)^{c}}d_{\mathbb{R}^{n}}(x_{i},x'). \] 
Put $p':=\varphi(x',y_{i})\in U$ and let $\beta>0$ be a constant so that 
$\varphi_{i}^{*}g\leq \beta^{2}g_{\mathrm{st}}$ on $\varphi(\overline{B(r)\times B(r')})$, 
where $g_{\mathrm{st}}:=dx^2+dy^2$
Then, since $\tilde{p}_{i}=\varphi(x_{i},y_{i})$ and $g_{i}=\tilde{k}_{i}g$, 
we have 
\begin{equation}\label{estofdists}
\begin{aligned}
d_{g_{i}}(\tilde{p}_{i},p')
=&\tilde{k}_{i}^{1/2}d_{g}(\tilde{p}_{i},p')
\leq \tilde{k}_{i}^{1/2}d_{\varphi_{i}^{*}g}((x_{i},y_{i}),(x',y_{i}))\\
\leq & \beta\tilde{k}_{i}^{1/2}d_{g_{\mathrm{st}}}((x_{i},y_{i}),(x',y_{i}))
=\beta\tilde{k}_{i}^{1/2}d_{\mathbb{R}^{n}}(x_{i},x'), 
\end{aligned}
\end{equation}
and the proof of this claim is complete. 
\hspace{\fill}\qedsymbol\vspace{3mm}

Fix a radius $0<R<\infty$. 
Then, we know that, by Claim \ref{Claim3}, 
there exists $N\in\mathbb{N}$ such that 
$\tilde{\varphi}_{i}:B(R)\times B(r')\to U_{i}$ is defined 
for all $i\geq N$. 
Furthermore, by \eqref{infinf}, 
we can also assume that $(-R,R)\subset [-S_{i}',S_{i})$. 
Put a K\"ahler metric on $B(R)\times B(r')$ by 
\[G_{i}:=\tilde{\varphi}_{i}^{*}g_{i}. \]
Moreover, for each $s\in (-R,R)$, put 
\[\hat{h}_{i}(s):=\exp\left(-\tilde{\varphi}_{i}^{*}\tilde{\phi}_{i}(s)\right), \]
where $\tilde{\phi}_{i}(s)$ is defined by \eqref{phitildei}. 
Then, $\hat{h}_{i}(s)$ is a positive function on $B(R)\times B(r')$ and 
it can be regarded as a Hermitian metric on the trivial $\mathbb{C}$-bundle 
over $B(R)\times B(r')$. 

Since $\tilde{h}_{i}$ is a line bundle mean curvature flow on $(U_{i},g_{i})$, 
$\tilde{\phi}_{i}(=-\log \tilde{h}_{i})$ satisfies the line bundle mean curvature flow equation. 
Here note that actually the line bundle mean curvature flow equation is a PDE for $\tilde{\phi}_{i}$. 
Then, since we just defined $-\log \hat{h}_{i}(=\tilde{\varphi}_{i}^{*}\tilde{\phi}_{i})$ and $G_{i}$ as the pull back of $\tilde{\phi}_{i}$ and $g_{i}$, 
it is clear that $\tilde{\varphi}_{i}^{*}\tilde{\phi}_{i}$ satisfies the line bundle mean curvature flow equation with respect to the K\"ahler metric $G_{i}$. 
This means that $\hat{h}_{i}:=\{\,\hat{h}_{i}(s)\,\}_{s\in(-R,R)}$ 
is a line bundle mean curvature flow 
on $B(R)\times B(r')$ 
with respect to the K\"ahler metric $G_{i}$. 

\begin{claim}\label{Claim3.5}
There exists a subsequence, we still denote it by $i$, such that 
the K\"ahler metrics $G_{i}$ converge to 
a smooth K\"ahler metric $G_{\infty}$ on 
$\mathbb{R}^{n}\times B(r')$ in $C^{\infty}$-sense on each compact subset. 
Moreover, when we write the associated K\"ahler form of $G_{\infty}$ 
by $(\sqrt{-1}/2)G_{\bar{k}j}dz^{j}\wedge d\bar{z}^{k}$, 
then $G_{\bar{k}j}$ are constants satisfying 
$G_{\bar{k}j}=G_{\bar{j}k}$. 
\end{claim}

\noindent\textit{\proofname.}
Since $\tilde{p}_{i}$ is in $U$ and $\varphi(\overline{B(r)\times B(r')})$ is compact and 
contained in $U'=\varphi(B(4r)\times \overline{B(r')})$, 
there exists a point $\tilde{p}_{\infty}\in U'$ and a subsequence, we still denote it by $i$, such that 
$\tilde{p}_{i}\to \tilde{p}_{\infty}$ as $i\to \infty$. 
Then, by the definition of $G_{i}$, semi-flat assumption and 
the fact that $\tilde{k}_{i}\to \infty$ by \eqref{tildkinf}, 
the claim is proved. 
In addition, we see that $G_{\bar{k}j}=g_{\bar{k}j}(\tilde{p}_{\infty})$. 
\hspace{\fill}\qedsymbol\vspace{3mm}

\begin{claim}\label{Claim4}
Put $f_{i}:=\tilde{\varphi}_{i}^{*}\tilde{\phi}_{i}$. 
Then, there exist $M(R)\in\mathbb{N}$ and $C=C(R,A)>0$ such that 
\begin{equation}\label{C0bound'}
\begin{aligned}
&\sup_{Q\in W_{R}}
\left(|\partial_{s}f|+|\nabla^2 f_{i}|+|\partial_{s}\nabla f_{i}|+|\nabla^3 f_{i}|\right)(Q)\\
&+\sup_{\substack{Q_{1},Q_{2}\in W_{R}\\ Q_{1}\neq Q_{2}}}
\frac{|\partial_{s}\nabla f_{i}(Q_{1})-\partial_{s}\nabla f_{i}(Q_{2})|}{\mathrm{dist}_{G_{i}}(Q_{1},Q_{2})^{\alpha}}\\
&+\sup_{\substack{Q_{1},Q_{2}\in W_{R}\\ Q_{1}\neq Q_{2}}}
\frac{|\nabla^{3} f_{i}(Q_{1})-\nabla^{3} f_{i}(Q_{2})|}{\mathrm{dist}_{G_{i}}(Q_{1},Q_{2})^{\alpha}}
\leq C
\end{aligned}
\end{equation}
for all $i\geq M(R)$, 
where $W_{R}:=(B(R)\times B(r'))\times (-R,R)$. 
\end{claim}
 
\noindent\textit{\proofname.}
Fix a space-time point $Q=((x,y),s)\in W_{R}$. 
Put $p:=\tilde{\varphi}_{i}(x,y)$ and $Q':=(p,s)$. 
Then, by \eqref{estDDF2}, we see that 
\begin{equation}\label{estDDF22'}
K_{3,\alpha}((g_{i},\tilde{\phi}_{i}),Q')\leq 2\sqrt{1+\frac{\nu_{i}}{\tilde{k}_{i}}}
 \left(
 1-\frac{\left(d_{g_{i}}(\tilde{p}_{i},p)+\sqrt{|s|}\right)}{\mathrm{dist}_{g_{i}}((\tilde{p}_{i},0),V'_{i})}\right)^{-1}. 
 \end{equation}
First, we have $\sqrt{|s|}\leq R$ since $s\in (-R,R)$. 
Next, it follows that 
\[d_{g_{i}}(\tilde{p}_{i},p)\leq\beta\sqrt{R^2+(r')^2}. \]
This is seen as follows. By the definition of $\tilde{\varphi}_{i}$, 
that is, \eqref{defoftilphi}, we have 
$p=\tilde{\varphi}_{i}(x,y)=\varphi(x_{i}+\tilde{k}_{i}^{-1/2}x,y_{i}+\tilde{k}_{i}^{-1/2}y)$ and 
we also have $\tilde{p}_{i}=\varphi(x_{i},y_{i})$. 
Then, by the same argument as \eqref{estofdists}, we get 
\[
d_{g_{i}}(\tilde{p}_{i},p)
\leq \beta\tilde{k}_{i}^{1/2}d_{g_{\mathrm{st}}}((x_{i}+\tilde{k}_{i}^{-1/2}x,y_{i}+\tilde{k}_{i}^{-1/2}y),(x_{i},y_{i})). 
\]
Then, the proof is complete since $x\in B(R)$ and $y\in B(r')$. 

Then, by \eqref{infdisout} and \eqref{tildkinf}, we see that the right hand side of \eqref{estDDF22'} converges to 2 uniformly when $i\to\infty$. 
Especially, there exists $M(R)\in\mathbb{N}$ such that the right hand side of \eqref{estDDF22'} is less than 2.5 for all $i\geq M(R)$. 
Then, by the definition of $K_{3,\alpha}((g_{i},\tilde{\phi}_{i}),Q')$, we have 
\[3\in \{\, \sqrt{\lambda}>0\mid |D_{\lambda}^{s}(g_{i},\tilde{\phi}_{i})|_{3,\alpha}(p,0)\leq 1 \,\}. \]
This implies that $|D_{9}^{s}(g_{i},\tilde{\phi}_{i})|_{3,\alpha}(p,0)\leq 1$ for each $i\geq M(R)$. 
Put $(g, f):=D_{9}^{s}(g_{i},\tilde{\phi}_{i})$ for simplicity. 
Then, we have $g=9g_{i}$ and 
$f(t)=9\tilde{k}_{i}\phi_{i}( \tilde{T}_{i}+s/\tilde{k}_{i}+t/(9\tilde{k}_{i}))$, 
where $t$ is the variable of $f(t)$ and $s$ is fixed. 
Then, for example, we have
\[\frac{\partial}{\partial t}\bigg|_{t=0}\nabla f(t)
=\frac{\partial}{\partial \bar{s}}\bigg|_{\bar{s}=s}\nabla\Bigl(9\tilde{k}_{i}\phi_{i}( \tilde{T}_{i}+\bar{s}/\tilde{k}_{i})\Bigr)\times \frac{1}{9}
=\frac{\partial}{\partial t}\bigg|_{\bar{s}=s}\nabla \tilde{\phi}_{i}(\bar{s}). \]
Next, we consider the set $D_{9}^{s}(V_{i}')$. Then, we have 
\[D_{9}^{s}(V_{i}')=D_{9}^{s}(U\times [-S_{i}',S_{i}))=U\times [9(-S_{i}'-s),9(S_{i}-s)). \]
In particular, this set contains $(p,0)$ since $s\in (-R,R)\subset [-S_{i}',S_{i})$. 
Thus, we see that $(p,0)\in B((p,0))\cap D_{9}^{s}(V_{i}')$.  
By the definition of $|\,\cdot \,|_{3,\alpha}(p,0)$, we have 
\[
\begin{aligned}
|(g,f)|_{3,\alpha}(p,0)=&\sup_{Q''\in B((p,0))\cap D_{9}^{s}(V_{i}')}
\left(|\partial_{t} f|+|\partial_{t}\nabla f|+|\nabla^3 f|\right)(Q'')\\
&+\sup_{\substack{Q_{1},Q_{2}\in B((p,0))\cap D_{9}^{s}(V_{i}')\\ Q_{1}\neq Q_{2}}}
\frac{|\partial_{t}\nabla f(Q_{1})-\partial_{t}\nabla f(Q_{2})|}{\mathrm{dist}_{g}(Q_{1},Q_{2})^{\alpha}}\\
&+\sup_{\substack{Q_{1},Q_{2}\in B((p,0))\cap D_{9}^{s}(V_{i}')\\ Q_{1}\neq Q_{2}}}
\frac{|\nabla^{3} f(Q_{1})-\nabla^{3} f(Q_{2})|}{\mathrm{dist}_{g}(Q_{1},Q_{2})^{\alpha}}, 
\end{aligned}
\]
with respect to the Riemannian metric $g=9g_{i}$. 
Then, since $9g_{i}$ and $g_{i}$ is uniformly equivalent and the value of $f$ on a neighborhood of $(p,0)$ 
corresponds to the one of $\tilde{\phi}_{i}$ on a neighborhood of $(p,s)$, 
we can say that for each compact set $B \subset V_{i}'$ there exists $C(B)>0$ such that 
\begin{equation}\label{preestimate}
\begin{aligned}
&\sup_{Q''\in B}
\left(|\partial_{s} \tilde{\phi}_{i}|+|\partial_{s}\nabla \tilde{\phi}_{i}|+|\nabla^3 \tilde{\phi}_{i}|\right)(Q'')\\
&+\sup_{\substack{Q_{1},Q_{2}\in B\\ Q_{1}\neq Q_{2}}}
\frac{|\partial_{s}\nabla \tilde{\phi}_{i}(Q_{1})-\partial_{s}\nabla \tilde{\phi}_{i}(Q_{2})|}{\mathrm{dist}_{g_{i}}(Q_{1},Q_{2})^{\alpha}}\\
&+\sup_{\substack{Q_{1},Q_{2}\in B\\ Q_{1}\neq Q_{2}}}
\frac{|\nabla^{3} \tilde{\phi}_{i}(Q_{1})-\nabla^{3} \tilde{\phi}_{i}(Q_{2})|}{\mathrm{dist}_{g_{i}}(Q_{1},Q_{2})^{\alpha}}
\leq C(B)
\end{aligned}
\end{equation}
for all $i\geq M(R)$. 
On the other hand, by the scaling invariance of the quantity $|F|$ and the assumption $|F(h_{i})|\leq A$, we have $|F(\tilde{h}_{i})|\leq A$. 
Since $|F(\tilde{h}_{i})|=|\frac{1}{4}(\frac{\partial^2\tilde{\phi}_{i}}{\partial x^{k}\partial x^{\ell}})_{k\ell}|$, 
adding this term to the left hand side of \eqref{preestimate} implies that
\[
\begin{aligned}
&\sup_{Q''\in B}
\left(|\partial_{s} \tilde{\phi}_{i}|+|\nabla^2\tilde{\phi}_{i}|+|\partial_{s}\nabla \tilde{\phi}_{i}|+|\nabla^3 \tilde{\phi}_{i}|\right)(Q'')\\
&+\sup_{\substack{Q_{1},Q_{2}\in B\\ Q_{1}\neq Q_{2}}}
\frac{|\partial_{s}\nabla \tilde{\phi}_{i}(Q_{1})-\partial_{s}\nabla \tilde{\phi}_{i}(Q_{2})|}{\mathrm{dist}_{g_{i}}(Q_{1},Q_{2})^{\alpha}}\\
&+\sup_{\substack{Q_{1},Q_{2}\in B\\ Q_{1}\neq Q_{2}}}
\frac{|\nabla^{3} \tilde{\phi}_{i}(Q_{1})-\nabla^{3} \tilde{\phi}_{i}(Q_{2})|}{\mathrm{dist}_{g_{i}}(Q_{1},Q_{2})^{\alpha}}
\leq C(B,A)
\end{aligned}
\]
for all $i\geq M(R)$. 
Finally, since what we want to estimate is $f_{i}=\tilde{\varphi}_{i}^{*}\tilde{\phi}_{i}$ with respect to $G_{i}=\tilde{\varphi}_{i}^{*}g_{i}$, 
the same estimates hold by replacing $\tilde{\phi}_{i}$ with $f_{i}$ 
and $g_{i}$ with $G_{i}$. 
Then, the proof is complete. 
\hspace{\fill}\qedsymbol\vspace{3mm}

Put $w_{i}:=A_{(O,0)}f_{i}$, see \eqref{tildephi}. Explicitly, 
\[w_{i}((x,y),t):=f_{i}((x,y),t)-\left(f_{i}((0,0),0)+\sum_{i=1}^{n}\frac{\partial f_{i}}{\partial x^{j}}((0,0),0)x^{j}\right).\]
Then, we have $w_{i}((0,0),0)=0$ and $\frac{\partial w_{i}}{\partial x^{k}}((0,0),0)=0$. 
Since the difference between $f_{i}$ and $w_{i}$ is affine linear 
with respect to $x$-coordinates, 
$w_{i}$ also satisfies the same uniform estimate as in \eqref{C0bound'}. 
With this fact and the normalization $w_{i}((0,0),0)=\frac{\partial w_{i}}{\partial x^{k}}((0,0),0)=0$, 
we can say that there exist $M(R)\in\mathbb{N}$ and $C(R)>0$ such that 
\begin{equation}\label{locunibound}
|w_{i}|_{C^{3,\alpha}(W_{R})}\leq C(R)
\end{equation}
for all $i\geq M(R)$, where $|\,\cdot\,|_{C^{3,\alpha}(W_{R})}$ is the standard parabolic $C^{3,\alpha}$-norm on $W_{R}$. 

\begin{claim}\label{Claim5}
There exists a subsequence, we still denote it by $i$, such that functions $w_{i}$ converge to 
a smooth function $w_{\infty}$ defined 
on $(\mathbb{R}^{n}\times B(r'))\times(-\infty,\infty)$ 
in $C^{\infty}$-sense on each compact subset. 
Moreover, $w_{\infty}(s)$ is independent of the second component of $\mathbb{R}^{n}\times B(r')$ for 
all $s\in(-\infty,\infty)$. 
\end{claim}

\noindent\textit{\proofname.}
Let $R_{i}$ be a sequence such that $R_{i}\to\infty$ as $i\to \infty$. 
First, we work on $W_{R_{k}}=(B(R_{k})\times B(r'))\times (-R_{k},R_{k})$ for fixed $k$. 
By the definition of the line bundle mean curvature flow, 
we have 
\begin{equation}\label{pdeforpp}
\frac{\partial}{\partial s}(\partial_{j}\tilde{\phi_{i}})=H_j(\tilde{h}_{i})
=\tilde{\eta}(i)^{p\bar{q}}\nabla_p \tilde{F}_{\bar{q}j}(i)
=\tilde{\eta}(i)^{p\bar{q}}\nabla_p\partial_{\bar{q}}(\partial_j\tilde{\phi_i})=\Delta_{\tilde{\eta}(i)}(\partial_{j}\tilde{\phi_{i}}), 
\end{equation}
where $\tilde{F}(i)$ and $\tilde{\eta}(i)$ are defined by $\tilde{h}_{i}$ and $\nabla$ is the Levi-Civita connection of $g_{i}$. 
Since $f_i$ and $G_{i}$ are the pull back of $\tilde{\phi_i}$ and $g_{i}$ by $\tilde{\varphi}_{i}$, 
$\partial_{j}w_{i}$ also satisfies the same equation as \eqref{pdeforpp}. 

Then, by the following argument, we can get the higher derivatives of $w_i$.
Since $\tilde{\eta}^{p\bar{q}}(i)$ is the combination of the second derivatives of $w_i$, the first derivatives of the coefficients of $\Delta_{\tilde{\eta}(i)}$ and its $\alpha$-H\"older norm are uniformly bounded on each compact set by \eqref{locunibound}. 
Taking the derivatives of \eqref{pdeforpp}, then $\partial^2 w_i$ satisfies the following equation: 
\[\frac{\partial}{\partial s}(\partial^2 w_i)=\Delta_\eta(\partial^2 w_i)+\partial\tilde{\eta}(i)\ast \partial^3 w_i, \]
where $A\ast B$ is a term which can be written as 
linear combinations of 
some products of components of $A$ and $B$. 
Since the last term of the above equation is uniformly bounded in $C^\alpha$ by \eqref{locunibound}, 
by the Schauder estimate we see that 
$|\partial^2 w_i|_{C^{2,\alpha}(W_{R_{k}})}\leq C'(R_{k})$ 
for some $C'(R_{k})>0$. 
We can continue to do the bootstrap argument in this fashion 
and get all higher order bounds for $w_i$. 
Thus, from the standard Arzela-Ascoli theorem, 
we can get a subsequence which converges to a smooth function on $W_{R_{k}}$. 
Of course, this limit function inherits the graphical condition, that is, 
it does not depend on $y$. 
Finally, by using the usual diagonal argument with $R_{i}\to\infty$, 
we prove this claim. 
\hspace{\fill}\qedsymbol\vspace{3mm}

\begin{claim}
For any compact set $K\times[a,b]$ in $(\mathbb{R}^{n}\times B(r'))\times\mathbb{R}$ including $(O,0)$ 
there exists $M(K)\in\mathbb{N}$ such that for all $i\geq M(K)$
 \begin{equation}\label{keytocontra1}
 \begin{aligned}
& \sup_{Q'\in K\times[a,b]}\left(|\partial_{s}w_{i}|+|\partial_{s}\nabla w_{i}|+|\nabla^3 w_{i}|\right)(Q')\\
&+\sup_{\substack{Q_{1},Q_{2}\in K\times[a,b]\\ Q_{1}\neq Q_{2}}}\frac{|\partial_{s}\nabla w_{i}(Q_{1})-\partial_{s}\nabla w_{i}(Q_{2})|}{\mathrm{dist}_{G_{i}}(Q_{1},Q_{2})^{\alpha}}\\
&+\sup_{\substack{Q_{1},Q_{2}\in K\times[a,b]\\ Q_{1}\neq Q_{2}}}\frac{|\nabla^{3} w_{i}(Q_{1})-\nabla^{3} w_{i}(Q_{2})|}{\mathrm{dist}_{G_{i}}(Q_{1},Q_{2})^{\alpha}}\geq 1. 
\end{aligned}
\end{equation}
\end{claim}

\noindent\textit{\proofname.}
First, by \eqref{normDF}, we have 
\begin{equation}\label{claim6-1}
\inf\{\, \sqrt{\lambda}>0\mid |D_{\lambda}^{0}(g_{i},\tilde{\phi}_{i})|_{3,\alpha}(\tilde{p}_{i},0)\leq 1 \,\}\geq 1. 
\end{equation}
Then, we can prove that 
\begin{equation}\label{claim6-2}
|(g_{i},\tilde{\phi}_{i})|_{3,\alpha}(\tilde{p}_{i},0)\geq 1       
\end{equation}
as follow. Assume that $|(g_{i},\tilde{\phi}_{i})|_{3,\alpha}(\tilde{p}_{i},0)=:\nu<1$. 
Then, by \eqref{scalebig}, we have 
\[|D^{0}_{\lambda}(g_{i},\tilde{\phi}_{i})|_{3,\alpha}(\tilde{p}_{i},0)\leq |(g_{i},\tilde{\phi}_{i})|_{3,\alpha}(\tilde{p}_{i},0)<1\]
for $\lambda\geq 1$. 
Similarly, by \eqref{scalesmall}, we have 
\[|D^{0}_{\lambda}(g_{i},\tilde{\phi}_{i})|_{3,\alpha}(\tilde{p}_{i},0)\leq \lambda^{-(1+\alpha)/2}|(g_{i},\tilde{\phi}_{i})|_{3,\alpha}(\tilde{p}_{i},0)\leq 1\]
for $\nu^{2/(1+\alpha)}\leq \lambda\leq 1$. 
This implies that 
\[\inf\{\, \sqrt{\lambda}>0\mid |D_{\lambda}^{0}(g_{i},\tilde{\phi}_{i})|_{3,\alpha}(\tilde{p}_{i},0)\leq 1 \,\}\leq \nu^{1/(\alpha+1)}<1. \]
However, this contradicts to \eqref{claim6-1}. 
Thus, \eqref{claim6-2} holds. 
Then, by the definition of $|(g_{i},\tilde{\phi}_{i})|_{3,\alpha}(\tilde{p}_{i},0)$, we have 
\[
\begin{aligned}
&\sup_{Q'\in B((\tilde{p}_{i},0))\cap V_{i}'}\left(|\partial_{s} \tilde{\phi}_{i}|+|\partial_{s}\nabla \tilde{\phi}_{i}|+|\nabla^3 \tilde{\phi}_{i}|\right)(Q')\\
&+\sup_{\substack{Q_{1},Q_{2}\in B((\tilde{p}_{i},0))\cap V_{i}'\\Q_{1}\neq Q_{2}}}\frac{|\partial_{s}\nabla \tilde{\phi}_{i}(Q_{1})-\partial_{s}\nabla \tilde{\phi}_{i}(Q_{2})|}{\mathrm{dist}_{g}(Q_{1},Q_{2})^{\alpha}}\\
&+\sup_{\substack{Q_{1},Q_{2}\in B((\tilde{p}_{i},0))\cap V_{i}'\\ Q_{1}\neq Q_{2}}}\frac{|\nabla^{3} \tilde{\phi}_{i}(Q_{1})-\nabla^{3} \tilde{\phi}_{i}(Q_{2})|}{\mathrm{dist}_{g}(Q_{1},Q_{2})^{\alpha}}\geq 1. 
\end{aligned}
\]
One can easily see that for any compact set 
$K\times[a,b]$ in $(\mathbb{R}^{n}\times B(r'))\times\mathbb{R}$ including $(O,0)$ 
there exists $M(K)\in\mathbb{N}$ such that $\tilde{\varphi}_{i}(K)\times[a,b]\subset B((\tilde{p}_{i},0))\cap V_{i}'$ 
for all $i\geq M(K)$. 
Then, since $f_{i}$ and $G_{i}$ are the pull back of $\tilde{\phi_i}$ and $g_{i}$ by $\tilde{\varphi}_{i}$ 
and the difference between $f_{i}$ and $w_{i}$ is affine linear 
with respect to $x$-coordinates, we get \eqref{keytocontra1}. 
This completes the proof of this claim. 
\hspace{\fill}\qedsymbol\vspace{3mm}

\begin{claim}\label{Claim7}
$w_{\infty}$ is a quadratic function for all $s\in\mathbb{R}$. 
More precisely, there exist $A=(a_{ij})\in\mathrm{Sym}(n)$ such that $w_{\infty}(x,s)=a_{ij}x^{i}x^{j}$. 
\end{claim}

\noindent\textit{\proofname.}
Put $X_{\infty}:=\mathbb{R}^{n}\times B(r')$ and $H_{\infty}:=e^{-w_{\infty}}$. 
Then, $H_{\infty}$ is a line bundle mean curvature flow of the trivial bundle $\underline{\mathbb{C}}$ over $X_{\infty}$ defined for all $s\in\mathbb{R}$. 
By Claim \ref{Claim5}, $H_{\infty}$ is globally graphical on $\mathbb{R}^{n}$. 

Fix $s\in(-\infty,0)$ and $R>0$. We only consider all $i$ bigger than $N=N(R)$ 
appeared in Claim \ref{Claim3}. 
Put $t_{i}:=\tilde{T}_{i}+s/\tilde{k}_{i}<\tilde{T}_{i}$. 
Then, by \eqref{infdisout}, we see that $(\mathrm{dist}_{g_{i}}((\tilde{p}_{i},0),V_{i}'))^2>-s$ 
for all sufficiently large $i$ with $i>N$. 
Using \eqref{scaleprodist} implies that 
\[\mathrm{dist}_{g_{i}}((\tilde{p}_{i},0),V_{i}')
=\tilde{k}_{i}^{1/2}\mathrm{dist}_{g}(\tilde{Q}_{i},V_{i}). \]
Thus, combining the definition of $t_{i}$, we get $t_{i}>\tilde{T}_{i}-(\mathrm{dist}_{g}(\tilde{Q}_{i},V_{i}))^2$. 
This means that we can use the assumption \eqref{trigger1} for $t=t_{i}$. Then, we have 
$\bar{\Theta}(h_{i},\tilde{Q}_{i},t_{i})\leq 1+\varepsilon_{i}$. 
By definition, we have 
\begin{equation}\label{precom0}
   \bar{\Theta}(h_{i},\tilde{Q}_{i},t_{i})=\frac{(2\sqrt{2})^{n}}{\mathop{\mathrm{Vol}_{g}}(B(r')_{\tilde{p}_{i}})}\Theta(A_{\tilde{Q}_{i}}h_{i},\tilde{Q}_{i},t_{i}).  
\end{equation}
By the scaling invariance of the density \eqref{densityinv2}, for $\tilde{Q}_{i}=(\tilde{p}_{i},\tilde{T}_{i})$, we have 
\begin{equation}\label{precom1}
\begin{aligned}
\frac{(2\sqrt{2})^{n}}{\mathop{\mathrm{Vol}_{g}}(B(r')_{\tilde{p}_{i}})}\Theta(A_{\tilde{Q}_{i}}h_{i},\tilde{Q}_{i},t_{i})=\frac{(2\sqrt{2})^{n}}{\mathop{\mathrm{Vol}_{\tilde{k}_{i}g}}(B(r')_{\tilde{p}_{i}})}\Theta(D_{\tilde{k}_{i}}^{\tilde{T}_{i}}A_{\tilde{Q}_{i}}h_{i},Q_{i}',s), 
\end{aligned}
\end{equation}
where $Q_{i}'=(\tilde{p}_{i},0)$. 
Since $A_{\tilde{Q}_{i}}h_{i}$ is defined by 
\[(A_{\tilde{Q}_{i}}h_{i})_{t}:=e^{-(A_{\tilde{Q}_{i}}\phi_{i})(t)}(\bar{e}_{i}^{*}\otimes e_{i}), \]
we have 
\[(D_{\tilde{k}_{i}}^{\tilde{T}_{i}}A_{\tilde{Q}_{i}}h_{i})_{s}:=e^{-\tilde{k}_{i}(A_{\tilde{Q}_{i}}\phi_{i})(\tilde{T}_{i}+s/\tilde{k}_{i})}(\bar{\xi}_{i}^{*}\otimes \xi_{i}), \]
where $\xi_{i}:=e_{i}^{\otimes \tilde{k}_{i}}$. 
On the other hand, we have 
\[(D_{\tilde{k}_{i}}^{\tilde{T}_{i}}h_{i})_{s}
=e^{-\tilde{k}_{i}\phi_{i}(\tilde{T}_{i}+s/\tilde{k}_{i})}(\bar{\xi}_{i}^{*}\otimes \xi_{i})
=e^{-\tilde{\phi}_{i}(s)}(\bar{\xi}_{i}^{*}\otimes \xi_{i}), \]
where the last equality follows from the definition of $\tilde{\phi}_{i}$. 
From these equality, one can easily see that 
\begin{equation}\label{DAHAH}
(D_{\tilde{k}_{i}}^{\tilde{T}_{i}}A_{\tilde{Q}_{i}}h_{i})_{s}=
(A_{Q_{i}'}D_{\tilde{k}_{i}}^{\tilde{T}_{i}}h_{i})_{s}=(A_{Q_{i}'}\tilde{h}_{i})_{s}
=e^{-(A_{Q_{i}'}\tilde{\phi}_{i})(s)}(\bar{\xi}_{i}^{*}\otimes \xi_{i}), 
\end{equation}
where the last equality follows from the definition of $\tilde{h}_{i}$. 
Then, combining \eqref{precom0}, \eqref{precom1}, \eqref{DAHAH} and 
$\mathop{\mathrm{Vol}_{\tilde{k}_{i}g}}(B(r')_{\tilde{p}_{i}})=\tilde{k}_{i}^{n/2}\mathop{\mathrm{Vol}_{g}}(B(r')_{\tilde{p}_{i}})$ implies that 
\begin{equation}\label{precom2}
\begin{aligned}
\bar{\Theta}(h_{i},\tilde{Q}_{i},t_{i})=\frac{(2\sqrt{2})^{n}}{\tilde{k}_{i}^{n/2}\mathop{\mathrm{Vol}_{g}}(B(r')_{\tilde{p}_{i}})}\Theta(A_{Q_{i}'}\tilde{h}_{i},Q_{i}',s). 
\end{aligned}
\end{equation}
Put $\tilde{w}_{i}:=(A_{Q_{i}'}\tilde{\phi}_{i})(s)$. 
Then, by the definition of $\Theta$, we have 
\[\Theta(A_{Q_{i}'}\tilde{h}_{i},Q_{i}',s)
=\int_{X}\frac{1}{(-4\pi s)^{n/2}}\exp\left(\frac{|\mathcal{P}_{x_{i}}(s)|^2}{4s}\right)
\tilde{f}\left(\frac{4|\mathcal{P}_{x_{i}}(s)|}{\lambda(g_{i}) 4r}\right)d\mu((A_{Q_{i}'}\tilde{h}_{i})_{s}), \]
where $\lambda(g_{i})=\tilde{k}_{i}^{1/2}\lambda(g)$, $d\mu((A_{Q_{i}'}\tilde{h}_{i})_{s})=|\zeta((A_{Q_{i}'}\tilde{h}_{i})_{s})|\omega_{i}^{n}/n!$ and 
\[|\mathcal{P}_{x_{i}}(s)|^2=4(g_{i})_{\bar{p}q}(x^{p}-x_{i}^{p})(x^{q}-x_{i}^{q})+\frac{1}{4}(g_{i})^{p\bar{q}}\frac{\partial \tilde{w}_{i}}{\partial x^{p}} \frac{\partial \tilde{w}_{i}}{\partial x^{q}}. \]
Put $X_{i}(R):=B(R)\times B(\tilde{k}_{i}^{1/2})$ 
and $X(R):=B(R)\times B(r')$. 
Then, by the definition of $\tilde{\varphi}_{i}$, 
we see that $\tilde{\varphi}_{i}$ restricted on $X_{i}(R)$ is bijective onto its image 
and the image is included in $U$. 
Then, we have 
\begin{equation}\label{precom3}
\begin{aligned}
&\int_{X}\frac{1}{(-4\pi s)^{n/2}}\exp\left(\frac{|\mathcal{P}_{x_{i}}(s)|^2}{4s}\right)
\tilde{f}\left(\frac{|\mathcal{P}_{x_{i}}(s)|}{\lambda(g_{i}) r}\right)d\mu((A_{Q_{i}'}\tilde{h}_{i})_{s})\\
\geq & \int_{\tilde{\varphi}_{i}(X_{i}(R))}\frac{1}{(-4\pi s)^{n/2}}\exp\left(\frac{|\mathcal{P}_{x_{i}}(s)|^2}{4s}\right)
\tilde{f}\left(\frac{|\mathcal{P}_{x_{i}}(s)|}{\tilde{k}_{i}^{1/2}\lambda(g)r}\right)d\mu((A_{Q_{i}'}\tilde{h}_{i})_{s})\\
=&\int_{X_{i}(R)}\frac{1}{(-4\pi s)^{n/2}}\exp\left(\frac{\tilde{\varphi}_{i}^{*}|\mathcal{P}_{x_{i}}(s)|^2}{4s}\right)
\tilde{f}\left(\frac{\tilde{\varphi}_{i}^{*}|\mathcal{P}_{x_{i}}(s)|}{\tilde{k}_{i}^{1/2}\lambda(g)r}\right)\tilde{\varphi}_{i}^{*}d\mu((A_{Q_{i}'}\tilde{h}_{i})_{s})\\
=&\tilde{k}_{i}^{n/2}\int_{X(R)}\frac{1}{(-4\pi s)^{n/2}}\exp\left(\frac{\tilde{\varphi}_{i}^{*}|\mathcal{P}_{x_{i}}(s)|^2}{4s}\right)
\tilde{f}\left(\frac{\tilde{\varphi}_{i}^{*}|\mathcal{P}_{x_{i}}(s)|}{\tilde{k}_{i}^{1/2}\lambda(g)r}\right)\tilde{\varphi}_{i}^{*}d\mu((A_{Q_{i}'}\tilde{h}_{i})_{s}), 
\end{aligned}
\end{equation}
where the first inequality simply follows form $\tilde{\varphi}_{i}(X_{i}(R))\subset X$ and 
$\lambda(g_{i})=\tilde{k}_{i}^{1/2}\lambda(g)$, the second equality is just the change of variables and 
the last equality follows from that the integrand does not depend on $y$-variable. 
Thus, combining \eqref{precom2} and \eqref{precom3} implies that 
\begin{equation}\label{precom4}
\begin{aligned}
&\bar{\Theta}(h_{i},\tilde{Q}_{i},t_{i})\geq \frac{(2\sqrt{2})^{n}}{\mathop{\mathrm{Vol}_{g}}(B(r')_{\tilde{p}_{i}})}\int_{X(R)}\frac{1}{(-4\pi s)^{n/2}}\exp\left(\frac{\tilde{\varphi}_{i}^{*}|\mathcal{P}_{x_{i}}(s)|^2}{4s}\right)\\
&\hspace{45mm}\times\tilde{f}\left(\frac{\tilde{\varphi}_{i}^{*}|\mathcal{P}_{x_{i}}(s)|}{\tilde{k}_{i}^{1/2}\lambda(g)r}\right)\tilde{\varphi}_{i}^{*}d\mu((A_{Q_{i}'}\tilde{h}_{i})_{s}), 
\end{aligned}
\end{equation}
where $\tilde{k}_{i}^{n/2}$ canceled out. 
On the other hand, by the straightforward computation, 
we can prove that $\tilde{\varphi}_{i}^{*}\tilde{w}_{i}=w_{i}$, 
where we recall that $w_{i}=A_{(O,0)}f_{i}$ and $f_{i}=\tilde{\varphi}_{i}^{*}\tilde{\phi}_{i}$. 
Then, since $w_{i}$ uniformly converges to $w_{\infty}$ on $X(R)\subset X_{\infty}$ and 
we supposed that $\tilde{p}_{i}$ converges to $\tilde{p}_{\infty}$ in the proof of Claim \ref{Claim3.5}, 
letting $i\to \infty$ in \eqref{precom4} 
with $\bar{\Theta}(h_{i},\tilde{Q}_{i},t_{i})\leq 1+\varepsilon_{i}$ implies that 
\begin{equation}\label{precom5}
\begin{aligned}
1\geq \frac{(2\sqrt{2})^{n}}{\mathop{\mathrm{Vol}_{G_{\infty}}}(B(r')_{O})}\int_{X(R)}\frac{1}{(-4\pi s)^{n/2}}\exp\left(\frac{|\mathcal{P}_{O}(H_{\infty}(s))|^2}{4s}\right)d\mu(H_{\infty}(s)), 
\end{aligned}
\end{equation}
where $d\mu(H_{\infty}(s))$ is the induced measure 
defined by $H_{\infty}(s)$ and the limit metric $G_{\infty}$ appeared in Claim \ref{Claim3.5}. 
To deduce \eqref{precom5}, we also used some facts. 
The first couple of facts is that $\tilde{\varphi}_{i}^{*}|\mathcal{P}_{x_{i}}(s)|$ is uniformly bounded 
(because it uniformly converges to $|\mathcal{P}_{O}(H_{\infty}(s))|$), the cut-off function $\tilde{f}(x)$ is identically 1 for $x\leq 1$ and $\tilde{k}_{i}\to \infty$ when $i\to\infty$. 
These imply that the term $\tilde{f}(\ast)$ in \eqref{precom4} uniformly converges to $1$. 
The second couple of facts is that $\tilde{p}_{i}$ converges to $\tilde{p}_{\infty}$ and 
$G_{\infty}$ is actually the constant metric $g(\tilde{p}_{\infty})$ 
as mentioned in the proof of Claim \ref{Claim3.5}. 
These imply that $\mathop{\mathrm{Vol}_{g}}(B(r')_{\tilde{p}_{i}})$ converges to $\mathop{\mathrm{Vol}_{G_{\infty}}}(B(r')_{O})$. 
Since $s\in(-\infty,0)$ and $R>0$ are arbitrary, we proved that 
\begin{equation}\label{precom6}
\begin{aligned}
1\geq \frac{(2\sqrt{2})^{n}}{\mathop{\mathrm{Vol}_{G_{\infty}}}(B(r')_{O})}\int_{X_{\infty}}\frac{1}{(-4\pi s)^{n/2}}\exp\left(\frac{|\mathcal{P}_{O}(H_{\infty}(s))|^2}{4s}\right)d\mu(H_{\infty}(s))
\end{aligned}
\end{equation}
for all $s\in (-\infty,0)$. 
Since $A_{O}H_{\infty}=H_{\infty}$ by the construction of $H_{\infty}$, 
\eqref{precom6} means that 
\[\bar{\Theta}_{\infty}(H_{\infty},O,s)\leq 1\]
for all $s\in (-\infty,0)$, where $\bar{\Theta}_{\infty}$ is defined in \eqref{densitynoncomp}. 
Then, by Corollary \ref{impcor}, we see that $H_{\infty}$ satisfies 
\[\mathcal{H}=\frac{1}{2s}\mathcal{P}^{\bot}\]
for all $s\in (-\infty,0)$. 
Then, by Theorem \ref{tosayquadra}, 
there exist $b\in \mathbb{R}$ and a symmetric matrix $A=(a_{ij})\in M(n,\mathbb{R})$ such that 
$w_{\infty}=-\log H_{\infty}=b+a_{ij}x^{i}x^{j}$. 
Since $w_{\infty}(O,0)=0$ by the normalization, $b=0$. 
Then, the proof is complete. 
\hspace{\fill}\qedsymbol\vspace{3mm}

By Claim \ref{Claim7}, we see that 
\[
 \begin{aligned}
& \sup_{Q'\in K\times[a,b]}\left(|\partial_{s}w_{\infty}|+|\partial_{s}\nabla w_{\infty}|+|\nabla^3 w_{\infty}|\right)(Q')\\
&+\sup_{\substack{Q_{1},Q_{2}\in K\times[a,b]\\ Q_{1}\neq Q_{2}}}\frac{|\partial_{s}\nabla w_{\infty}(Q_{1})-\partial_{s}\nabla w_{\infty}(Q_{2})|}{\mathrm{dist}_{G_{\infty}}(Q_{1},Q_{2})^{\alpha}}\\
&+\sup_{\substack{Q_{1},Q_{2}\in K\times[a,b]\\ Q_{1}\neq Q_{2}}}\frac{|\nabla^{3} w_{\infty}(Q_{1})-\nabla^{3} w_{\infty}(Q_{2})|}{\mathrm{dist}_{G_{\infty}}(Q_{1},Q_{2})^{\alpha}}=0 
\end{aligned}
\]
for any compact set $K\times[a,b]$ in $X_{\infty}\times\mathbb{R}$. 
However, this contradict to the uniform lower bound \eqref{keytocontra1}. 
Then, the proof is complete. 
\end{proof}

As a corollary of Theorem \ref{epregthm}, 
we give a sufficient condition so that 
a line bundle mean curvature flow defined on a finite time interval $[0,T)$ 
can be extended beyond the time $T$. 
We denote the open right lower triangle of $(0,T)\times (0,T)$ by 
\[D:=\{\,(T',t)\in\mathbb{R}^{2}\mid 0<T'<T,\, 0<t<T'\,\}. \]
Fix a K\"ahler manifold $(X,g)$, a bounded open set $U'\subset X$, $\alpha\in(0,1)$ and $A>0$. 
Assume that $(X,g)$ is semi-flat on $U'$ with respect to $\varphi:B(4r)\times B(r')\to U'$. 
Let $\varepsilon, C>0$ be constants appeared in Theorem \ref{epregthm}. 

\begin{corollary}
Suppose $L\to X$ is a holomorphic line bundle, 
$h=\{\,h_{t}\,\}_{t\in [0,T)}$ is a line bundle mean curvature flow of $L$ with $T<\infty$ 
and $e\in\Gamma(U',L)$ is a nonvanishing holomorphic section 
so that $h_{t}$ is graphical on $U'$ for all $t\in[0,T)$ 
with respect to $e\in\Gamma(U',L)$. 
Put $V:=\varphi(B(r)\times B(r'))\times [0,T)$. 
Further assume that $\sup_{V}|F(h(t))|\leq A$ and 
\[\limsup_{(q,T',t)\to (p,T,T)}\bar{\Theta}(h,(q,T'),t)< 1+\varepsilon, \]
where $(q,T',t)\in X\times D$, then $h$ can be extended beyond $T$ around $p$. 
\end{corollary}

\begin{proof}
By the assumption, we know that there is an open neighborhood $U''$ of $p$ and $a\in(0,T)$ such that
\[\bar{\Theta}(h,(q,T'),t)\leq 1+ \varepsilon\]
for all $q\in U'', T'\in (a, T)$ and $t\in (a, T')$. 
Making $U''$ smaller if necessary so that $a<T'-(\mathrm{dist}_{g}(Q,V))$ for all $Q=(q,T')\in U''\times (b,T)$ for some $b\in (a,T)$, 
we can apply Theorem \ref{epregthm} (with truncating the time interval to $[b,T)$). 
Then, we know that
\[K_{3,\alpha;V}(g,\phi)\leq C, \]
where $\phi:=-\log h(\bar{e},e)$. 
Then, by the similar argument as in the proofs of Claim \ref{Claim4} and Claim \ref{Claim5}, 
one can see that all derivatives of $\phi$ is bounded around $p$. 
Thus, the flow can be extended beyond $T$ around $p$.
\end{proof}


\end{document}